\documentclass[12pt]{article}
\pdfoutput=1

\usepackage{enumerate}
\usepackage[OT1]{fontenc}
\usepackage{multirow}
\usepackage{amsmath,amssymb,mathrsfs}
\usepackage{natbib}
\usepackage[usenames]{color}
\usepackage{dsfont,makecell}
\usepackage{smile}
\usepackage{multirow}

\newcommand\blfootnote[1]{%
  \begingroup
  \renewcommand\thefootnote{}\footnote{#1}%
  \addtocounter{footnote}{-1}%
  \endgroup
}

\usepackage[colorlinks,
            linkcolor=red,
            anchorcolor=blue,
            citecolor=blue
            ]{hyperref}

\usepackage{mathtools}

\def\tr{\mathop{\text{tr}}\kern.2ex}

\def\rmH{\mathrm{H}}
\def\rmA{\mathrm{A}}

\def\rmI{\mathrm{I}}
\def\rmv{\mathrm{v}}

\def\cL{\mathcal{L}}

\def\sign{\mathop{\text{sign}}}
\def\supp{\mathop{\text{supp}}}

\long\def\comment#1{}

\def\tr{\mathop{\text{Tr}}}

\def\cS{{\mathcal{S}}}
\def\dr{\displaystyle \rm}

\newcommand{\bel}{\begin{eqnarray}\label}
\newcommand{\eel}{\end{eqnarray}}
\newcommand{\bes}{\begin{eqnarray*}}
\newcommand{\ees}{\end{eqnarray*}}

\let\emptyset\varnothing
\let\hat\widehat
\let\tilde\widetilde

\def\EE{{\mathbb E}}

\def\supp{\mathop{\text{supp}\kern.2ex}}

\def\tr{{\rm{Tr}}}
\def\sign{\mathrm{sign}}
\def\supp{\mathop{\text{supp}}}

\def\tr{\mathrm{Tr}}

\def \mA {\mathscr{A}}
\def \mG {\mathscr{G}}

\usepackage{mathrsfs}

\usepackage{fullpage}

\usepackage[    protrusion=true,
            expansion=true,
            final,
            babel
                ]{microtype}
\def\##1\#{\begin{align}#1\end{align}}
\def\$#1\${\begin{align*}#1\end{align*}}
\usepackage{enumitem}

\usepackage{undertilde}
\theoremstyle{plain}

\theoremstyle{mytheoremstyle}

\begin{document}

\title{\huge High-Temperature Structure Detection in Ferromagnets}

\author{
Yuan Cao\thanks{Department of Computer Science
University of California, Los Angeles, CA 90095, USA; e-mail: {\tt yuancao@cs.ucla.edu}}~~~~
Matey Neykov\thanks{Department of Statistics \& Data Science, Carnegie Mellon University,
Pittsburgh, PA 15213, USA; e-mail: {\tt mneykov@stat.cmu.edu}}~~~~Han Liu\thanks{Department of Electrical Engineering and Computer Science, Northwestern University, Evanston, IL 60208, USA; e-mail: {\tt 	
hanliu.cmu@gmail.com}}
}

\maketitle

\begin{abstract}
{This paper studies structure detection problems in high temperature ferromagnetic  (positive interaction only) Ising models. The goal is to distinguish whether the underlying graph is empty, i.e., the model consists of independent Rademacher variables, versus the alternative that the underlying graph contains a subgraph of a certain structure. We give matching upper and lower minimax bounds under which testing this problem is possible/impossible respectively. Our results reveal that a key quantity called graph arboricity drives the testability of the problem. On the computational front, under a conjecture of the computational hardness of sparse principal component analysis, we prove that, unless the signal is strong enough, there are no polynomial time tests which are capable of testing this problem. In order to prove this result we exhibit a way to give sharp inequalities for the even moments of sums of i.i.d. Rademacher random variables which may be of independent interest. \blfootnote{Published in \textit{Information and Inference: A Journal of the IMA.}}}
% {Graph structure detection, minimax testing, ferromagnetic Ising model, Total Variation distance}
%%%% If classification number provided then
% \\
% 2000 Math Subject Classification: 62F03, 62H15
\end{abstract}

\section{Introduction}\label{section:introduction}

%\fbox{The intro needs a lot of work! Let's do it after we finish the rest first! }

Graphical models are a powerful tool in high dimensional statistical inference. The graph structure of a graphical model gives a simple way to visualize the dependency among the variables in multivariate random vectors. The analysis of graph structures plays a fundamental role in a wide variety of applications, including information retrieval, bioinformatics, image processing and social networks \citep{besag1993statistical, durbin1998biological,wasserman1994social, grabowski2006ising}. Motivated by these applications, theoretical results on graph estimation \citep{meinshausen2006high,liu2009nonparanormal,bento2009graphical,ravikumar2011high,cai2011constrained}, single edge inference \citep{jankova2015confidence,ren2015asymptotic,neykov2015unified,gu2015local} and combinatorial inference \citep{neykov2016combinatorial, neykov2017property} have been studied in the literature. 

In this paper we are concerned with the distinct problem of \textit{structure detection}. In structure detection problems one is interested in testing whether the underlying graph is empty, (i.e., the random variables are independent) versus the alternative that the graph contains a subgraph of a certain structure.  A variety of detection problems have been previously considered in the literature \citep[see for example][]{addario2010combinatorial,arias2012detection,arias2015detectinga,arias2015detectingb}. These works mainly focus on covariance or precision matrix detection problems and establish minimax lower and upper bounds. 

While covariance and precision matrix detection problems are inherently related to the Gaussian graphical model, in this paper we focus on  detection problems under the \textit{zero-field ferromagnetic Ising model}. The Ising model is a probability model for binary data originally developed in statistical mechanics \citep{ising1925beitrag} and has wide range of modern applications including image processing \citep{geman1984stochastic}, social networks and bioinformatics \citep{ahmed2009recovering}. Below we formally introduce the model and problems of interest.\\

\noindent\textbf{Zero-field ferromagnetic Ising model}.
Under a zero-field Ising model, the binary vector $\bX\in \{\pm 1\}^d$ follows a distribution with probability mass function given by
\begin{align*}
\PP_{\Theta}(\bX) = \frac{1}{Z_{\Theta}}\exp\Bigg( \sum_{i,j = 1}^d \theta_{ij}X_i X_j \Bigg),
\end{align*}
where $\Theta = (\theta_{ij})_{d\times d}$ is a symmetric interaction matrix with zero diagonal entries and $Z_{\Theta}$ is the partition function defined as
\begin{align*}
Z_{\Theta} = \sum_{\bX\in \{\pm 1\}^d} \exp\Bigg( \sum_{i,j=1}^d \theta_{ij}X_i X_j  \Bigg).
\end{align*}
The non-zero elements of the symmetric matrix $\Theta$ specify a graph $G(\Theta) = G = (\overline{V},E)$ with vertex set $\overline{V}=\{1,\ldots,d\}$ and edge set $E=\{(i,j):\theta_{ij}\neq 0\}$. We will refer to the graph $G(\Theta)$ as $G$ whenever it is clear what the underlying matrix $\Theta$ is. It is not hard to check that by the definition of $G$, the vector $\bX$ is Markov with respect to $G$, that is, each two elements $X_i$ and $X_j$ are independent given the remaining values of $\bX_{-(i,j)}$ if and only if $(i,j) \not \in E$. 

Here, the term \textit{zero-field} specifies that there is no external magnetic field affecting the system, meaning that the energy function $\sum_{i,j=1}^d \theta_{ij}X_i X_j$ consists purely the terms of degree $2$ (i.e., there are no main effects). In this paper, we further focus on zero-field \textit{ferromagnetic} models, where we also assume that $\theta_{ij}\geq 0$, $i,j \in \{1,\ldots,d\}$. In addition, our analysis is under the high-temperature setting, where the magnitudes of $\theta_{ij}$'s are under a certain level. More specifically, throughout this paper we assume that $\|\Theta \|_F \leq \frac{1}{2}$, where $\|\Theta \|_F = \bigr[\sum_{i,j = 1}^d \theta^2_{ij}\bigr]^{1/2}$ is the Frobenius norm of $\Theta$. \\

\noindent\textbf{Structure detection problems}.
%A zero-field ferromagnetic Ising model defines a graph $G=(\overline{V},E)$, where $\overline{V}=\{1,\ldots,d\}$ and $E = \{(i,j): \theta_{ij}\neq 0 \}$. We denote by $G(\Theta)$ the graph defined by the Ising model with parameter matrix $\Theta$. 
As described in the previous paragraph, a zero-field ferromagnetic Ising model specifies a graph $G=(\overline{V},E)$.
In a structure detection problem, we are interested in testing whether the underlying graph $G$ is an empty graph versus the alternative that $G$ belongs to a set of graphs with a certain structure. Specifically, let $G_\varnothing = (\overline{V},\emptyset)$ be the empty graph, and let $\cG_1$ be a class of graphs not containing $G_\varnothing$. The following hypothesis testing problem is an example of a detection problem. Given a sample of $n$ independent observations $\bX_1,\ldots,\bX_n\in \RR^d$ from a zero-field ferromagnetic Ising model we aim to test
\begin{align}\label{eq:testdef}
\rmH_0: G = G_\varnothing ~~\text{ versus }~~\rmH_1:G\in \cG_1.
\end{align}

The term ``detection'' here is used in the sense that if one rejects the null hypothesis, the presence of a non-null graph has been detected. %The hardness of this detection problem is clearly determined via the signal strength $\theta$ and the graph set $\cG_1$. 
In \eqref{eq:testdef} the graph class $\cG_1$ can be arbitrary, which makes the hypothesis testing problem \eqref{eq:testdef} a very general problem. We now give a specific instance of this problem which is of  particular importance. Let $G_*$ be a fixed graph with $s=o(\sqrt{d})$\footnote{For two positive sequences $a_n$ and $b_n$ we write $a_n = o(b_n)$ if $\lim_{n\rightarrow \infty} a_n/b_n = 0.$} 
non-isolated vertices which represents some specific graph structure. The structure detection problem that considers all possible ``positions'' of $G_*$ is of the following form:
\begin{align}\label{eq:detectiontestdef}
\rmH_0: G = G_\varnothing  \text{ versus }\rmH_1:G\in \cG_1(G_*),
\end{align}
where $\cG_1(G_*)$ is the class of all graphs that contain a subgraph isomorphic to $G_*$.

%The term ``structure detection'' is used in the sense that if (\ref{eq:detectiontestdef}) is rejected then the structure of $G$ is certified to contain $G_*$.

While problems \eqref{eq:testdef} and \eqref{eq:detectiontestdef} give a good intuition what a detection problem is, in order to facilitate testing we need to impose certain assumptions on the matrix $\Theta$, as otherwise even with graphs vastly different from the empty graph there might not be enough ``separation'' between the null and the alternative hypothesis. Since the underlying graph $G$ is specified by the matrix $\Theta$, we can reformulate problems \eqref{eq:testdef} and \eqref{eq:detectiontestdef} into testing problems on $\Theta$. Given a class of graphs $\cG_1$, we define the corresponding parameter space with minimum signal strength $\theta>0$ as
\begin{align}\label{def:cS:cG1}
\cS(\cG_1,\theta) = \Big\{ \Theta= (\theta_{ij})_{d\times d}: \Theta = \Theta^T, G(\Theta)\in \cG_1, 
\|\Theta \|_F \leq 1/2, \min_{(i,j)\in E[G(\Theta)]} \theta_{ij}\geq \theta \Big\}.
\end{align}
We now reformulate the hypothesis testing problems \eqref{eq:testdef} and \eqref{eq:detectiontestdef} as follows:
\begin{align}
&\rmH_0: \Theta = \mathbf{0} ~~\text{ versus }~~\rmH_1:\Theta\in \cS(\cG_1,\theta),\label{eq:testdefS1}\\
&\rmH_0: \Theta = \mathbf{0} ~~\text{ versus }~~\rmH_1:\Theta\in \cS[\cG_1(G_*),\theta].\label{eq:testdefS2}
\end{align}
The results of our paper cover the following examples.\\

\noindent \textbf{Empty graph versus non-empty graph.} We consider testing whether the underlying graph of the Ising model is empty or not. Clearly, since our null hypothesis is that the graph is empty, this is a detection problem. We have $\cG_1 = \{G: E(G)\neq \emptyset\}$.\\
%{\color{red} This is just testing a clique with $s = 1$?}

\noindent \textbf{Clique detection.} A clique is a set of vertices such that every two distinct vertices are adjacent. We consider detecting graphs that contain a clique of size $s$. We have $\cG_1 = \{ G=(\overline{V},E): \exists V\subseteq \overline{V} \text{ such that } |V| = s\text{ and } (i,j)\in E \text{ for all }i,j\in V\}$. This is a more general version of the previous example, since one can think of a non-empty graph as a graph containing a clique of size $s = 2$. \\

\noindent \textbf{Star detection.} A star is a tree in which all leaves are connected to the same node. We consider detecting graphs that contain an $s-1$ star. In this example, we have $\cG_1 = \{ G=(\overline{V},E):$ there exist distinct  $i_0,i_1,\ldots, i_{s-1}\in \overline{V} \text{ such that } (i_0,i_1),(i_0,i_2),\ldots,(i_0,i_{s-1})\in E \}$.\\

\noindent \textbf{Community structure detection.} In this example we consider a class of graphs with more complex structure. Let $k$ and $l$ be positive integers. A community $\cC$ is represented by a $k$-clique, which means that every two members in the same community are connected. For a community $\cC$, we select one fixed representative vertex and denote it as $v(\cC)$. We consider the class of graphs $\cG_1$ that contains graphs with at least $l$ disjoint communities, such that for every two different communities $\cC$ and $\cC'$, there exists an edge connecting $v(\cC)$ and $v(\cC')$. In this example we set $s =k l$.\\

All of the above examples are of the type \eqref{eq:testdefS2}. We show examples of these detection problems in Figure~\ref{fig:prop-example}. In the following section we outline the main contributions of our work. %{\color{red} It is worth mentioning that in the example of community structure detection problem, the graphs we detect is very similar to the caveman graphs \citep{watts1999small} but more general in the sense that we allow the communities to be connected by multiple edges. The detection of caveman graphs yields similar properties to community structure detection with $l = 1$.}
\begin{figure}[t]
	\begin{center}
		\begin{tabular}{cccc}
			\includegraphics[width=.15\textwidth,angle=0]{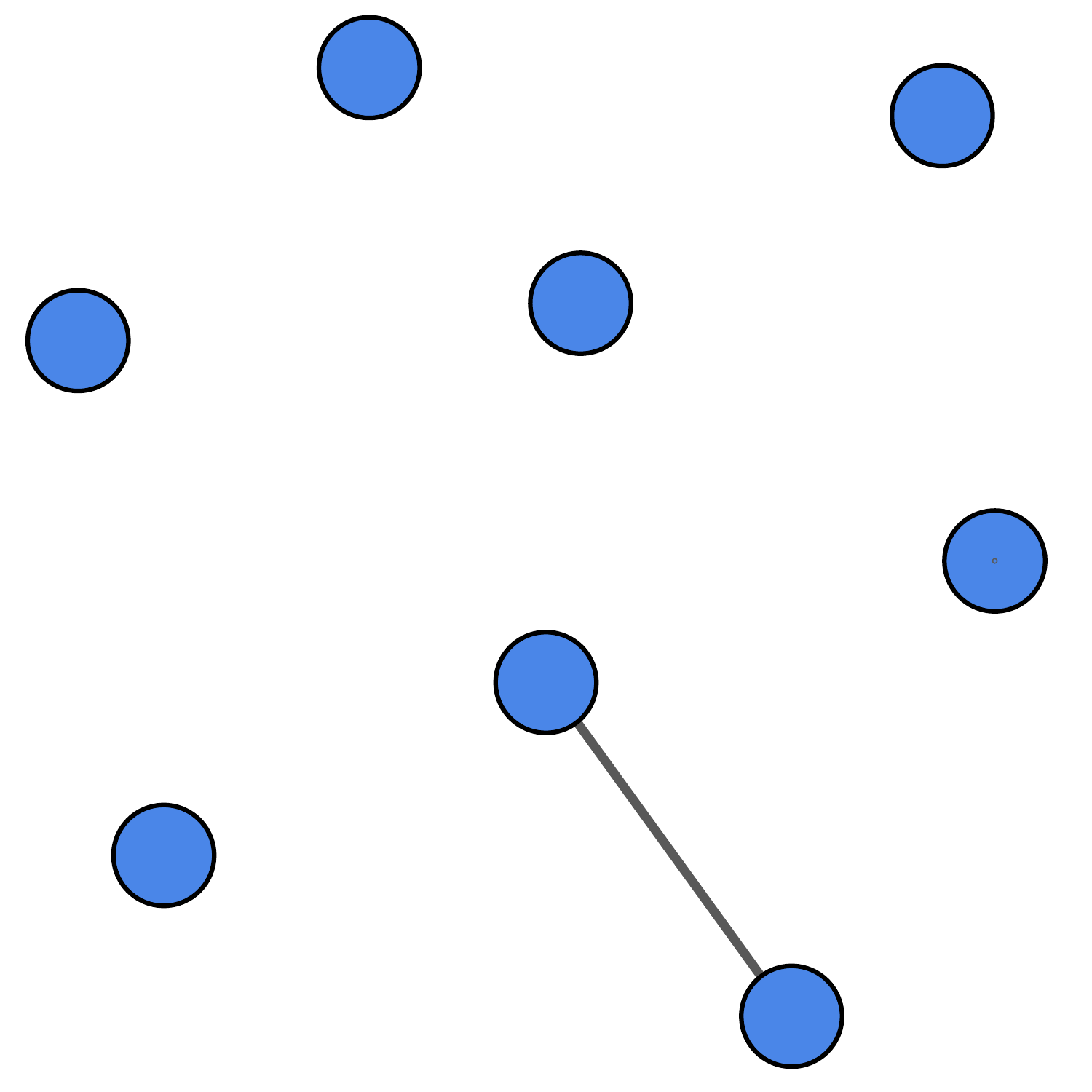}
			&  
			\includegraphics[width=.15\textwidth,angle=0]{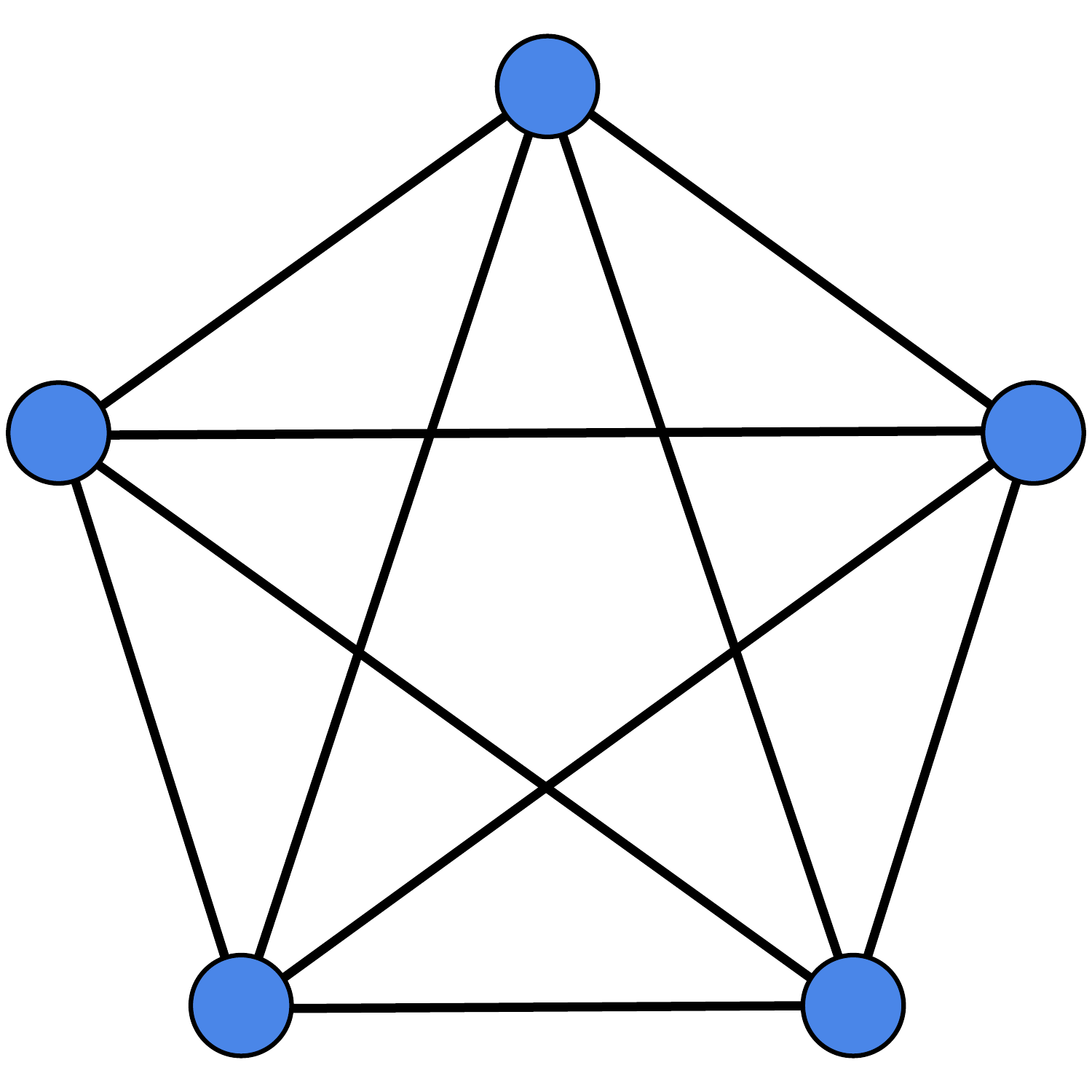}
			& \includegraphics[width=.15\textwidth,angle=0]{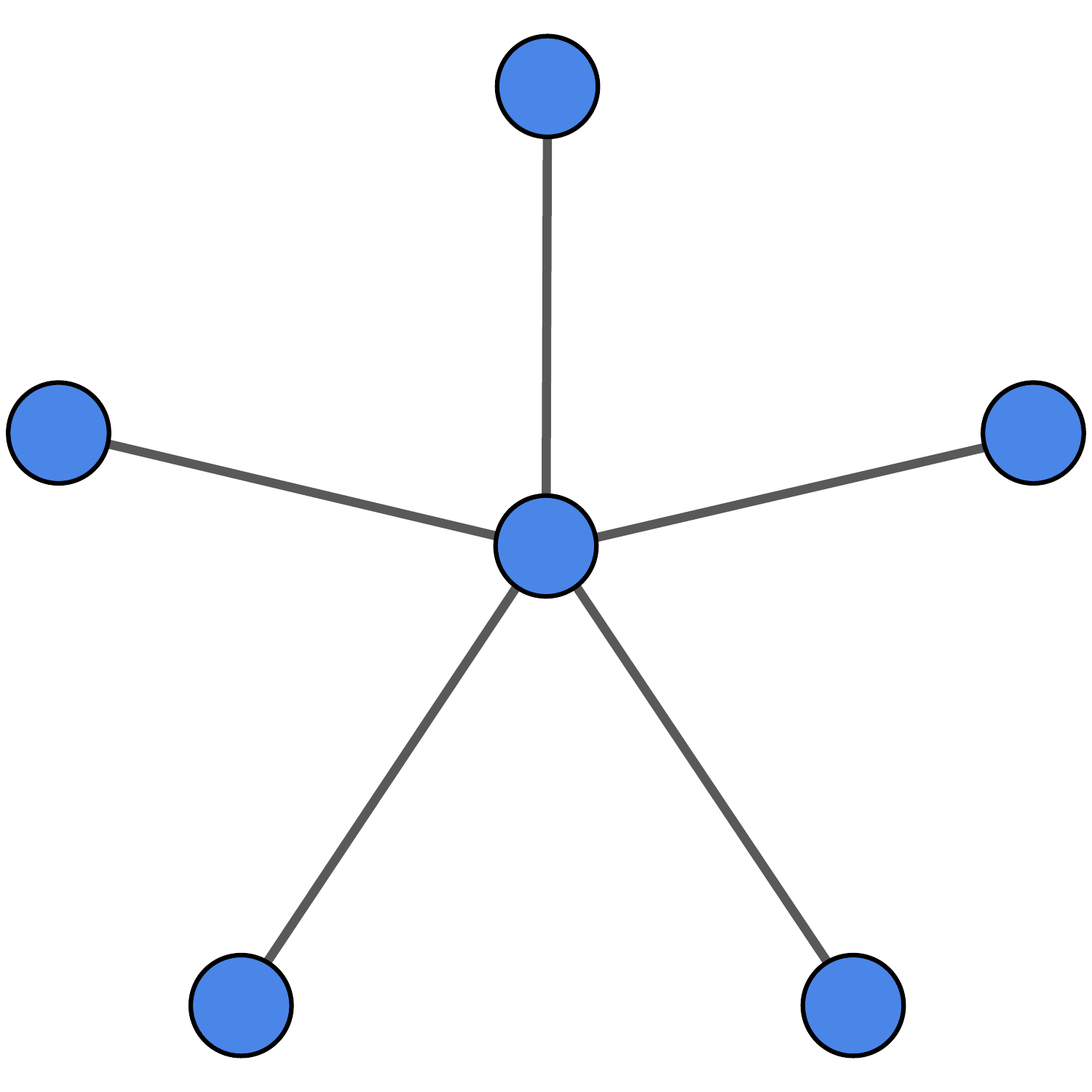}
			& \includegraphics[width=.15\textwidth,angle=0]{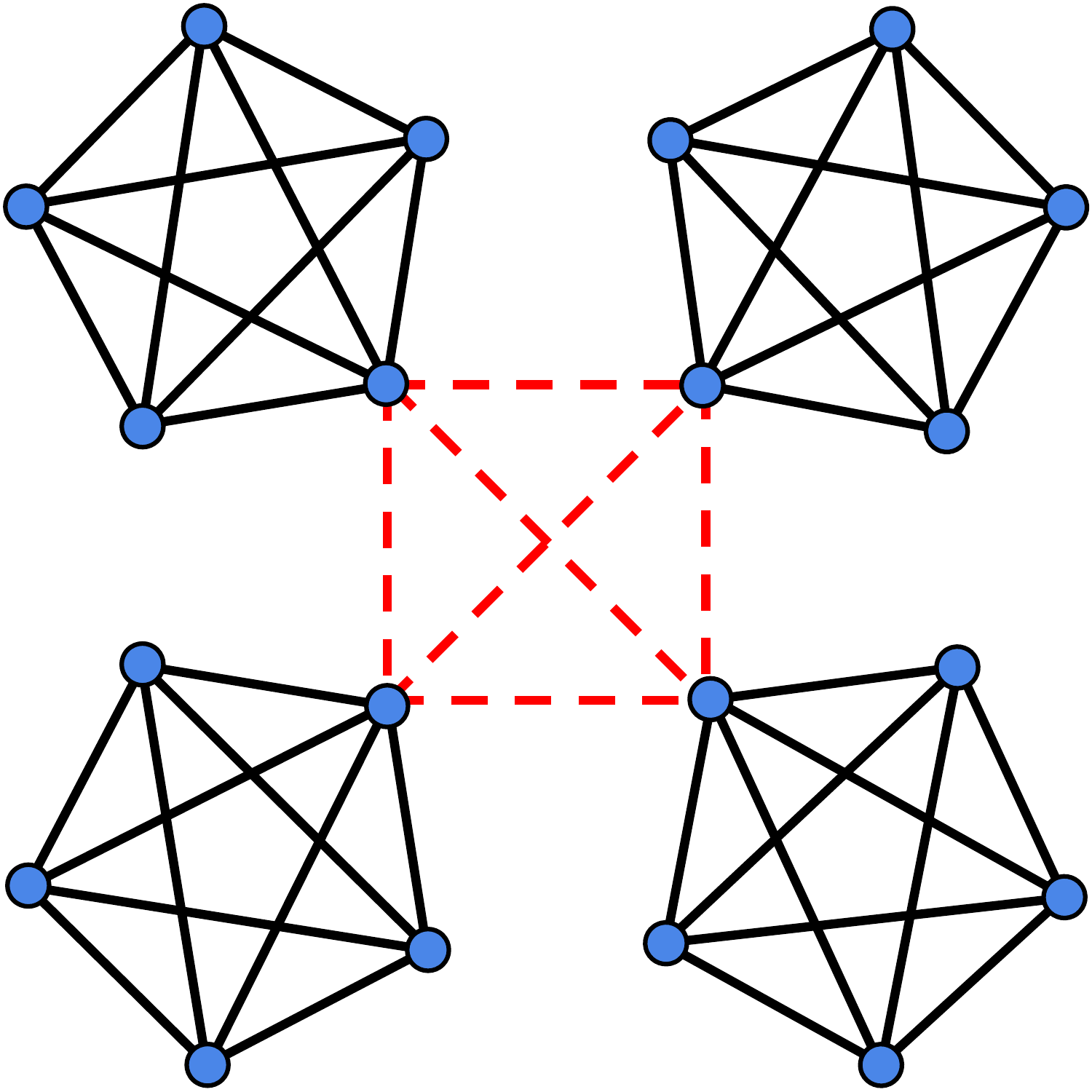}\\
			(a) non-empty graph & (b) $5$-clique & (c) $5$-star & (d) community
			\end{tabular}
	\end{center}
	\vskip-10pt
	\caption{Illustration of the examples considered in this paper. (a) shows a single-edge graph; (b) is a $5$-clique; (c) is a $5$-star; (d) is an example of a graph that has community structure with $k=5$ and $l=4$. We can write the detection problems as  \eqref{eq:testdefS2} by defining the corresponding shown graphs as $G_*$.}
	\label{fig:prop-example}
	\vspace{-10pt}
\end{figure}

\subsection{Main Contributions}
There are three major contributions of this paper.

First, we develop a novel technique to derive minimax lower bounds of structure detection problems in Ising models. Our proof technique relates the Ising model probability mass function and the $\chi^2$-divergence between two distributions to the number of certain Eulerian subgraphs of the underlying graph. With this technique, we are able to obtain a general information-theoretic lower bound for arbitrary alternative hypothesis, which can be immediately applied to examples including any of the four examples described in the previous section. 

Second, we propose a linear scan test on the sample covariance matrix that matches our minimax lower bound for arbitrary structure detection problems in certain regimes. Along with our general minimax lower bound result, this procedure reveals the fact that a quantity called \textit{arboricity}, (i.e., a certain maximum edge to vertex ratio of graphs in the alternative hypothesis) essentially determines the information-theoretic limit of the testing problem. This matches the intuition that in order to distinguish a graph with small signal strength from the empty graph, one need to examine the densest part of the graph. Furthermore, the denser the graph is, the easier it is to detect it, where the precise measurement of graph density turns out to be graph arboricity.

In addition, we also study the computational lower bound of structure detection problems. Based on a conjecture on the computational hardness of sparse Principal Component Analysis (PCA), which has been studied by recent works \citep{berthet2013optimal,berthet2013complexity,gao2014sparse, brennan2019optimal, brennan2018reducibility}, we prove that no polynomial time test can detect structures successfully unless there is a sufficiently large signal strength. In order to prove this result we exhibit a way to give sharp inequalities for the even moments of sums of i.i.d. Rademacher random variables which may be of independent interest. %{\color{red} By this result, it is shown that single edge test and tests that only utilize a constant number of edges are optimal polynomial time linear scan tests, since they match the computational minimax lower bound.}
Furthermore, in addition to this result, we also derive another computational lower bound result under the oracle computational model studied by \cite{feldman2015statistical,feldman2018complexity,wang2015sharp}.

\subsection{Related Work}

Plenty of work has been done on graph estimation (also known as graph selection) in Ising models. \cite{santhanam2012information} gave the first information-theoretic lower bounds of graph selection problems for bounded edge cardinality and bounded vertex degree models. Later, \cite{tandon2014information} proposed a general framework for obtaining information-theoretic lower bounds for graph selection in ferromagnetic Ising models, and showed that the lower bound is specified by certain structural conditions. On the other hand, \cite{ravikumar2010high} proposed an algorithm for structure learning based on $l_1$-regularized logistic regression that works in the high temperature regime \citep{bento2009graphical}. \cite{bresler2015efficiently} gave a polynomial time algorithm that works for both low and high temperature regimes. In addition see \cite{vuffray2016interaction, lokhov2018optimal} for best polynomial time structure learning algorithms. Compared to graph estimation, structure detection is a statistically easier problem. As a consequence, the limitations on signal strength that we exhibit in this paper are weaker than the corresponding requirements used in the graph estimation literature. %All these works focus on graph estimation, which is different from the structure detection problem we consider in this paper. 

Structure detection problems have been studied in \cite{addario2010combinatorial,arias2012detection,arias2015detectinga,arias2015detectingb}. However, all these works focus on Gaussian random vectors. Specifically, \cite{addario2010combinatorial} study testing the existence of specific subsets of components in a Gaussian vector whose means are non-zero based on a single observation. \cite{arias2012detection} consider the correlation graph of a Gaussian random vector and establish upper and lower bounds for detecting certain classes of fully connected cliques based on one sample. In a follow up work, \cite{arias2015detectinga} generalize the result to multiple i.i.d. samples. \cite{arias2015detectingb} give another related result on detecting a region of a Gaussian Markov random field against a background of white noise. 
The major difference between these existing works and our work is that we focus on detection in the Ising model, and our results not only work for cliques, but also for general graph structures. Recently, \citep{neykov2016combinatorial, lu2017adaptive, neykov2017property} proposed a novel problem where one considers testing whether the underlying graph obeys certain combinatorial properties. We stress that while related to structure detection, these problems are fundamentally different as structure detection is a statistically simpler task. It is not surprising therefore that the algorithms we develop are very different from those in the aforementioned works, and the proofs of our lower bounds use different techniques. 
 
% This is related to our work since the detecting a structure with certain combinatorial properties is covered by our result. However, since we focus on Ising model, the details in the proof is drastically different as we use the novel technique that utilizes properties of Eulerian graphs.

Our result on computational lower bound follows the recent line of work on computational barriers for statistical models \citep{berthet2013optimal,berthet2013complexity,ma2015computational,gao2014sparse,brennan2018reducibility, brennan2019optimal} based on the planted clique conjecture. \cite{berthet2013optimal} focus on the testing method based on Minimum Dual Perturbation (MDP) and semidefinite programming (SDP) and prove that such polynomial time testing methods cannot attain the minimax optimal rate for sparse PCA. \cite{berthet2013complexity} prove the computational lower bound on a generalized sparse PCA problem which includes all multivariate distributions with certain tail probability assumptions on the quadratic form. \cite{ma2015computational} consider the Gaussian submatrix detection problem and propose a framework to analyze computational limits of continuous random variables via constructing a sequence of asymptotically
equivalent discretized models. Inspired by the results in \cite{ma2015computational}, \cite{gao2014sparse} consider the computational lower bound for Gaussian sparse Canonical Correlation Analysis (CCA) as well as sparse PCA problems. Our computational lower bound result is based on the previous studies on the sparse PCA problem. We summarize these results and directly base our result for Ising models on a sparse PCA conjecture. %By doing this, we are able to use a novel proof technique that utilizes the high-dimensional central limit theorems of \cite{chernozhukov2014central}.
%% COMMENT ON BRENNAN AND BRESLER WORKS!

Other related works on Ising models include the following.  \cite{berthet2016exact} study the Ising block model by providing efficient methods for block structure recovery as well as information-theoretic lower bounds. \cite{mukherjee2018global} study the upper and lower bounds for detection of a sparse external magnetic field in Ising models. \cite{daskalakis2018testing} consider goodness-of-fit and independence testing in Ising models using pairwise correlations. \cite{gheissari2017concentration} establish concentration inequalities for polynomials of a random vector in contracting Ising models.
\subsection{Notation}
We use the following notations in our paper. For a vector $\rmv=(\rmv_1,\ldots,\rmv_d)^T\in \RR^d$ and a number $1\leq p < \infty$, let $\|\rmv\|_p = (\sum_{i=1}^d |\rmv_i|^p)^{1/p}$. We also define $\|\rmv\|_\infty = \max_i|\rmv_i|$. For a matrix $\rmA$, we denote $\|\rmA\|_{\max} = \max_{j,k} |\rmA_{jk}|$, $\| \rmA \|_F = (\sum_{i,j=1}^d \rmA_{ij}^2)^{1/2}$ and $\|\rmA\|_p = \max_{\|\rmv\|_p = 1} \|\rmA\rmv\|_p$ for $p\geq 1$. 

We also use the standard asymptotic notations $O(\cdot)$ and $o(\cdot)$. Let $a_n$ and $b_n$ be two sequences and assume that $b_n$ is non-zero for large enough $n$. We write $a_n = O(b_n)$ if $\limsup_{n\rightarrow \infty} |a_n/b_n| < \infty$ and $a_n = o(b_n)$ if $\lim_{n\rightarrow \infty} a_n/b_n =0$.

Let $\overline{V} = \{ 1,\ldots, d \}$ be the complete vertex set. In this paper we consider graphs with $d$ vertices over the vertex set $\overline{V}$. For a graph $G$, let $E(G) = \{ (i,j): G\text{ has  an edge connecting vertex }\allowbreak i \text{ and }j \}$, where $(i,j)=(j,i)$ are undirected pairs. Moreover, we denote by $V(G) = \{ i\in \overline{V}: G \text{ has an edge connecting vertex } i \}$ the set of non-isolated vertices of $G$. 

%\fbox{Do we use asymptotic notations $O$ and $o$? If so they need to be defined here.}\\
%\fbox{Define the Frobenius norm and other matrix norms?}

\subsection{Organization of the Paper}
Our paper is organized as follows. In Section~\ref{section:lowerbound}, we present our main information-theoretic lower bound result as well as its applications to various detection problems. In Section~\ref{section:upperbounds} we develop a general procedure to construct optimal linear scan tests on the sample covariance matrix. In Section~\ref{section:complowerbound} we examine the computational limit of the polynomial time tests by comparing the the Ising and sparse PCA models. Sections \ref{section:five} and \ref{section:six} contain the proofs of the main results of Sections \ref{section:lowerbound} and \ref{section:upperbounds} respectively. The remaining detailed proofs are all placed in Section~\ref{proofs:mainsection}. In Section~\ref{section:oralowerbound} we provide an additional proof of a computational lower bound under the oracle computational model.

\section{Lower Bounds}\label{section:lowerbound}
The minimax risk of detection problem \eqref{eq:testdefS1} is defined as
\begin{align}\label{minimax:risk}
\gamma[\cS(\cG_1, \theta)] := \inf_{\psi}\bigg[\PP_{0,n}(\psi = 1) + \max_{\Theta\in\cS(\cG_1,\theta)}\PP_{\Theta,n}(\psi = 0)\bigg],
\end{align}
where $\PP_{0,n}$ and $\PP_{\Theta,n}$ are the joint probability measures of $n$ i.i.d. samples under null and alternative hypotheses respectively. The infimum in \eqref{minimax:risk} is taken over all measurable test functions $\psi: \{\bX_1, \ldots, \bX_n\} \mapsto \{0,1\}$. 
\iffalse
\begin{align*}
\gamma(G_\varnothing,\cG_1) := \inf_{\psi}\bigg[\PP_{\emptyset,n}(\psi = 1) + \max_{G\in\cG_1}\PP_{G,n}(\psi = 0)\bigg],
\end{align*}
\fi
If $\liminf_{n\rightarrow \infty} \gamma[\cS(\cG_1, \theta)] = 1 $, we say that any test is asymptotically powerless. 

In this section, we derive necessary conditions on the signal strength $\theta$ required for detection problems to admit tests which are not asymptotically powerless. Our results will show that the difficulty of testing an empty graph against $\cG_1$ is determined by a quantity called arboricity, which was originally introduced in graph theory by \cite{nash1961edge} to quantify the minimum number of disjoint forests into which the edges of a given graph can be partitioned.

For a graph $G\in \cG_1$ and a vertex set $V\subseteq \overline{V}$, let $G_V$ be the graph obtained by restricting $G$ on the vertices in $V$ (i.e., removing all edges which are connected to vertices $\overline{V}\setminus V$). The arboricity of $G$ is defined as follows:
\begin{align}\label{arboricity:def}
\cR(G) := \bigg\lceil \max_{V\subseteq \overline{V}} \frac{|E(G_V)|}{|V| - 1} ~\bigg\rceil,
\end{align}
%\fbox{$\cR(G)$ should be redefined as \textbf{arboricity} of the graph. See wikipedia and find references.}
where $\lceil \cdot \rceil$ is the ceiling function, and $0/0$ is understood as $0$. 
The arboricity of a graph measures how dense the graph is. For an illustration of arboricity see Figure \ref{fig:example_arboricity}. Let $G_\emptyset = (\overline V, \emptyset)$ denote the empty graph. By definition $\cR(G_\emptyset) = 0$. 
\begin{figure}[t]
	\begin{center}
		\includegraphics[width=.6\textwidth,angle=0]{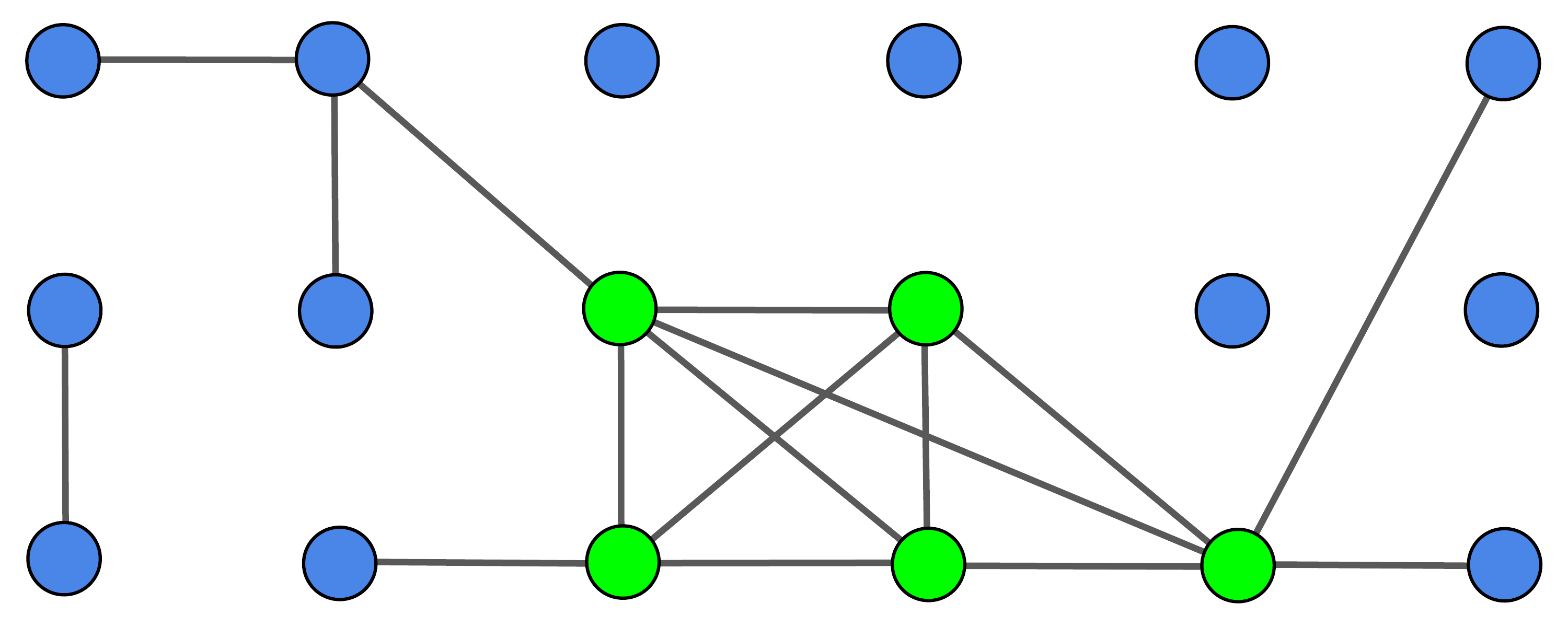}
	\end{center}
	\vskip-10pt
	\caption{Illustration of arboricity. Here the nodes and black lines represent the vertices and edges of graph $G$ respectively. The vertex set $V$ that maximizes $|E(G_V)|/ (|V|-1)$ is denoted by green nodes, which also gives the densest subgraph of $G$. We have $\cR(G) = 3$.}
	\label{fig:example_arboricity}
	\vspace{-10pt}
\end{figure}
For a given graph $G$ the larger $\cR(G)$ is, the more different $G_\emptyset$ and $G$ are. We further define
\begin{align*}
\cR := \min_{G\in \cG_1} \cR(G)
%\cG^* := \{ G\in \cG_1: \cR(G) = \cR \}.
\end{align*}
to measure the difference in graph density between $G_\emptyset$ and $\cG_1$ in a worst case sense. Let $\cG^*$ be a nonempty subset of $\cG_1$ such that all graphs in $\cG^*$ have arboricity $\cR$. By the definition of $\cR$, such nonempty $\cG^*$ exists, and may not be unique. 
%By the definition of arboricity, $\cG^*$ contains the graphs in $\cG_1$ that have less 'dense' subgraphs. Intuitively, less dense graphs are more difficult to distinguish from empty graphs. 
Our analysis works for arbitrary choices of $\cG^*$ which satisfy the incoherence condition \citep{neykov2016combinatorial} defined as follows.
\begin{definition}(Negative association and incoherence condition) 
For $k\geq 0$, we say the random variables $Y_1,\ldots,Y_k$ are negatively associated if for any $k_1,k_2\geq 0$ with $k_1 + k_2 \leq k$, any distinct indices $i_1,i_2,\ldots, i_{k_1},j_1,\ldots,j_{k_2}$, and any coordinate-wise non-decreasing functions $f$ and $g$, we have
\begin{align*}
\Cov[ f(Y_{i_1},\ldots, Y_{i_{k_1}}),g(Y_{j_1},\ldots, Y_{j_{k_2}})] \leq 0.
\end{align*}
We say that the graph set $\cG^*$ is incoherent if for any fixed graph $G$, the binary random variables 
$$ \{\ind[i \in V(G')]\}_{i\in V(G)}$$ 
are negatively associated with respect to uniformly sampling $G'\in \cG^*$.	
\end{definition}
%Therefore, we can choose an appropriate $\cG^*$ and then focus on $\cG^*$ when calculating lower bound. 
%{\red Remark right after the definition that such $\cG^*$ is not difficult to construct?}
For a graph $G$, we denote by $\rmA_G$  the adjacency matrix of $G$. Then 
given $\cG^*$, we define the corresponding parameter set with minimal signal strength $\theta$ as
\begin{align*}
\cS^* = \{ \Theta = \theta \rmA_G: G \in \cG^*\}.
\end{align*}
Let $V_{\max} = \max_{G\in \cG^*} |V(G)|$, $\Lambda = \max_{G\in \cG^*} \|\rmA_{G}\|_F$, $\Gamma = \max_{G\in \cG^*} \|\rmA_{G}\|_1$ and $\cB = 512 \{ \Lambda^4  \land [ V_{\max}(\Gamma \lor \Lambda)^2 ]\}$. Then for $\theta \leq (2\Lambda)^{-1}$, by definition (recall \eqref{def:cS:cG1}) we have $\cS^* \subseteq \cS(\cG_1,\theta)$, and therefore
\begin{align}\label{eq:riskSstardef}
\gamma[\cS(\cG_1, \theta)] \geq \gamma(\cS^*) := \inf_{\psi}\bigg[\PP_{0,n}(\psi = 1) + \max_{\Theta\in\cS^*}\PP_{\Theta,n}(\psi = 0)\bigg].
\end{align}
%\begin{align*}
%\cK = \max_{G,G'\in\cG_1}\frac{\Delta_{E,E'}}{|V(E)\cap V(E')|},~~ \cB = 512 \big\{ \Lambda^4  \land \big[ V_{\max}(\Gamma \lor \Lambda)^2 \big]\big\},
%\end{align*}
%\fbox{We need to somehow give explanation about those.}\\
%\fbox{I know this is not easy, but we should think about it. They appear seemingly from nowhere.}
%Lemma~\ref{lemma:pkGVbound} relates the number of connected Eulerian subgraphs of $G$ containing certain vertices to its adjacency matrix $\rmA$. 
By \eqref{eq:riskSstardef}, it follows that to give a lower bound on $\gamma[\cS(\cG_1, \theta)]$ it suffices to lower bound $\gamma(\cS^*)$. We are ready to introduce our main theorem.

%\fbox{Define negative association.}
\begin{theorem}\label{thm:lowerbound}
	Let $\cG^*$ be a non-empty subset of $\cG_1$ such that all graphs in $\cG^*$ have arboricity $\cR$. Define 
    $N(\cG^*) := \max_{G\in\cG^*}\EE_{G'\sim U(\cG^*)} |V(G)\cap V(G')|$, where $U(\cG^*)$ is the uniform distribution over $\cG^*$. If $\cG^*$ is incoherent,  $N(\cG^*) = o(1)$, and
	\begin{align}\label{eq:informationlowerboundsignalstrength}
	\theta \leq \sqrt{\frac{\log [N^{-1}(\cG^*)]}{6n\cR}} \land \sqrt{\frac{\cR}{\cB}} \land \frac{1}{8(\Lambda \vee \Gamma)},
	\end{align}
	then we have
% 	\begin{align*}
% 	    {\red \liminf _{n \rightarrow \infty} \gamma[\cS(\cG_1,\theta)] = 1}.
% 	\end{align*}
	\begin{align*}
	\liminf _{n \rightarrow \infty} \gamma(\cS^*) = 1.
	\end{align*}
\end{theorem}

The following is a proof sketch of Theorem~\ref{thm:lowerbound}. 
The full proof is given in Section~\ref{section:five}.

\begin{proof}[Proof Sketch] The proof of Theorem~\ref{thm:lowerbound} follows the following three steps. 

\textbf{Step 1. Lower bounding the minimax risk using Le Cam's method.} We put uniform prior on the graphs in the alternative hypothesis and use Le Cam's method of fuzzy hypotheses to derive a lower bound of $\gamma(\cS^*)$. The resulting lower bound of the minimax risk is inversely related to the $\chi^2$-divergence between the null distribution and the uniform mixture of the distributions characterized by $\cS^*$. This converts the minimax lower bound problem into the problem of upper bounding the $\chi^2$-divergence.

\textbf{Step 2. Connecting the minimax risk lower bound to an Eulerian subgraphs counting problem.} We translate the $\chi^2$-divergence obtained in \textbf{Step 1} to a combinatorial quantity to ease the calculation. Specifically, based on the \textit{high temperature expansion} \citep{fisher1967critical} of Ising model densities, we show that the $\chi^2$-divergence can be bounded by a polynomial of $t = \tanh(\theta)$, whose coefficients are related to the number of certain Eulerian graphs defined by $\cG^*$.

\textbf{Step 3. Finalizing the proof using the incoherence condition of $\cG^*$.} In this step we further bound the polynomial derived in \textbf{Step 2} by neglecting the higher degree terms and increasing the coefficients in front of the lower degree terms (namely only the terms of degree $2, 3$ and $4$ survive). In addition, using that $\theta \leq \sqrt{\cR/\cB} \land [8(\Lambda \vee \Gamma)]^{-1}$ we can further fold the third and fourth degree terms into the second degree term. It turns out that the coefficient in front of the second degree term is bounded by $\cR | V(G)\cap V(G') |$, where the quantity $| V(G)\cap V(G') |$ is the number of overlapping vertices between graphs $G,G'$ in $\cG^*$. We then invoke the incoherence property of $\cG^*$, and show that under the conditions of Theorem~\ref{thm:lowerbound}, $\liminf_{n\rightarrow \infty} \gamma(\cS^*) = 1$.
% \CCC{Use high-temperature expansion. It is shown that the \\
% step 2. based on high-temperature expansion, we relate the bound with the number of Eulerian subgraphs of...\\
% step 3. Lemma 5.6}
\end{proof}

\begin{remark}
    The goal of this remark is to explain the intuition behind the first term on the right hand side of \eqref{eq:informationlowerboundsignalstrength}. It is intuitive that some sort of quantity measuring the density of the graph appears in our bound, since the denser a graph is, the easier it is to detect it. Moreover, it is clear that a single test examining the overlapping part of different graphs in $\cG^*$ can simultaneously distinguish these overlapped graphs from the null. As an extreme example, consider the case where all graphs in $\cG^*$ are perfectly overlapped, then $\cG^*$ is essentially a singleton, which is easy to detect. Therefore the degree of overlaps among graphs in $\cG^*$ also affects the difficulty of the test. In our bound in Theorem~\ref{thm:lowerbound}, the density of the graph is measured by the arboricity $\cR$, and the degree of overlaps is measured by the quantity $N(\cG^*)$.
\end{remark}

\begin{remark}
	Inequality \eqref{eq:informationlowerboundsignalstrength} shows that the necessary signal strength of detection problems is determined by the minimum of three terms. While the first term $\sqrt{\frac{\log [N^{-1}(\cG^*)]}{6n\cR}}$ is related to both the structural properties of graphs in $\cG_1$ and the sample size $n$, the second term $\sqrt{\frac{\cR}{\cB}}$ and third term $\frac{1}{8(\Lambda \vee \Gamma)}$ are independent of $n$. Therefore when the sample size is large enough, $\sqrt{\frac{\log [N^{-1}(\cG^*)]}{6n\cR}}$ is the leading term determining the necessary signal strength, and the other two terms mainly serve as scaling conditions of $\theta$. 
\end{remark}

\begin{remark}
    The condition \eqref{eq:informationlowerboundsignalstrength} given by Theorem \ref{thm:lowerbound} is comparable to the ``multi-edge'' results given in \cite{neykov2016combinatorial}, where the authors give minimax lower bounds of combinatorial inference problems in Gaussian graphical models. Unlike our results in Theorem \ref{thm:lowerbound}, the necessary signal strength for Gaussian graphical models given by \cite{neykov2016combinatorial} does not explicitly involve graph arboricity. It is also not very clear under what condition the lower bound given by \cite{neykov2016combinatorial} is sharp. In comparison, %Theorem \ref{thm:lowerbound}, together with the general framework of tests given in Section~\ref{section:upperbounds}, shows 
    in this paper we show that graph arboricity is an appropriate quantity that gives sharp lower bounds for \textit{any} structure detection problems under the incoherence condition and the sparsity assumption $s=O(d^{1/2-c})$ for some $c>0$. %As the counterpart of our results in Gaussian graphical models, this gives extra evidence of the sharpness of our results.
    It is also worth comparing Theorem \ref{thm:lowerbound} to the results of \cite{neykov2017property}. The lower bounds on the signal $\theta$ of \cite{neykov2017property}, typically involve the quantity $\sqrt{\frac{\log d}{n}}$ which is generally much larger than the right hand side of  \eqref{eq:informationlowerboundsignalstrength} when $\cR$ is large enough. This is intuitively clear since detection problems are statistically easier than graph property testing. Our proof strategy is also completely different than the one used by \cite{neykov2017property}, and relies on high temperature expansions rather than  Dobrushin's comparison theorem. 
\end{remark}

In Theorem~\ref{thm:lowerbound}, the incoherence condition of $\cG^*$ is not always easy to check. However, it is known that this condition is satisfied by a various discrete distributions including the multinomial and hypergeometric distributions \citep{joag1983negative,dubhashi1998balls}. In particular, Theorem~2.11 in \cite{joag1983negative} states that negative association holds for all permutation distributions. Therefore, for detection problems of the form \eqref{eq:testdefS2}, incoherence condition is always satisfied by picking $\cG^*$ to be the set of all graphs isomorphic to $G_*$. This leads to the following corollary (recall that we are assuming $s = o(\sqrt{d})$).

\begin{corollary}\label{cor:lowerbound}
	Let $G_*$ be a graph with $s$ vertices and $\cG_1(G_*)$ be the class of all graphs that contain a size-$s$ subgraph isomorphic to $G_*$. Let $\cB(G_*) = 512 \{ \|\rmA_{G_*}\|_F^4  \land [ (\|\rmA_{G_*}\|_1 \lor \|\rmA_{G_*}\|_F)^2s ]\}$. If
	\begin{align*}
	\theta \leq \sqrt{\frac{\log (d/s^2)}{6n\cR(G_*)}} \land \sqrt{\frac{\cR(G_*)}{\cB(G_*)}} \land \frac{1}{8(\|\rmA_{G_*}\|_F \vee \|\rmA_{G_*}\|_1)},
	\end{align*}
	then we have
	\begin{align*}
	    \liminf _{n \rightarrow \infty} \gamma\{ \cS[\cG_1(G_*),\theta] \} = 1.
	\end{align*}
% 	\begin{align*}
% 	\liminf _{n \rightarrow \infty} \gamma(\cS^*) = 1.
% 	\end{align*}
\end{corollary}

\subsection{Examples}\label{section:lowerboundexample}
In this section we apply Corollary~\ref{cor:lowerbound} to specific detection problems. %{\color{red} Through out this section we denote by $U(\cG^*)$ the uniform distribution over $\cG^*$.} \fbox{Do we use this $U(\cG^*)$?}

\begin{example}[Empty graph versus non-empty graph]
	Consider testing empty graph versus non-empty graph defined in Section~\ref{section:introduction}. If
	\begin{align}\label{ex1:bound}
	\theta \leq \sqrt{\frac{\log (d/4)}{6n}}\land \frac{1}{32\sqrt{2}},
	\end{align}
	we have $\liminf_{n\rightarrow \infty} \gamma[\cS(\cG_1,\theta)] = 1$.
% 	{\red we have $\liminf_{n\rightarrow \infty} \gamma(\cS^*) = 1$.}
\end{example}
\begin{proof}
	In this example $s=2$, $\cG_1 = \cG_1(G_*)$, where $G_*$ is a single-edge graph, and we have $\cR(G_*) = 1$. By a  direct calculation we have $\| \rmA_{G_*} \|_F = \sqrt{2}$, $\| \rmA_{G_*} \|_1 = 1$, and $\cB(G_*) = 2048$.  By Corollary~\ref{cor:lowerbound}, if \eqref{ex1:bound} holds 
% 	\begin{align*}
% 	\theta \leq \sqrt{\frac{\log (d/4)}{6n}} \land \frac{1}{32\sqrt{2}},
% 	\end{align*}
	we have $\liminf_{n\rightarrow \infty} \gamma(\cS^*) = 1$.
\end{proof}

\begin{example}[Clique Detection]
For the clique detection problem defined in Section~\ref{section:introduction}, if 
\begin{align}\label{ex2:bound}
\theta \leq \sqrt{\frac{\log (d/s^2)}{6ns}} \land \frac{1}{32s},
\end{align}
we have $\liminf_{n\rightarrow \infty} \gamma[\cS(\cG_1,\theta)] = 1$.
% {\red we have $\liminf_{n\rightarrow \infty} \gamma(\cS^*) = 1$.}
% we have $\liminf_{n\rightarrow \infty} \gamma(\cS^*) = 1$.
\end{example}
\begin{proof}
	In this example $\cG_1 = \cG_1(G_*)$ with $G_*$ being an $s$-clique graph. We have $\cR(G_*) = \lceil s/2 \rceil$ and therefore $s/2 \leq \cR(G_*) \leq s$. By direct calculation we have $\| \rmA_{G_*} \|_F = \sqrt{s(s-1)}\leq s$, $\| \rmA_{G_*} \|_1 = s-1 \leq s$, and therefore $\cB(G_*) \leq 512 s^3$. By Corollary~\ref{cor:lowerbound}, if \eqref{ex2:bound} holds 
% 	\begin{align*}
% 	\theta \leq \sqrt{\frac{\log (d/s^2)}{6ns}} \land \frac{1}{32 s},
% 	\end{align*}
	we have $\liminf_{n\rightarrow \infty} \gamma(\cS^*) = 1$.
\end{proof}

\begin{example}[Star Detection]
	For the star detection problem defined in Section~\ref{section:introduction}, if $s\geq 4$ and
	\begin{align}\label{ex3:bound}
	\theta \leq \sqrt{\frac{\log (d/s^2)}{6n}} \land \frac{1}{32\sqrt{2}s},
	\end{align}
	we have $\liminf_{n\rightarrow \infty} \gamma[\cS(\cG_1,\theta)] = 1$.
% 	{\red we have $\liminf_{n\rightarrow \infty} \gamma(\cS^*) = 1$.}
% 	then $\liminf_{n\rightarrow \infty} \gamma(\cS^*) = 1$.
\end{example}
\begin{proof}
	In this example $G_*$ is a star graph and we have $\cR(G_*) =1$. By direct calculation we have $\| \rmA_{G_*} \|_F = \sqrt{2(s-1)}\leq \sqrt{2s}$, $\| \rmA_{G_*} \|_1 = s-1 \leq s$. If $s\geq 4$, we have $\cB(G_*) = 2048 s^2$. By Corollary~\ref{cor:lowerbound}, if \eqref{ex3:bound} holds 
% 	\begin{align*}
% 	\theta \leq \sqrt{\frac{\log (d/s^2)}{6n}} \land \frac{1}{32\sqrt{2} s},
% 	\end{align*}
	we have $\liminf_{n\rightarrow \infty} \gamma(\cS^*) = 1$.
\end{proof}

\begin{example}[Community structure detection]
	For the community structure detection problem defined in Section~\ref{section:introduction}, if $k\geq 4$, $l\geq 2$ and
	\begin{align}\label{ex4:bound}
	\theta \leq \sqrt{\frac{\log (d/s^2)}{6n(l \lor k)}} \land \frac{1}{32\sqrt{2} s},
	\end{align}
	we have $\liminf_{n\rightarrow \infty} \gamma[\cS(\cG_1,\theta)] = 1$.
% 	{\red we have $\liminf_{n\rightarrow \infty} \gamma(\cS^*) = 1$.}
% 	we have $\liminf_{n\rightarrow \infty} \gamma(\cS^*) = 1$.
\end{example}
\begin{proof} We have $\cG_1 = \cG_1(G_*)$, where $G_*$ is a community structure graph defined in Section~\ref{section:introduction}. To calculate $\cR(G_*)$, we utilize the fact that arboricity equals the minimum number of forests into which the edges of a given graph can be partitioned \citep{nash1961edge}. Let $\cC_1,\ldots,\cC_l$ be the communities. For $i=1,\ldots,l$, we know that $\cC_i$ is a $k$-clique, and the arboricity is $\lceil k/2 \rceil$. Therefore inside $\cC_i$, we can partition the graph into $\lceil k/2 \rceil$ forests. There is also an $l$-clique in $G_*$ consisting of the cross-community edges. This clique can be partitioned into $\lceil l/2 \rceil$ forests. Note that this $l$-clique shares only one vertex $v(\cC_i)$ with the community $\cC_i$. Therefore for any forest in the partition of this $l$-clique and any forest in the partition of $\cC_i$, we can merge them into a single forest because the resulting graph is still acyclic. We can keep merging forests from other communities. Eventually, we can merge $l$ forests from distinct communities to a forest in the $l$-clique, without introducing any cycles. If $\lceil l /2 \rceil \geq \lceil k/2 \rceil$, we will obtain $\lceil l /2 \rceil$ forests that form a partition of $G_*$; if $\lceil l /2 \rceil < \lceil k/2 \rceil$, then the partition will contain $\lceil k/2 \rceil$ forests. Therefore by the equivalent definition of arboricity given in \citep{nash1961edge} we have $\cR(G_*) \leq \lceil (l \lor k)/2 \rceil$. On the other hand, since $G_*$ contains an $l$ clique, obviously $ \cR(G_*) \geq \lceil l/2 \rceil $. Similarly, $ \cR(G_*) \geq \lceil k/2 \rceil $ and hence we have $\cR(G_*) = \lceil (l \lor k)/2 \rceil$. Therefore $(l \lor k)/2 \leq  \cR(G_*)\leq l \lor k$.
	
	By direct calculation, we have $\| \rmA_{G_*} \|_F = \sqrt{l k(k-1) + l(l -1)}\leq \sqrt{l k^2 + l^2}$, $\| \rmA_{G_*} \|_1 = k-1 + l-1 \leq k + l$. We now compare the upper bounds of $\|\rmA_{G_*}\|_F$ and $\|\rmA_{G_*}\|_1$. If $k\geq 4$ and $l\geq 2$, we have $l \geq 1+ l/2$ and
	\begin{align*}
	l k^2 + l^2 \geq (1+l/2)k^2 +l^2 = k^2 + l k^2/2 +l^2 \geq k^2 + 2kl + l^2 = (k+l)^2.
	\end{align*}
	Therefore 
	$$
	\cB(G_*) \leq 512 [ (l k^2 + l^2)^2 \land ( l k^2 + l^2 )lk ] = 512 ( l k^2 + l^2 )lk = 512 (s^2 k + sl^2), 
	$$
	and
	\begin{align*}
	\sqrt{\frac{\cR(G_*)}{\cB(G_*)}} & \geq \sqrt{\frac{l \lor k}{ 1024(s^2 k + sl^2)}} \geq \sqrt{\frac{l \lor k}{ 1024(s^2 k + s^2l)}} \geq \sqrt{\frac{l \lor k}{ 2048s^2( k \lor l)}} = \frac{1}{32\sqrt{2}s}.
	\end{align*}
	Moreover,
	\begin{align*}
	\frac{1}{8(\|\rmA_{G_*}\|_F \vee \|\rmA_{G_*}\|_1)} \geq \frac{1}{8 \sqrt{lk^2 + l^2}}\geq \frac{1}{8\sqrt{2}s}.
	\end{align*}
	Therefore by Corollary~\ref{cor:lowerbound}, if \eqref{ex4:bound} holds 
% 	\begin{align*}
% 	\theta \leq \sqrt{\frac{\log (d/s^2)}{6n(k \lor s')}} \land \frac{1}{32\sqrt{2} s},
% 	\end{align*}
	we have $\liminf_{n\rightarrow \infty} \gamma(\cS^*) = 1$.
\end{proof}

\section{Upper Bounds}\label{section:upperbounds}
In this section we construct upper bounds for the hypothesis testing problem \eqref{eq:testdef}. We propose a general framework for testing an empty graph $G_\varnothing$ against an arbitrary graph set $\cG_1$. We remind the reader that the arboricity of a graph $G$ is defined in \eqref{arboricity:def} as
\begin{align*}
\cR(G) := \bigg\lceil \max_{V\subseteq \overline{V}} \frac{|E(G_V)|}{|V| - 1} ~\bigg\rceil,
\end{align*}
where $G_V$ is the graph obtained by restricting $G$ on the vertex set $V$. The arboricity $\cR$ of $\cG_1$ is then defined as
\begin{align*}
\cR := \min_{G\in \cG_1} \cR(G).
%\cG^* := \{ G\in \cG_1: \cR(G) = \cR \}
\end{align*}
We now introduce the concept of witnessing subgraph and witnessing set. Before that, we remind the reader, that in this paper all graphs have $d$ vertices (i.e., all graphs are over the vertex set $\overline V$), unless otherwise specified. Therefore a subgraph $G'$ of a graph $G = (\overline V, E)$ is a graph with $d$ vertices whose edge set is a subset of the edge set of the larger graph, i.e., $G' = (\overline V, E')$ where $E' \subseteq E$. Importantly, the notation $V(G)$ and $V(G')$ refer to the \underline{non-isolated} vertices of $G$ and $G'$ which may be strict subsets of $\overline V$.%\\\fbox{What is a subgraph? Is it vertex set equal, or vertex set is a subset?}

\begin{definition}[Witnessing Subgraph] %For a graph $G\in \cG_1$ and a graph $H$, if $H$ is a subgraph of $G$ and $\lceil  |E(H)|/[|V(H)|-1]  \rceil = \cR$, we call $H$ a witnessing subgraph of $G$ with respect to $\cG_1$.
For a graph $G\in \cG_1$ we call the graph $H$ a witnessing subgraph of $G$ with respect to $\cG_1$, if $H$ is a subgraph of $G$ and $\lceil  |E(H)|/[|V(H)|-1]  \rceil \geq \cR$.
\end{definition}
Here we remark that for $H$ to be a witnessing subgraph of $G$, it is unnecessary to have $\lceil |E(H)|/[|V(H)|-1]  \rceil = \cR(G)$. Instead, we only require that $\lceil  |E(H)|/[|V(H)|-1]  \rceil \geq \cR$, which is a weaker requirement since by definition we have $\cR \leq \cR(G)$ for any $G\in \cG_1$. This implies that every graph $G \in \cG_1$ has at least one witnessing graph, which may be obtained from the densest subgraph of $G$ (with potential edge pruning). 
\begin{definition}[Witnessing Set]
	%Let $\cH$ be a set of graphs. If for every $G\in \cG_1$, there exists  $H\in \cH$ such that $H$ is a witnessing subgraph of $G$, we call $\cH$ a witnessing set of $\cG_1$.
	We call the collection of graphs $\cH$ a witnessing set of  $\cG_1$, if for every $G\in \cG_1$, there exists  $H\in \cH$ such that $H$ is a witnessing subgraph of $G$.
\end{definition}
%The intuition behind the definition of witnessing set $\cH$ is that, as long as every $H\in \cH$ has arboricity $\cR$, tests that scan over $\cH$ can potentially match the lower bound result given by Theorem~\ref{thm:lowerbound}. 
By the definition of $\cR$, and as we previously argued, every graph $G\in \cG_1$ must have at least one witnessing subgraph. Therefore at least one witnessing set $\cH$ of $\cG_1$ exists. We define the set of witnessing graphs in order to facilitate the development of \textit{scan tests}. Below we will formalize a test statistic which scans over all graphs in $\cH$. Importantly, in order to match the lower bound result given by Theorem~\ref{thm:lowerbound}, it is not sufficient to scan directly over the graphs from the set $\cG_1$. This is because the graphs in $\cG_1$ may contain non-essential edges which may introduce noise during the testing. In contrast, the graphs from $\cH$ trim down those non-essential edges and focus only on the essential parts of the graphs in $\cG_1$.
 %In practice, it is usually unnecessary to design a unique witnessing subgraph for every graph $G\in \cG_1$. 
%{\color{red}The intuition of constructing such witnessing set $\cH$ is that the original graph class $\cG_1$ may contain a large proportion of graphs that can be easily tested, which makes it unnecessary to design a unique statistic for every $G\in \cG_1$. }

We now introduce our general testing procedure. Our test is based on a witnessing set $\cH$. For $H\in \cH$ we define
\begin{align}\label{WH:def}
%W_{H} := \frac{1}{|E(H)|} \sum_{(i,j)\in E(H)} X_iX_j,~~ \hat{W}_{H} := \frac{1}{n}\cdot \sum_{l=1}^{n} \Bigg( \frac{1}{|E(H)|} \sum_{(i,j)\in E(H)}  X_{l,i} X_{l,j}\Bigg),
 \hat{W}_{H} := \frac{1}{n}\cdot \sum_{l=1}^{n} \Bigg( \frac{1}{|E(H)|} \sum_{(i,j)\in E(H)}  X_{l,i} X_{l,j}\Bigg),
\end{align}
where $\bX_l$ is the $l$-th sample and $X_{l,i}$, $X_{l,j}$ are the $i$-th and $j$-th components of $\bX_l$ respectively. Our test then scans over all possible $H\in \cH$ and calculates the corresponding $\hat{W}_{H}$. We define
\begin{align}\label{psi:def}
\psi := \mathds{1} \Bigg[ \max_{H\in \cH} \hat{W}_H > \frac{\kappa}{4} \sqrt{\frac{M(\cH)}{\cR n}}\Bigg],
\end{align}
where $\kappa$ is a large enough absolute constant, and
\begin{align*}
 M(\cH) := \frac{\log(|\cH|)}{m(\cH)}, \text{ with } m(\cH):= \min_{H\in \cH}  |V(H)|. 
\end{align*}
The following theorem justifies the usage of the test defined in \eqref{psi:def}. %gives the test based on $\{\hat{W}_H\}_{H\in\cH}$. 
\begin{theorem}\label{thm:upperbound}
	Given any fixed $\alpha\in (0,1)$, suppose that $\log (|\cH|) /n = o(1)$ and $|\cH|\geq 2/\alpha$. If 
	\begin{align*}
	\theta > \kappa \sqrt{\frac{M(\cH)}{\cR n}}
	\end{align*}
	for a large enough absolute constant $\kappa$, %then for the test
%	\begin{align*}
%	\psi = \mathds{1} \Bigg[ \max_{H\in \cH} \hat{W}_H > \frac{\kappa}{4} \sqrt{\frac{M(\cH)}{\cR n}}\Bigg],
%	\end{align*}
	when $n$ is large enough we have that the test $\psi$ of \eqref{psi:def} satisfies
	\begin{align*}
	\PP_{0,n}(\psi = 1) + \max_{\Theta\in\cS(\cG_1,\theta)}\PP_{\Theta,n}(\psi = 0) \leq \alpha.
	\end{align*}
\end{theorem}
We give a proof sketch of Theorem~\ref{thm:upperbound} as follows.  
The detailed proof is given in Section~\ref{section:six}.

\begin{proof}[Proof Sketch] The proof of Theorem~\ref{thm:upperbound} follows the following three steps. 

\textbf{Step 1. Bounding the $\psi_1$-norm of the test variable.} The key observation in the proof is to establish that each of the i.i.d. summands of $\hat W_H$ \eqref{WH:def} is a sub-exponential random variable. To prove this we use a result of \cite{bhattacharya2015inference} and our assumption that $\|\Theta\|_F \leq 1/2$.

\textbf{Step 2. Upper/lower bounding $\hat W_H$ under the null/alternative.} Using the result of \textbf{Step 1} and the second Griffith's inequality we establish an upper bound on $\hat W_H$ under the null, and a lower bound on $\hat W_H$ under any graph from the alternative hypothesis. 

\textbf{Step 3. Proof completion.} We show that under the assumptions of Theorem 3.3. there is a sufficient gap between the upper and lower bounds established in \textbf{Step 2}, which renders the final claim of the theorem.
\end{proof}

%\begin{proof}
%\end{proof}
\begin{remark}
	We can compare our upper bound result with Corollary~\ref{cor:lowerbound}. For testing problems of the form \eqref{eq:testdefS2}, we can always choose a subgraph $H_*$ of $G_*$ as a witnessing subgraph (if there are multiple such subgraphs pick any of them), and construct $\cH$ to be the set consisting of all graphs isomorphic to $H_*$. 
	%\fbox{$H_*$ may not have $s$ vertices?? Does the bound below still apply?}
	 For this $\cH$ we have $|\cH| \leq \frac{d!}{(d-|V(H_*)|)!}$. Therefore 
	\begin{align*}
		M(\cH) \leq |V(H_*)|^{-1} \log[d! / (d-|V(H_*)|)!] \leq \log(d).
	\end{align*}
	If $s=O(d^{1/2 - c})$ for some $c>0$, $\log(d/s^2)$ is also of order $\log(d)$. Therefore the rate given by Theorem~\ref{thm:upperbound} matches Corollary~\ref{cor:lowerbound}. When applying Theorem~\ref{thm:upperbound} to specific detection problems, potentially there might exist different choices of $\cH$, which may result in lower values of $M(\cH)$. 
\end{remark}

\subsection{Examples}\label{subsection:upperexample}
In this section we apply Theorem~\ref{thm:upperbound} to the examples we discussed in Section~\ref{section:introduction} and Section~\ref{section:lowerboundexample}. 
\begin{example}[Empty graph versus non-empty graph] 
	Consider testing empty graph versus non-empty graph defined in Section~\ref{section:introduction}. If  $\log(d)/n = o(1)$, $4/[d(d-1)] \leq \alpha$ and 
	\begin{align}\label{ex1:upperbound}
	\theta > \kappa \sqrt{\frac{\log d}{n}}
	\end{align}
	for a large enough constant $\kappa$, then when $n$ is large enough, we have
	\begin{align}\label{ex1:typeItypeII:err}
	\PP_{0,n}(\psi = 1) + \max_{\Theta\in\cS(\cG_1,\theta)}\PP_{\Theta,n}(\psi = 0) \leq \alpha.
	\end{align}
\end{example}
\begin{proof}
	 In this example $\cR = 1$, and hence $\cH = \{\text{single-edge graphs}\}$ is a witnessing set of $\cG_1$. We have $|\cH| = d(d-1)/2$, $m(\cH)=2$ and $M(\cH) = \log(|\cH|)/m(\cH)\leq \log d$. Therefore by Theorem~\ref{thm:upperbound}, if
	 \eqref{ex1:upperbound} holds for a large enough constant $\kappa$, then when $n$ is large enough, we have that \eqref{ex1:typeItypeII:err} holds.
% 	 \begin{align*}
% 	 \PP_{0,n}(\psi = 1) + \max_{\Theta\in\cS(\cG_1,\theta)}\PP_{\Theta,n}(\psi = 0) \leq \alpha.
% 	 \end{align*}
\end{proof}

\begin{example}[Clique Detection] \label{ex:uppercliquedetection}
	For the clique detection problem defined in Section~\ref{section:introduction}, if $s\log (ed/s)/n = o(1)$, $(d/s)^s \geq 2/\alpha$ and 
	\begin{align}\label{ex2:upperbound}
	\theta > \kappa \sqrt{\frac{\log (ed/s)}{sn}}
	\end{align}
	for a large enough constant $\kappa$, then when $n$ is large enough, we have
	\begin{align}\label{ex2:typeItypeII:err}
	\PP_{0,n}(\psi = 1) + \max_{\Theta\in\cS(\cG_1,\theta)}\PP_{\Theta,n}(\psi = 0) \leq \alpha.
	\end{align}
\end{example}
\begin{proof}
	In this example we have $\cR = \lceil s/2 \rceil$, and $\cH = \{s\text{-cliques}\}$ is a witnessing set of $\cG_1$. We have $|\cH| = \binom{d}{s}$, and therefore $(d/s)^s \leq |\cH| \leq (e d/s)^s$. We have $m(\cH) = s$ and $M(\cH) = \log(|\cH|)/m(\cH)\leq \log (ed/s)$. Therefore by Theorem~\ref{thm:upperbound}, if \eqref{ex2:upperbound} holds
	for a large enough constant $\kappa$, then when $n$ is large enough, we have that \eqref{ex2:typeItypeII:err} holds.
% 	\begin{align*}
% 	\PP_{0,n}(\psi = 1) + \max_{\Theta\in\cS(\cG_1,\theta)}\PP_{\Theta,n}(\psi = 0) \leq \alpha.
% 	\end{align*}
\end{proof}

\begin{example}[Star Detection] 
	For the star detection problem defined in Section~\ref{section:introduction}, if $\log (d)/n = o(1)$, $4/[d(d-1)]\leq \alpha$ and 
	\begin{align}\label{ex3:upperbound}
	\theta > \kappa \sqrt{\frac{\log (ed/s)}{n}}
	\end{align}
	for a large enough constant $\kappa$, then when $n$ is large enough, we have
	\begin{align}\label{ex3:typeItypeII:err}
	\PP_{0,n}(\psi = 1) + \max_{\Theta\in\cS(\cG_1,\theta)}\PP_{\Theta,n}(\psi = 0) \leq \alpha.
	\end{align}
\end{example}
\begin{proof}
    In this example we have $\cR = 1$, and $\cH = \{(s-1)\text{-stars}\}$ is a witnessing set of $\cG_1$. We have $|\cH| = s\binom{d}{s}$, and therefore $s(d/s)^s \leq |\cH| \leq s(e d/s)^s$. We have $m(\cH) = s$. When $s = o(\sqrt{d})$ we have $s \leq (ed/s)^s$ and $M(\cH) = \log(|\cH|)/m(\cH)\leq 2\log (ed/s) $. Therefore by Theorem~\ref{thm:upperbound}, if \eqref{ex3:upperbound} holds
	for a large enough constant $\kappa$, then when $n$ is large enough, we have that \eqref{ex3:typeItypeII:err} holds.
\end{proof}

\iffalse
\begin{proof}
	In this example we have $\cR = 1$, and $\cH = \{\text{single-edge graphs}\}$ is a witnessing set of $\cG_1$. The rest of the proof is the same as the empty graph versus non-empty graph problem, and we omit the details.
\end{proof}
\fi

\iffalse

\begin{example}[Star Detection] In this example we have $\cR(G) = 1$ for all $G\in \cG_1$, and we can directly choose a witnessing set $\cH$ to be the set of all $(s-1)$-stars. We have $|\cH| = \binom{d}{s}s$, and $M(\cH) = \log(|\cH|)/m(\cH)\leq C\log (d/s)$ where $C$ is an absolute constant. Therefore by Theorem~\ref{thm:upperbound}, if $\theta \leq 1/(8\sqrt{s})$, $\log(2/\alpha)]/n = o(1)$ and 
	\begin{align*}
	\theta > \kappa \sqrt{\frac{\log (d/s)}{n}}
	\end{align*}
	for a large enough constant $\kappa$, then when $n$ is large enough, we have
	\begin{align*}
	\PP_{0,n}(\psi = 1) + \sup_{G\in \cG_1} \PP_{\Theta,n}(\psi = 0) \leq \alpha.
	\end{align*}
\end{example}
\fi

\begin{example}[Community structure detection]
	Consider the community structure detection problem defined in Section~\ref{section:introduction}. If
	$(l\lor k)\log [ed/(l\lor k)]/n = o(1)$, $[d/(l\lor k)]^{(l\lor k)} \geq 2/\alpha$ and 
	\begin{align*}
		\theta > \kappa \sqrt{\frac{\log [ed/(l\lor k)]}{(l\lor k)n}}
	\end{align*}
	for a large enough constant $\kappa$, then when $n$ is large enough, we have
	\begin{align*}
	\PP_{0,n}(\psi = 1) + \max_{\Theta\in\cS(\cG_1,\theta)}\PP_{\Theta,n}(\psi = 0) \leq \alpha.
	\end{align*}
\end{example}
\begin{proof}
	If $l \geq k$, we have $\cR = \lceil l/2 \rceil$, and we can choose $\cH = \{l\text{-cliques}\}$ as a witnessing set of $\cG_1$; if $l < k$, we have $\cR = \lceil k/2 \rceil$, and $\cH = \{k\text{-cliques}\}$ is a witnessing set of $\cG_1$. The rest of proof is identical to the clique detection problem, and we omit the details.
\end{proof}

\section{Computational Lower Bound}\label{section:complowerbound}
In this section we give our main result on the computational lower bound of structure testing problems in Ising models. Our result is based on a sparse PCA conjecture. Denote by $1_{i_1,\ldots, i_s} = e_{i_1}+\cdots+e_{i_s} \in \RR^d$ the vector whose $i_1,\ldots,i_s$-th entries are $1$ and other entries are $0$. Let
\begin{align*}
\cS_\sigma = \big\{\rmI + \sigma 1_{i_1,\ldots, i_s}^{}1_{i_1,\ldots, i_s}^T\in \RR^{d\times d}: i_1,\ldots, i_s \in \{1,\ldots,d\} \text{ are distinct} \big\}
\end{align*}
% \begin{align*}
% \cS_\sigma = \big\{\Sigma = \rmI + \sigma 1_{i_1,\ldots, i_s}^{}1_{i_1,\ldots, i_s}^T\in \RR^{d\times d}: , i_1,\ldots, i_s \text{ are } s \\ \text{ distinct indices in } \{1,\ldots,d\} \big\},
% \end{align*}
be the set of covariance matrices from the Gaussian spiked model. In sparse PCA, we consider the hypothesis testing problem for $n$ i.i.d samples $\bZ_1,\ldots, \bZ_n\in \RR^d$:
\begin{align}\label{eq:PCAdef}
\rmH_0^{\mathrm{PCA}}: \bZ_1,\ldots,\bZ_n \sim N(0,\rmI) \text{ versus }\rmH_1^{\mathrm{PCA}}: \bZ_1,\ldots,\bZ_n \sim N(0,\Sigma), \Sigma\in \cS_\sigma.
\end{align}
We denote by $\PP_{\rmI,n}$ and $\PP_{\Sigma,n}$ the probability measure under $\rmH_0^{\mathrm{PCA}}$ and $\rmH_1^{\mathrm{PCA}}$ respectively.
\begin{conjecture}[Computational Hardness of Sparse PCA]\label{conj:PCAassumption}
	Let $\delta > 0$ be any fixed small constant. If $\sigma \leq \eta [n^{-(1/2+\delta)} \land s^{-(1+\delta)}]$ for some small enough constant $\eta$, then for any polynomial time test $\psi$, we have
	\begin{align*}
	\liminf_{n \rightarrow \infty} \bigg[\PP_{\rmI,n}(\psi = 1) + \max_{\Sigma\in \cS_\sigma}\PP_{\Sigma,n}(\psi = 0)\bigg]\geq \frac{1}{4}. 
	\end{align*}
\end{conjecture}
Conjecture~\ref{conj:PCAassumption} is derived by \cite{gao2014sparse} under the widely believed planted clique conjecture and additional assumptions which essentially require that $2n\leq d\leq n^a$ for some constant $a>1$ and $n[\log(n)]^5\leq C s^4$ for some small enough constant $C>0$. It is also studied in \cite{berthet2013complexity}, \cite{brennan2018reducibility} and  \cite{brennan2019optimal}. In particular the latter two papers prove Conjecture \ref{conj:PCAassumption} for the two regimes $s \gg \sqrt{d}$ \cite{brennan2018reducibility} and $s < \sqrt{d}$ \cite{brennan2019optimal} respectively based on the planted clique conjecture.%in the case when $d$ is of the order of $n$.

% {\color{red} $U = \mathrm{sign} (W)$, $V$ Ising, $W$ Gaussian, $G$, still Gaussian, $W_i = G_i + G$, $G'$: CWN}

In the following, based on the Gaussian random vectors from the sparse PCA problem, we propose a polynomial time reduction algorithm that constructs $n$ $d$-dimensional random vectors which cannot be distinguished from $n$ samples from the $d$-dimensional Ising model with a parameter matrix $\Theta$. Importantly, this reduction only needs to be done for clique graphs because  detecting the $s$-clique containing $G_*$ is always easier than directly detecting $G_*$ (recall that we are testing whether the underlying graph is empty, or contains a graph isomorphic to $G_*$). Furthermore, in the sparse PCA problem each $\Sigma\in \cS_\sigma$ corresponds to an $s$-clique. More specifically, for any index set $\cI\subseteq \{1,\ldots,d\}$ of size $s$ representing the position of a clique, we consider probability measure $\PP_{\Theta, n}$ of the Ising model with parameter matrix $\Theta = \theta \cdot [\mathrm{1}(i,j\in \cI, i\neq j)]_{d\times d}$ (when $\theta = 0$ we denote $\PP_{0,n}$ the Ising clique model under the null hypothesis), and $\PP_{\Sigma, n}$ for the multivariate Gaussian distribution with mean $0$ and covariance matrix $\Sigma = \rmI + \sigma 1_\cI^{} 1_\cI^T$. 

We are now ready to state the main result of this section. We have

\begin{theorem}[Computational Hardness of Ising Clique Detection]\label{thm:Ising:computation}
	Suppose that Conjecture \ref{conj:PCAassumption} holds. If $\theta \leq \eta [n^{-(1/2+\delta)} \land s^{-(1+\delta)}]$ for some small enough constant $\eta$, then for any polynomial time test $\psi$ and an $s$-clique graph $G_*$, we have
	\begin{align*}
	\liminf_{n \rightarrow \infty} \bigg[\PP_{0,n}(\psi = 1) + \max_{\Theta\in \cS(\cG_1(G_*), \theta)}\PP_{\Theta,n}(\psi = 0)\bigg]\geq \frac{1}{4}. 
	\end{align*}
\end{theorem}

In order to prove the computational hardness result above, we need to formalize a polynomial time transformation which (approximately in total variation) maps the null hypothesis of a sparse PCA model to the null hypothesis of the Ising clique model, and simultaneously maps the corresponding alternative hypothesis of sparse PCA to the alternative hypothesis of the Ising clique model up to a small total variation. 

Our construction is surprisingly simple: taking the signs of Gaussians and showing that they are close enough to an Ising model in total variation. Clearly, for this type of reduction, the null hypothesis of a sparse PCA model is mapped to exactly the null hypothesis of Ising clique model. Similarly, under the alternative hypothesis of the sparse PCA model, the vertices out of the clique are still mapped to Rademacher random variables, which also correspond to the vertices out of the Ising clique. Therefore, we only need to focus on the vertices in the Gaussian and Ising cliques under the alternative hypotheses. We introduce the following notation for the random Gaussian/Ising vectors corresponding to the vertices in the clique. Let $\bW_1,\ldots \bW_n $ and $\bV_1,\ldots,\bV_n $ be $s$-dimensional i.i.d. Gaussian and Ising random vectors with parameter matrices $\rmI_{s \times s} + \sigma 1_{[s]} 1_{[s]}^T$ and $\theta (1_{[s]} 1_{[s]}^T - \rmI_{s \times s} )$ respectively. Denote $\bU_i = \mathrm{sign} (\bW_i)$, $i\in [n]$.

% $U = \mathrm{sign} (W)$, $V$ Ising, $W$ Gaussian, $G$, still Gaussian, $W_i = G_i + G$, $G'$: CWN
% Our construction is surprisingly simple{\color{red}: taking the signs of Gaussians and show that they are close enough to an Ising model in total variation.}

Before we introduce the proof details%reduction scheme,
, it is necessary to determine the parameter $\sigma$ for any fixed $\theta$. Part of our choice of $\sigma$ is based on matching the first terms in the Taylor expansions of the functions $x \mapsto \Phi(x)$ (here $\Phi(x)$ is the standard normal cumulative density function) and $x \mapsto \exp(c x) / [\exp(cx) + \exp(-cx)]$ for some constant $c$ (turns out that the ``best'' $c = \sqrt{2/\pi}$). Supposing that $s\theta < \frac{1}{2}$, we set $\sigma = \pi\theta / (1 - 2s\theta)$.
% $$\sigma = \frac{\pi\theta}{1 - 2s\theta}.$$}

% Let $\bY_i = \sign(\bZ_i)$, $i=1,\ldots,n$. Taking the signs can clearly be done in polynomial time. Note that if $\bZ_i \sim N(0, \rmI)$ the transformation $\sign(\bZ_i)$ automatically maps $\bZ_i$ to Rademacher random variables which corresponds to the null hypothesis of the Ising model. Hence in what follows we will show that $\sign(\bZ_i)$ are close to $\bX_i$ when $\{\bX_i\}_{i = 1}^n \sim \PP_{\Theta, n}$ and $\{\bZ_i\}_{i = 1}^n \sim \PP_{\Sigma, n}$. In addition, even under the alternative hypothesis, when $j$ does not belong to the clique, i.e., $j \not \in \cI$, $\sign(\bZ_{ij})$ is a Rademacher random variable just like $\bX_{ij}$ is. Therefore (since total variation is sub-additive on product distributions) it suffices to show that the total variation is small for indexes $j$ which belong to the clique $\cI$. 

Recall now that if one has a model $\bW \sim N(0, \rmI_{s \times s} + \sigma 1_{[s]} 1_{[s]}^T)$, this is equivalent to having generated i.i.d. $Y \sim N(0,1)$ and $Y_i \sim N(0,1)$ for $i \in [s]$, and set $W_i = Y_i + \sqrt{\sigma}Y$. This means that in terms of signs of $\bW$, i.e., $\sign(W_i)$, the generation is as follows: first generate $Y \sim N(0,1)$ and then generate $1$ or $-1$ conditionally i.i.d. on $Y$ with probability $\Phi(\sqrt{\sigma}Y)$ or $1-\Phi(\sqrt{\sigma}Y)$ respectively. Our next lemma, which is the key to the proof of bounding the total variation argues that similar conditional i.i.d. decomposition holds for the Curie-Weiss model. Recall that the Curie-Weiss model is given by
% \begin{align*}
%     \PP_{\theta}(\bX = \bx) \varpropto \exp\bigg(\theta \sum_{i \neq j} x_i x_j\bigg)\varpropto \exp\bigg(\theta \bigg(\sum_{i} x_i\bigg)^2\bigg),
% \end{align*}
\begin{align*}
    \PP_{\theta}(\bV = \bv) \varpropto \exp\bigg(\theta \sum_{i \neq j} v_i v_j\bigg)\varpropto \exp\bigg(\theta \bigg(\sum_{i} v_i\bigg)^2\bigg),
\end{align*}
for $\bv \in \{\pm 1\}^s$. Clearly the Curie-Weiss model corresponds to the ``clique'' part of the Ising clique model with parameter matrix $\Theta$. 

\begin{lemma}[Generating Curie-Weiss as Conditional i.i.d.]\label{CWN:lemma} The Curie-Weiss model with $s$ vertices and parameter $\theta \geq 0$ can be generated as: 1) generate $Y'$ coming from a distribution with density proportional to $p_{Y'}(y) \varpropto \cosh(\sqrt{2\theta}y)^s \exp(-y^2/2)$ and 2) generate $V_i |Y'$ i.i.d. with probabilities $\PP(V_i = 1 | Y') = \frac{\exp(\sqrt{2\theta}Y')}{\exp(\sqrt{2\theta}Y') + \exp(-\sqrt{2\theta}Y')}$. In addition, the normalizing constant $Z_{Y'}$ of $p_{Y'}(y)$ satisfies that 
\begin{align}\label{ZY:equation}
    \frac{Z_{Y'}}{\sqrt{2\pi}} = \sum_{k = 0}^s \frac{{s \choose k}\exp(\theta (2k - s)^2)}{2^s}.
\end{align}
\end{lemma}

The proof of Lemma~\ref{CWN:lemma} is deferred to the supplementary material. Call the distribution with density $p_{Y'}(y) \varpropto \cosh(\sqrt{2\theta}y)^s \exp(-y^2/2)$ Curie-Weiss normal (CWN) with parameters $s$ and $\theta$. Next we will argue that $n$ samples from the CWN distribution are close in total variation to $n$ samples from the Gaussian distribution with variance $(1 - 2s\theta)^{-1}$ provided that $\theta$ is small. 

\begin{theorem}[CWN is close to Gaussian]\label{thm:CWN:closeto:Gaussian}
Let $Y'$ be a CWN random variable with parameters $s$ and $\theta$, and let $Y \sim N(0,(1 - 2s\theta)^{-1})$ be a Gaussian random variable.  Suppose we have $n$ i.i.d. copies of each of the two random variables. We have that
\begin{align*}
    \operatorname{TV}(\mathcal{L}(\{Y_i\}_{i \in [n]}), \mathcal{L}(\{Y_i'\}_{i \in [n]})) \leq C \sqrt{n \theta^2 \cdot \sum_{i = 1}^5\frac{(s\theta)^i}{(1 - 2s\theta)^{i-1/2}}},
\end{align*}
where $C$ is an absolute constant, provided that $s\theta < c < \frac{1}{2}$ for some sufficiently small constant $c$.
\end{theorem}

The proof of Theorem \ref{thm:CWN:closeto:Gaussian} is given in Section \ref{section:complowerboundproof}. In order the prove this result, we develop novel bounds on the even moments of sums of i.i.d. Rademacher random variables. We are now ready to give the intuition of the proof of the main theorem of this section. Since both the normal clique and the Curie-Weiss model can be generated as conditional i.i.d., where the corresponding conditional variables have close (in total variation) distributions, and the functions $\Phi(x)$ and $\frac{\exp(\sqrt{2/\pi}x)}{\exp(\sqrt{2/\pi}x) + \exp(-\sqrt{2/\pi}x)}$ are close to each other it ought to follow that the total variation between the signs of the normal clique and the Curie-Weiss models are close. The proof of Theorem \ref{thm:Ising:computation}, which is presented in Section \ref{section:complowerboundproof}, makes this intuition precise by calculating a bound on the total variation.

\section{Proof of Theorem~\ref{thm:lowerbound}}\label{section:five}
In this section we give the proof of Theorem~\ref{thm:lowerbound}. Note that by the definition of $\cS^*$, we only need to consider the simple zero-field ferromagnetic Ising model where all non-zero entries in $\Theta$ are the same. Let $G=(\overline{V},E)$ be the underlying graph and $\theta = \theta_{ij},(i,j)\in E$ be the parameter. Let $t=\tanh(\theta)$ and $\EE_0$ denote the expectation under the probability measure that $X_1,\ldots,X_d$ are i.i.d. Rademacher variables. The following lemma gives an equivalent form of the probability mass function in simple zero-field ferromagnetic Ising models.
\begin{lemma}\label{lemma:equivalentprobabilitymassfunction}
	For a simple zero-field Ising model with underlying graph $G=(\overline{V},E)$ and parameter $\theta$, we have
\begin{align}\label{eq:simplifymassfunction}
\PP_\Theta(\bX) = \frac{\prod_{(i,j)\in  E} (1 + t X_iX_j) }{2^d \EE_{0} \big[ \prod_{(i,j)\in  E} (1 + t X_iX_j)\big]},
\end{align}
where $t=\tanh(\theta)$.
\end{lemma}

We now apply Le Cam's method. Let $\PP_{\Theta,n}$ be the joint probability mass function of $n$ i.i.d. samples of Ising model with parameter $\Theta$, $\EE_{\Theta,n}$ denote the expectation under $\PP_{\Theta,n}$, and $\overline{\PP} = \frac{1}{|\cS^*|}\sum_{\Theta\in\cS^*} \PP_{\Theta,n}$ be the averaged probability measure among $\Theta\in \cS^*$ under the alternative. Then the result of Le Cam's method is given in the following lemma.
\begin{lemma}\label{lemma:lecam}
	For the risk $\gamma(\cS^*)$ defined in \eqref{eq:riskSstardef} we have
	$
	\gamma(\cS^*)\geq 1 - \frac{1}{2}\sqrt{D_{\chi^2}(\overline{\PP},\PP_{0 ,n})}
	$, 
	where $D_{\chi^2}(\overline{\PP},\PP_{0 ,n})$ is the $\chi^2$-divergence between $\overline{\PP}$ and $\PP_{0,n}$ defined as
	\begin{align}\label{eq:chisquaredef}
	D_{\chi^2}(\overline{\PP},\PP_{0 ,n}) := \frac{1}{|\cS^*|^2}\sum_{\Theta,\Theta'\in \cS^*} \EE_{0 ,n}\bigg[ \frac{\PP_{\Theta,n}}{\PP_{0 ,n}}\frac{\PP_{\Theta',n}}{\PP_{0 ,n}}  \bigg] -1.
	\end{align}
\end{lemma}

%\fbox{It seems that $D_{\chi^2}$ is not defined.}

Lemma~\ref{lemma:lecam} is a direct result of the Le Cam's method. By Lemma~\ref{lemma:lecam}, $\liminf_{n\rightarrow \infty}\gamma(\cS^*) = 1$ is implied by $\limsup_{n\rightarrow \infty} D_{\chi^2}(\overline{\PP},\PP_{0,n}) = 0$. To prove this, we use a method similar to the high-temperature expansion of Ising model \citep{fisher1967critical,guttman1989asymptotic}. 
By Lemma~\ref{lemma:equivalentprobabilitymassfunction} and the fact that the $n$ samples are independent, for $\Theta,\Theta'\in \cS^*$ with corresponding graphs $G, G'$, we can rewrite the term $\EE_{0 ,n} \big[ \PP_{\Theta,n}  \PP_{\Theta',n} / \PP_{0 ,n}^2 \big]$ as follows:
\begin{align}\label{eq:chisquarecalculation}
\EE_{0 ,n}\bigg[ \frac{\PP_{\Theta,n}}{\PP_{0,n}}\frac{\PP_{\Theta',n}}{\PP_{0,n }} \bigg] = \frac{ \EE_{0 }^n \big[ \prod_{(i,j)\in  E(G)} (1 + t X_iX_j) \cdot \prod_{(i,j)\in  E(G')} (1 + t X_iX_j) \big] }{\EE_{0 }^n \big[ \prod_{(i,j)\in  E(G)} (1 + t X_iX_j)\big]\cdot \EE_{0 }^n \big[ \prod_{(i,j)\in  E(G')} (1 + t X_iX_j)\big]},
\end{align}
where $t = \tanh(\theta)$. 
Each expectation on the right-hand side above is a polynomial of $t$. For any $G,G'\in \cG^*$, we define 
\begin{align*}
f_{G}(t) & = \EE_{0 }  \Bigg[ \prod_{(i,j)\in  E(G)} (1 + t X_iX_j) \Bigg], \\ 
f_{ G, G'}(t) & = \EE_{0 }  \Bigg[ \prod_{(i,j)\in  E(G)} (1 + t X_iX_j) \prod_{(i,j)\in  E(G')} (1 + t X_iX_j) \Bigg].
\end{align*}
Plugging the definitions above into \eqref{eq:chisquarecalculation}, we obtain
\begin{align}\label{eq:chisquarepolyrepresentation}
\EE_{0 ,n}\bigg[ \frac{\PP_{\Theta,n}}{\PP_{0 ,n}}\frac{\PP_{\Theta',n}}{\PP_{0 ,n}} \bigg] = \Bigg[1 + \frac{f_{ G, G'}(t) - f_{ G}(t) f_{ G'}(t)}{f_{ G}(t)f_{ G'}(t)}\Bigg]^n.
\end{align}
We now analyze the coefficients of each polynomial in \eqref{eq:chisquarepolyrepresentation}. Let
\begin{align*}
f_{ G}(t) = \sum_{k=0}^\infty a_k t^k,~ f_{ G'}(t) = \sum_{k=0}^\infty b_k t^k,~ f_{G, G'}(t) = \sum_{k=0}^\infty c_k t^k.
\end{align*}
We also define
\begin{align*}
f_{G, G'}(t) - f_{G}(t)f_{G'}(t) = \sum_{k=0}^\infty \Bigg(c_k - \sum_{k_1+k_2 = k}a_{k_1}b_{k_2}\Bigg) t^k= \sum_{k=0}^\infty u_k t^k.
\end{align*}
For $f_{ G}(t)$, 
note that after expanding $\prod_{(i,j)\in  E(G)} (1 + t X_iX_j)$, the terms with non-zero expectations must have the form $ t^k X_{i_1}^2\cdots X_{i_k}^2$, where $i_1,\ldots, i_k \in \overline{V}$. Therefore by Lemma~\ref{lemma:Euler}, the coefficient of $t^k$ is equal to the number of $k$-edge subgraphs of $G$ where every vertex has an even degree. Similar arguments also applies to $f_{ G'}(t)$ and $f_{ G, G'}(t)$. This observation motivates us to introduce the definitions of multigraphs and Eulerian graphs. 

\begin{definition}[Multigraph]
	A multigraph is a graph which is permitted to have multiple edges connecting two vertices. We denote $G=(V,E)$, where $V$ is the vertex set, and $E$ is the edge multiset.
	
	For a multigraph $G$ with $d$ vertices, we define its adjacency matrix to be $\rmA=(\rmA_{ij})_{d\times d}$, where $\rmA_{ij} = \rmA_{ji}= $``the number of edges connecting vertices $i$ and $j$''. 
	A symmetric matrix $\rmA\in \RR^{d\times d}$ with nonnegative integer off-diagonal entries and zero diagonal entries naturally represents a multigraph with vertex set $\overline{V}$. 
	Given two multigraphs $G$ and $G'$ with adjacency matrices $\rmA$ and $\rmA'$, we define $G\oplus G'$ to be the multigraph defined by $\rmA + \rmA'$.
\end{definition}

\begin{definition}[Eulerian graph]
	An Eulerian circuit on a multigraph is a closed walk that uses each edge exactly once. 
	We say that a multigraph is Eulerian if every connected component has an Eulerian circuit.
\end{definition}

Note that in graph theory, the term 'Eulerian graph' has different meanings. Sometimes Eulerian graph is referred to as a graph that has an Eulerian circuit. This is different from our definition, because in this paper we do not require an Eulerian graph to be connected. 
The following famous lemma on Eulerian graph is first given by \cite{euler1741solutio} and then completely proved by \cite{hierholzer1873moglichkeit}.
\begin{lemma}\label{lemma:Euler}
	A graph is Eulerian if and only if all vertices in the graph have an even degree.
\end{lemma} 
Based on our previous discussion, Lemma~\ref{lemma:Euler} relates $a_k$, $b_k$ and $c_k$ to the number of $k$-edge Eulerian graphs. Define
\begin{align*}
\cE(k,G) := \big\{ \tilde{G} = (\overline{V},\tilde{E}): \tilde{E}\subseteq E,~|\tilde{E}| = k,~\tilde{G}\text{ is an Eulerian graph}\big\}.
\end{align*}
In words $\cE(k,G)$ is the set of $k$-edge Eulerian subgraphs of $G$. By Lemma~\ref{lemma:Euler} and our previous discussion, we have $a_k =|\cE(k,G)|$, $b_k =|\cE(k,G')|$ and $c_k =|\cE(k,G\oplus G')|$, and therefore
\begin{align}\label{eq:fE}
f_{ G}(t) = \sum_{k\geq 0} |\cE(k,G)| t^k,~ f_{ G'}(t) = \sum_{k\geq 0} |\cE(k,G')| t^k,\\
\text{ and } f_{ G, G'}(t) = \sum_{k\geq 0} |\cE(k,G\oplus G')| t^k.\nonumber
\end{align}
Figure~\ref{fig:EkG-example} gives an example of how to calculate $|\cE(k,G)|$ for a given multigraph $G$. 
\begin{figure}%[t]
	\begin{center}
		\begin{tabular}{cccc}
			\includegraphics[width=.18\textwidth,angle=0]{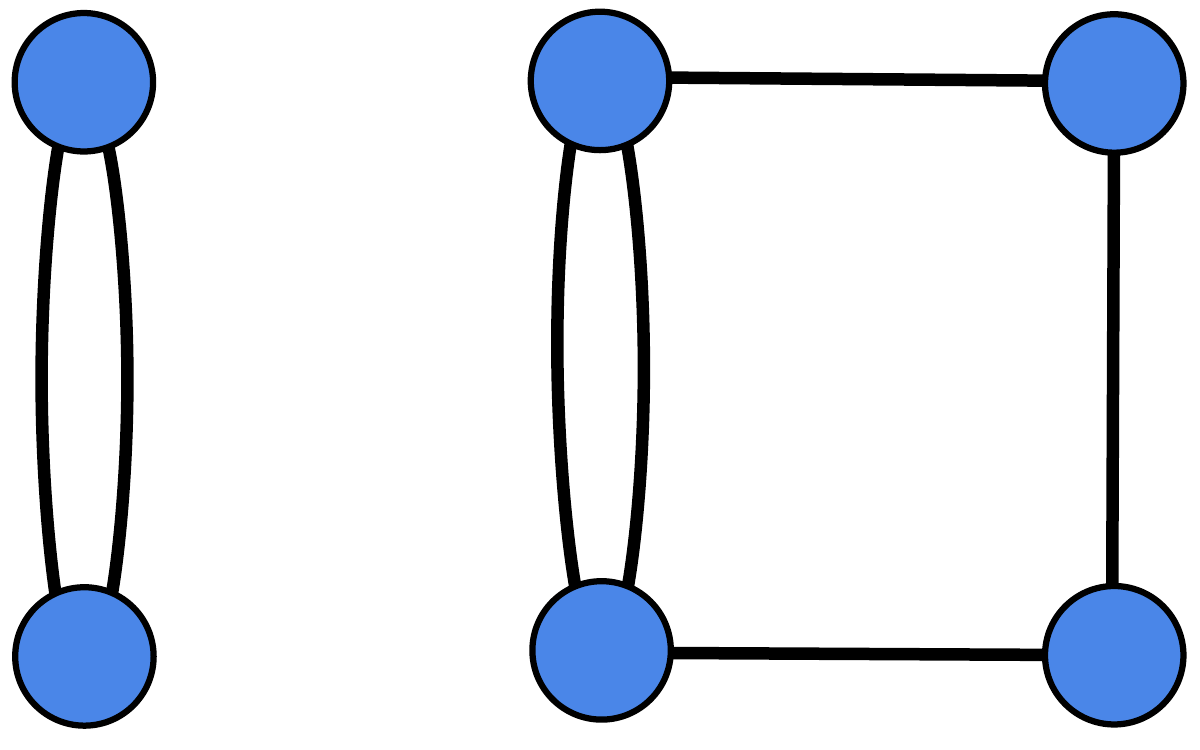}
			&
			\includegraphics[width=.18\textwidth,angle=0]{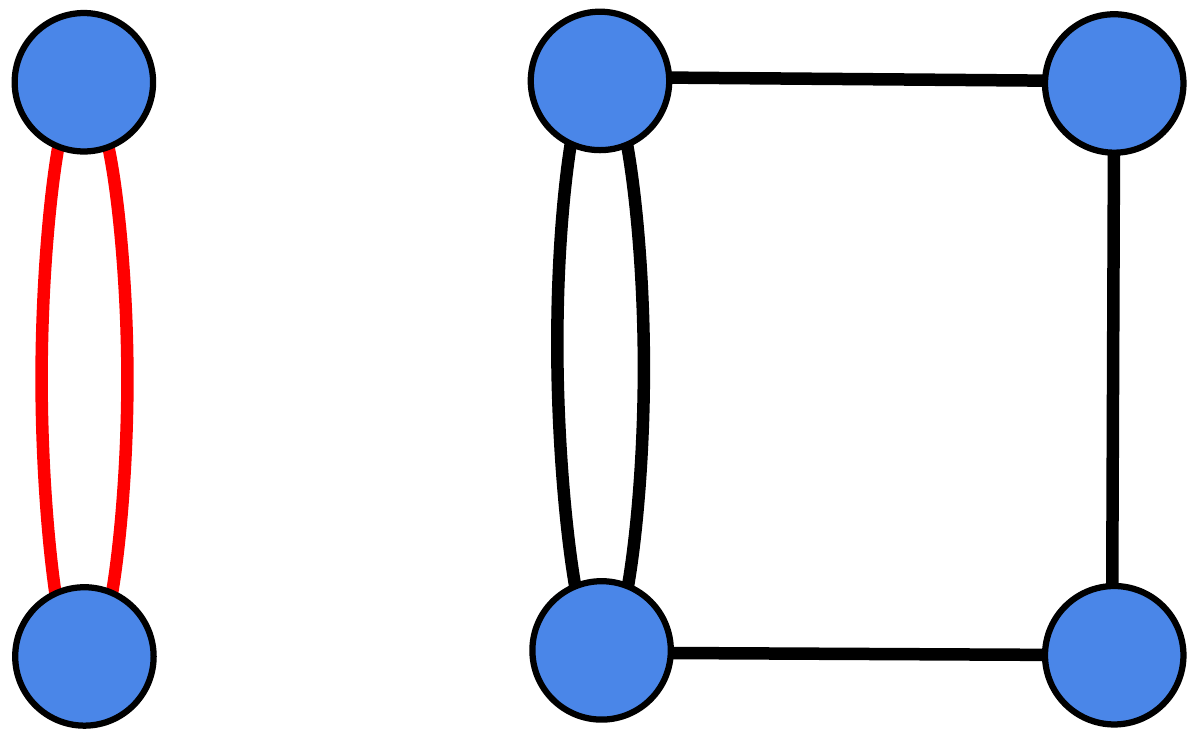}
			&
			\includegraphics[width=.18\textwidth,angle=0]{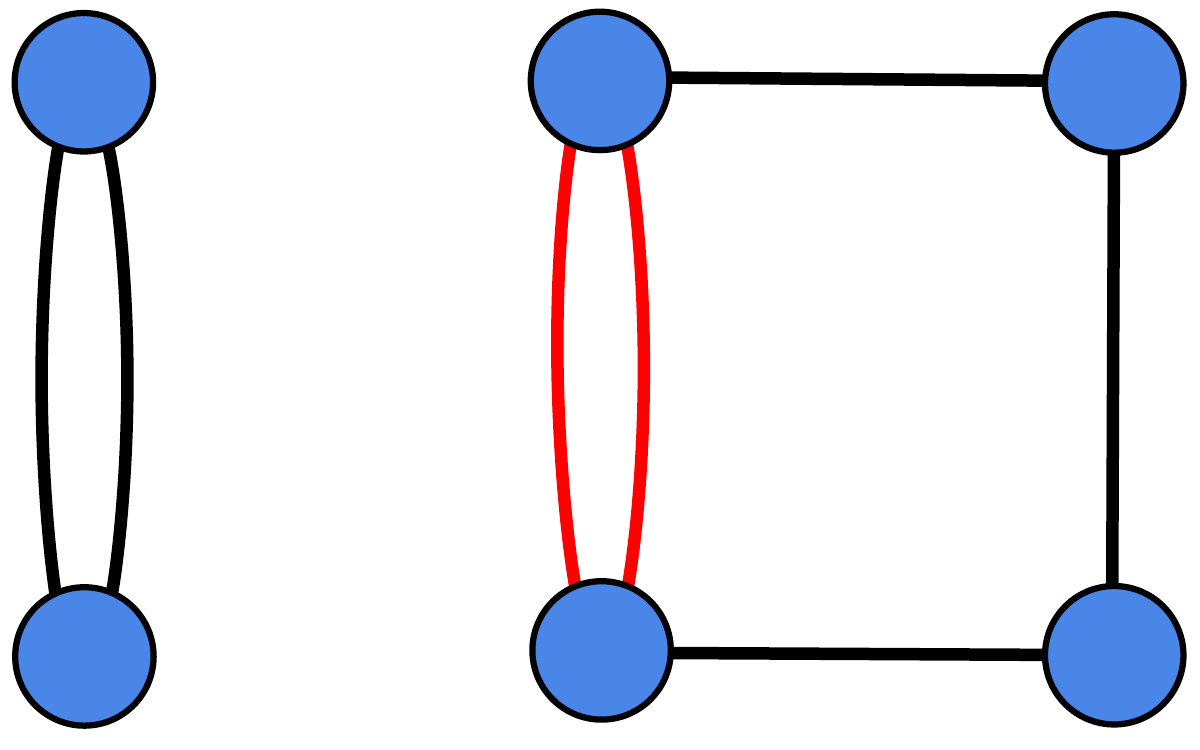}
			&
			\includegraphics[width=.18\textwidth,angle=0]{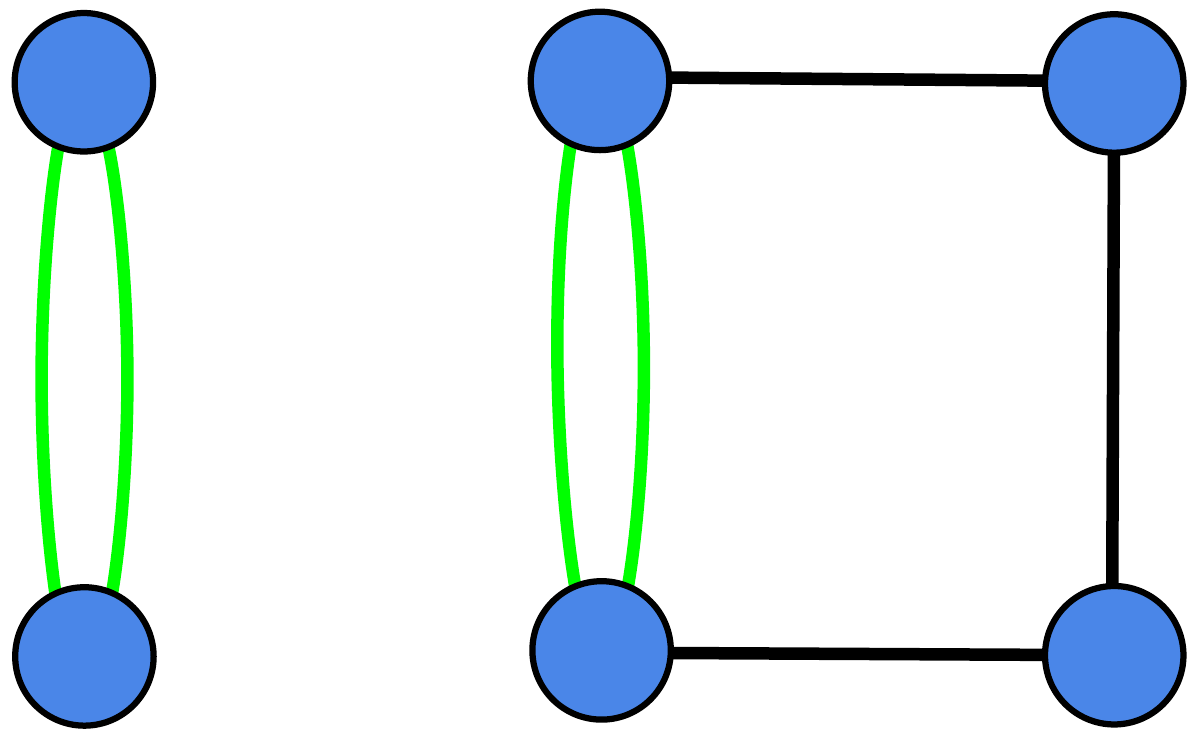}\\
			(a) & (b)  & (c) & (d) 
			\\
			\includegraphics[width=.18\textwidth,angle=0]{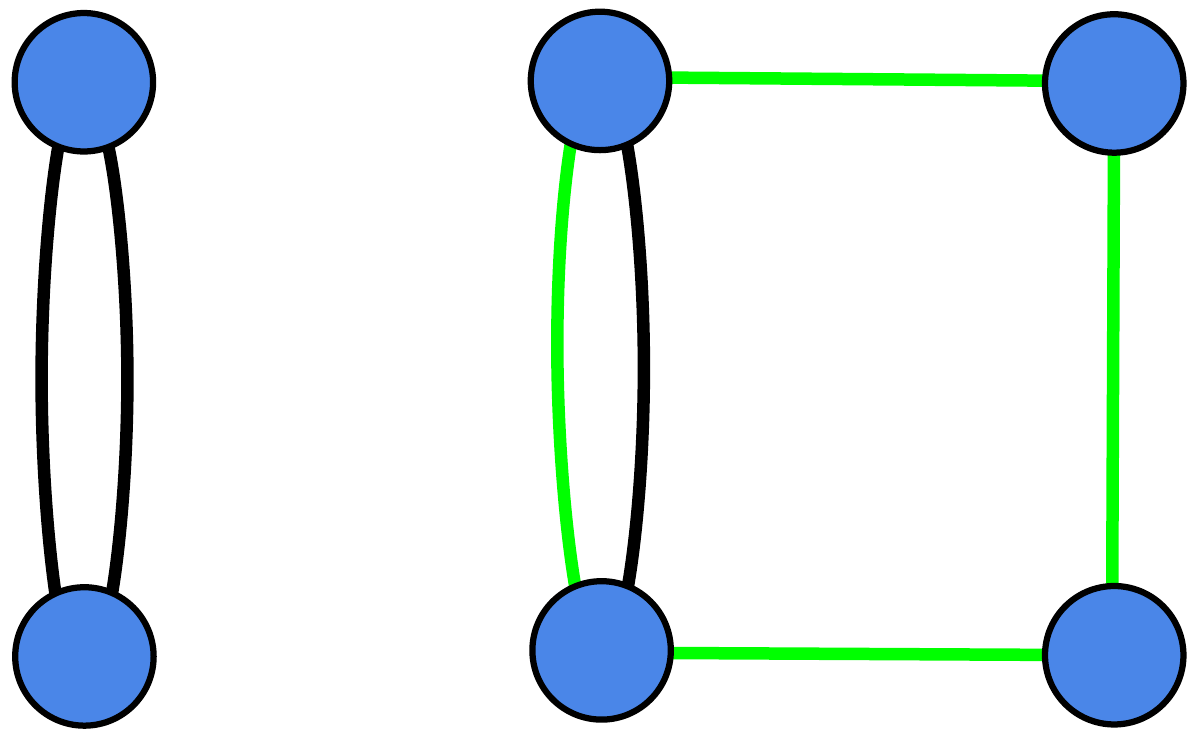}
			&
			\includegraphics[width=.18\textwidth,angle=0]{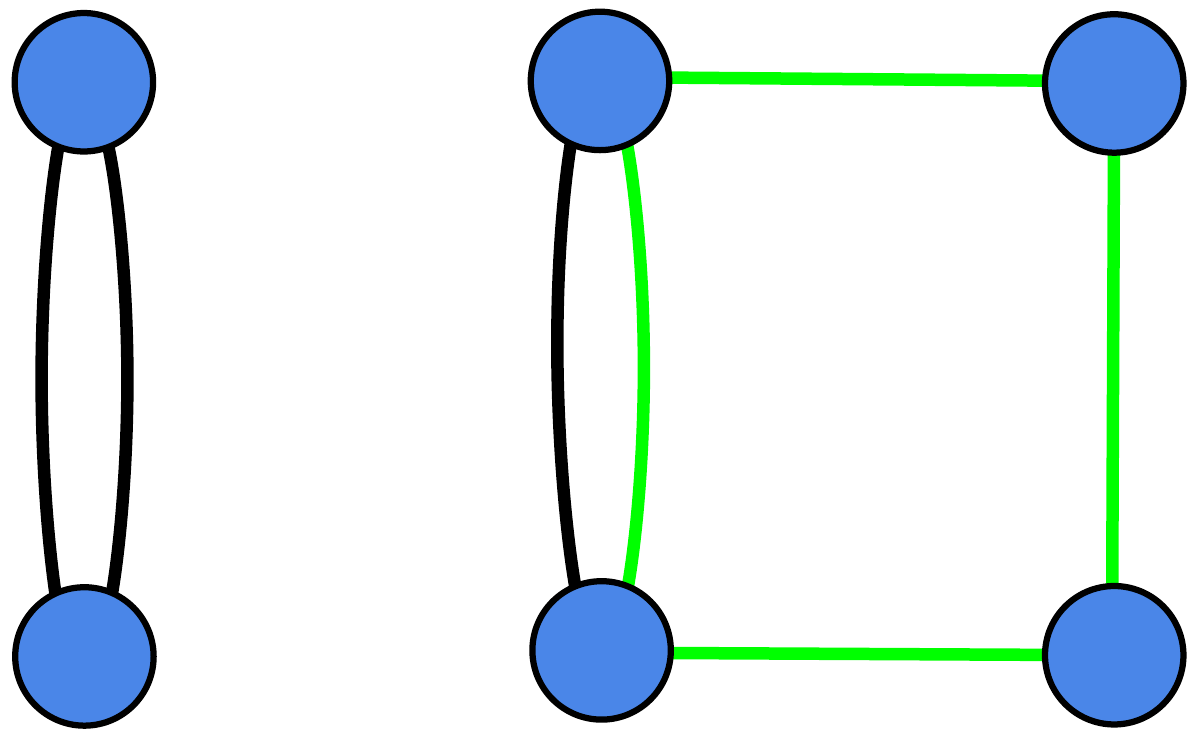}
			&
			\includegraphics[width=.18\textwidth,angle=0]{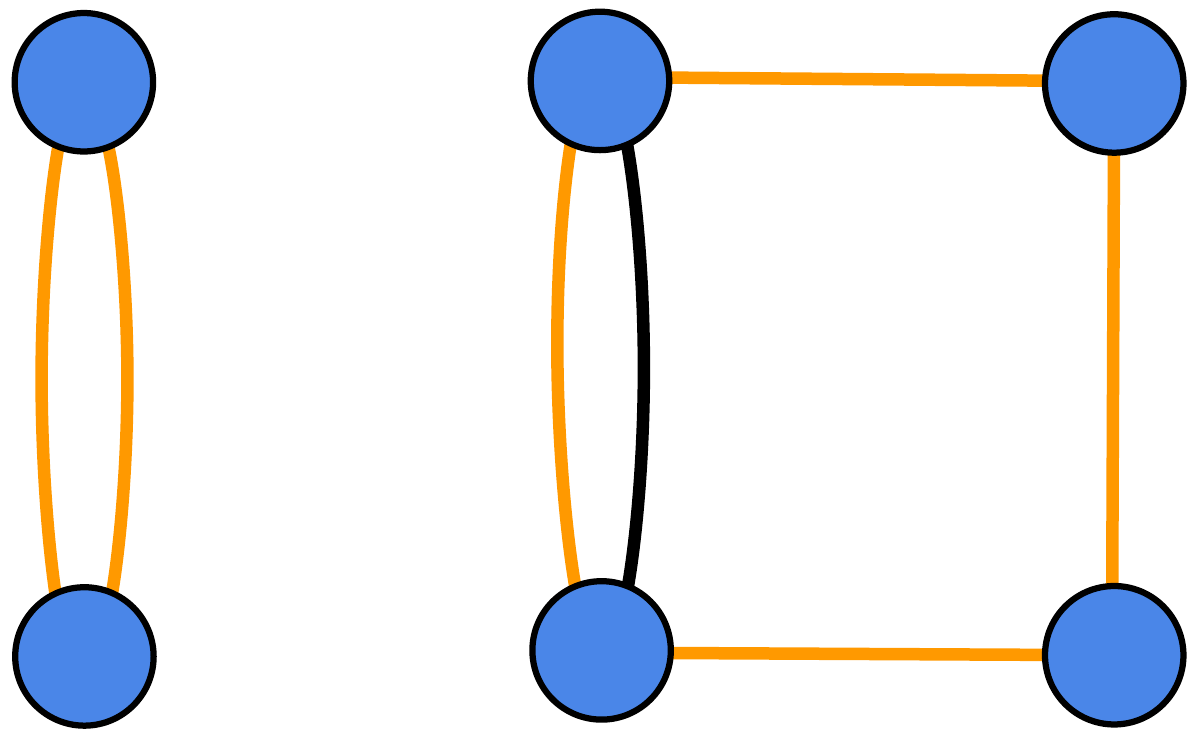}
			&
			\includegraphics[width=.18\textwidth,angle=0]{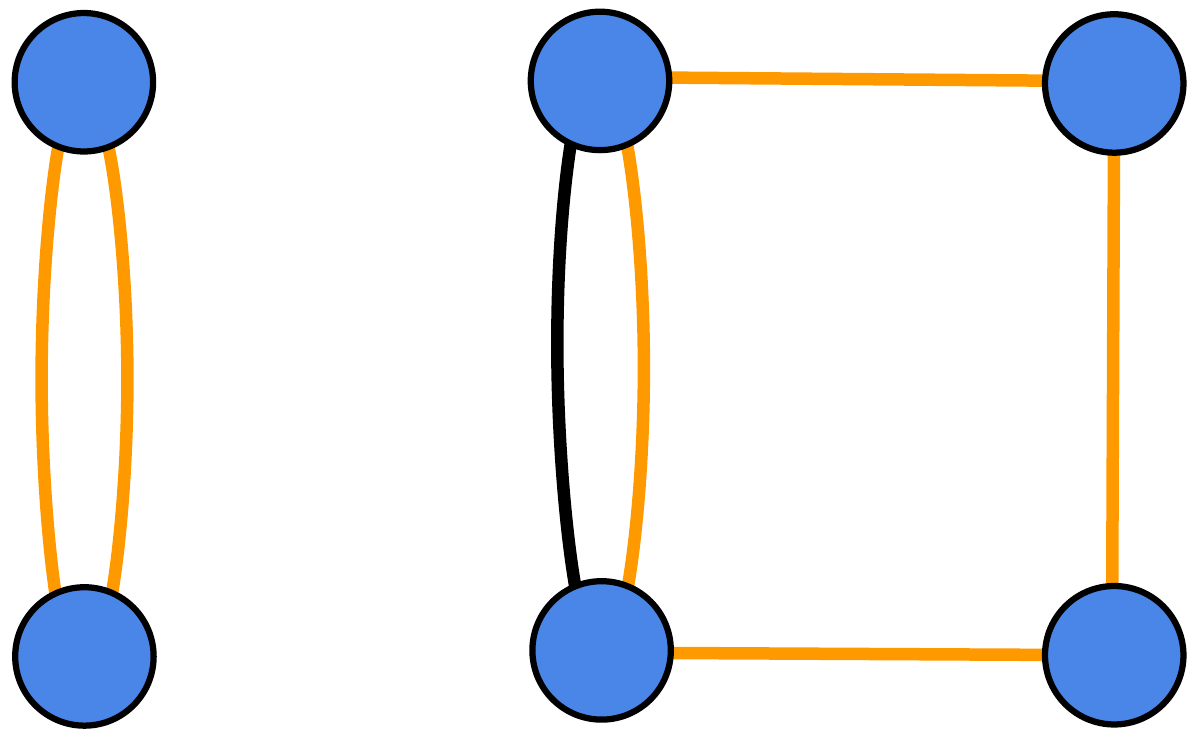}
			\\	
		    (e) & (f) & (g) & (h) 
		    \end{tabular}
	\end{center}
	\vskip-10pt
	\caption{An example of the calculation of $\{|\cE(k,G)|\}_{k\geq 1}$ for a multigraph $G$ is given in (a). We use red, green and orange edges to highlight $2$-edge, $4$-edge and $6$-edge Eulerian subgraphs of $G$ respectively. (b), (c) give the $2$-edge Eulerian subgraphs; (d), (e), (f) give the $4$-edge Eulerian subgraphs; and (g), (h) give the $6$-edge Eulerian subgraphs. We have $|\cE(2,G)| = 2$, $|\cE(4,G)| = 3$, $|\cE(6,G)| = 2$, and $|\cE(k,G)|=0$ for $k\neq 2,4,6$.}
	\label{fig:EkG-example}
	\vspace{-10pt}
\end{figure}
We now proceed to analyze $u_k$. Apparently, $u_0 = u_1 = 0$. For $k\geq 2$, by the definition of $u_k$ we can see that, if a $k$-edge Eulerian subgraph of $G\oplus G'$ can be split into two graphs $G_1$ and $G_2$ such that $G_1$ and $G_2$ are Eulerian subgraphs of $G$ and $G'$ respectively, then it is also counted in the sum $\sum_{k_1+k_2 = k}a_{k_1}b_{k_2}$ and therefore is not counted in $u_k$. Figure~\ref{fig:countednotcounted} gives examples of Eulerian subgraphs that are counted and not counted in $u_k$.
\begin{figure}%[t]
	\begin{center}
		\begin{tabular}{ccccccc}
			\includegraphics[width=.4\textwidth,angle=0]{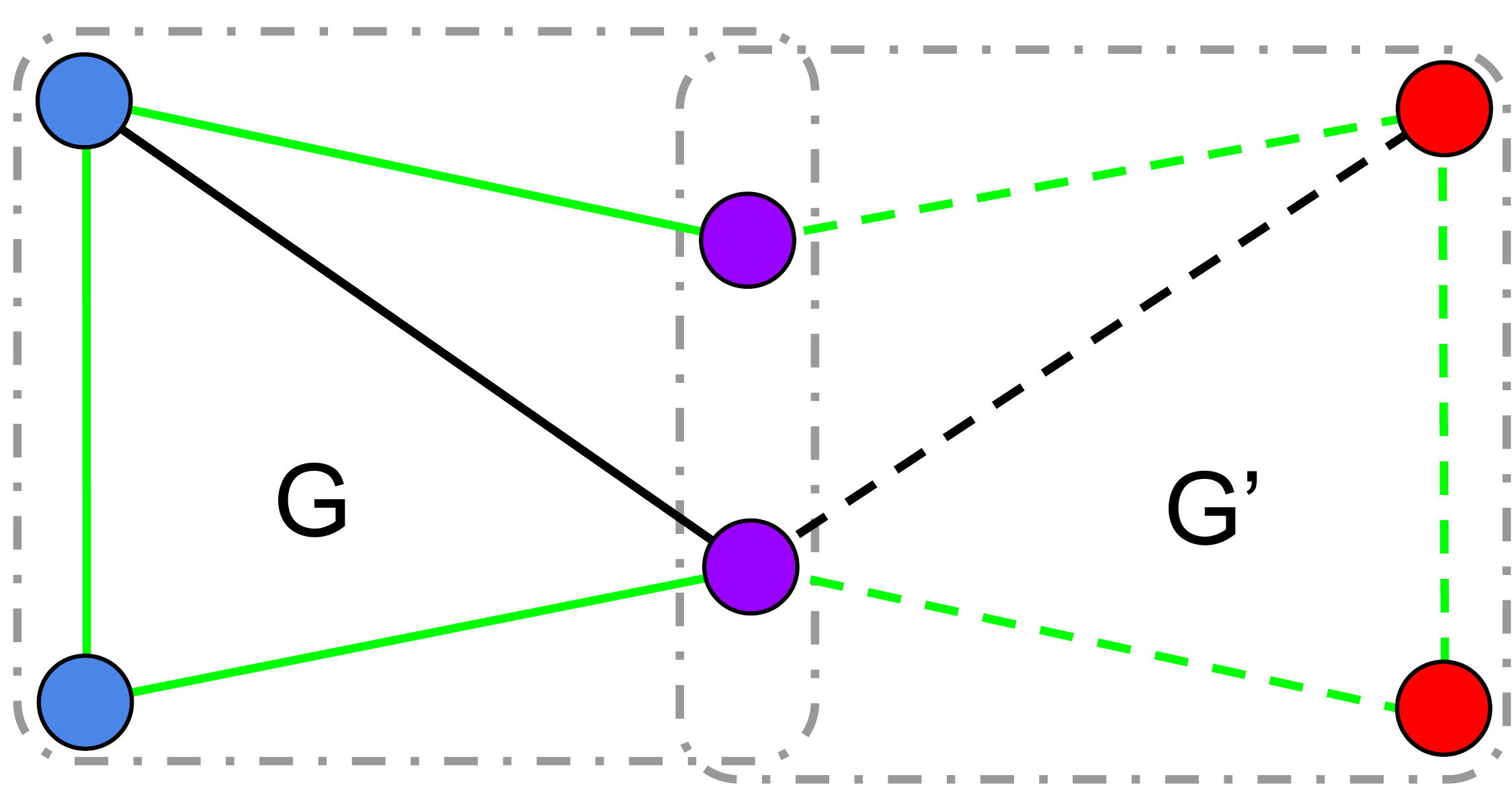}
			&   &
			\includegraphics[width=.4\textwidth,angle=0]{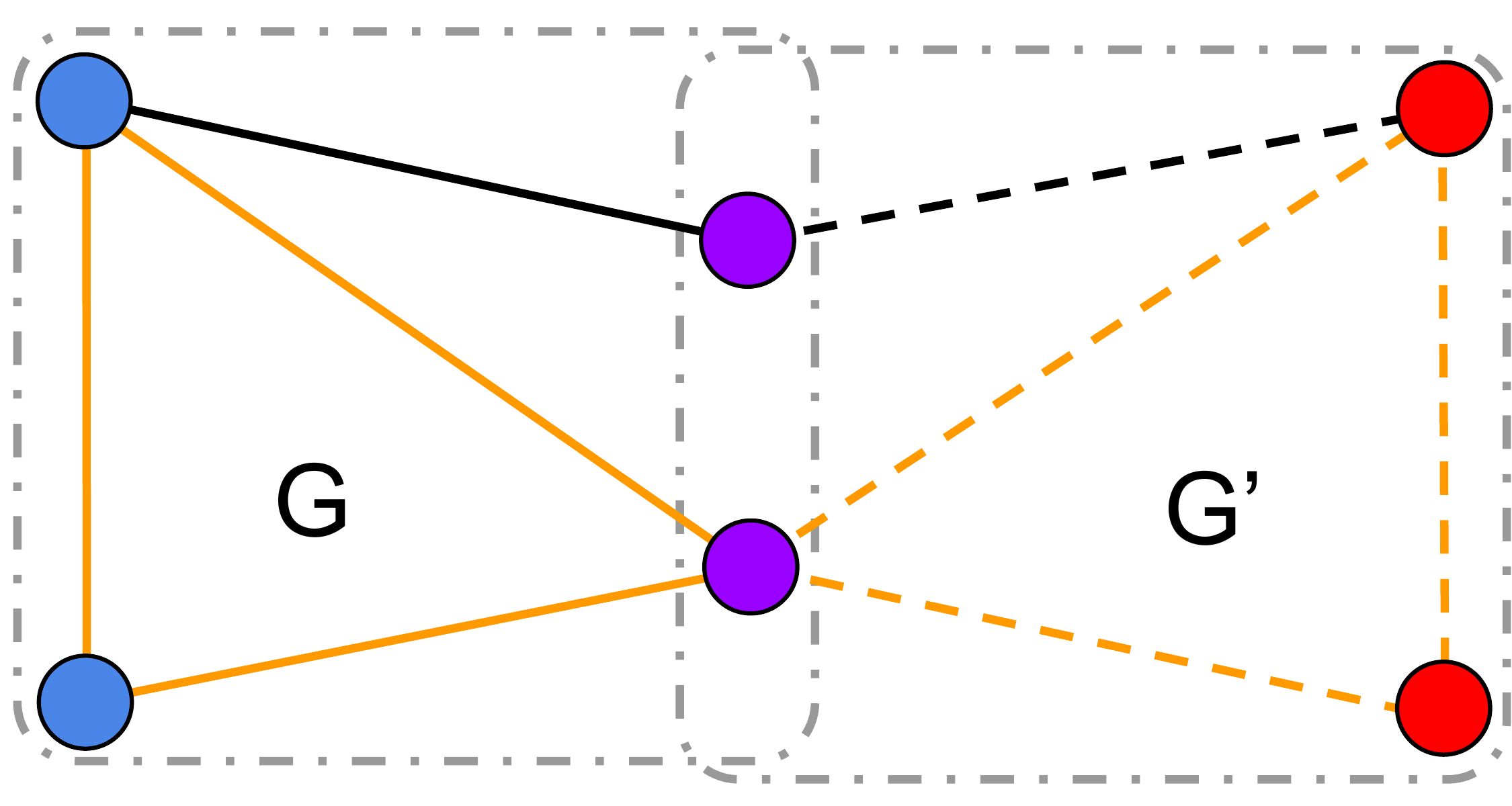}\\
			(a)  & & (b) 
		\end{tabular}
	\end{center}
	\vskip-10pt
	\caption{Illustration of graphs counted and not counted in $u_k$. The gray dot-dashed squares highlight the non-isolated vertices of $G$ and $G'$. The solid and dashed lines are edges in $G$ and $G'$ respectively.  We use purple vertices to represent the common non-isolated vertices of $G$ and $G'$. The blue vertices are non-isolated in $G$ but isolated in $G'$, and the red vertices are non-isolated in $G'$ but isolated in $G$. The green edges in (a) give an example of a $6$-edge Eulerian subgraph of $G\oplus G'$ counted in $u_6$, while the orange edges in (b) form a $6$-edge Eulerian subgraph of $G\oplus G'$ that is not counted in $u_6$.}
	\label{fig:countednotcounted}
	\vspace{-10pt}
\end{figure}
Using this type of argument, the following two lemmas together calculate and bound $u_k$ for $k\geq 2$.
\begin{lemma}\label{boundpolydiff1}
	We have
	\begin{align*}
	&u_2 = | E(G)\cap  E(G')|,~u_3 = \Delta_{G,G'}, \text{ and}\\
	&u_k\leq q_k[G\oplus G',V(G)\cap V(G')] \text{ for } k\geq 4,
	\end{align*}
	where $\Delta_{G,G'}$ denotes the number of triangles that contains at least one edge from $G$ and one edge from $G'$, and the function $q_k(\cdot,\cdot)$ is defined as follows:
	\begin{align*}
	\MoveEqLeft q_k(G,V) :=\big|\big\{\tilde{G}\in \cE(k,G): \exists i,j\in V,i,j\text{ are in one connected component of $\tilde{G}$} \big\}\big|.
	\end{align*}
\end{lemma}
\begin{lemma}\label{boundpolydiff2}
	We have
	\begin{align*}
	|\cE(k,G)| \leq 2^k\|\rmA_G\|_F^k, ~\Delta_{G,G'} \leq 2 |V(G)\cap V(G')|\cdot \cR \cdot \Gamma.
	\end{align*}
	Moreover, for any multigraph $G$ and vertex set $V\subseteq \overline{V}$, we have
	\begin{align*}
	q_k(G,V)\leq \big(2^k \cdot |V|\cdot \|\rmA_{G}\|_F^k\big) \land \big[ k\cdot 2^{k-2}\cdot |V|^2\cdot(\|\rmA_{G}\|_1 \lor \|\rmA_{G}\|_F)^{k-2}\big].
	\end{align*}
\end{lemma}
The upper bound of $|\cE(k,G)|$ in Lemma~\ref{boundpolydiff2} and the assumption that $\theta\leq [8(\Lambda\lor \Gamma)]^{-1}$ together show that $f_G(t)$, $f_{G'}(t)$ and $f_{G,G'}(t)$, as power series, all converge. 
Moreover, by the definition of $\cB$, the upper bound for $q_k(\cdot,\cdot)$ in Lemma~\ref{boundpolydiff2} and the assumption that $\theta\leq [8(\Lambda\lor \Gamma)]^{-1}$, we have
\begin{align*}
\sum _{k\geq 4} q_k[G\oplus G',V(G)\cap V(G')] \theta^k \leq |V(G)\cap V(G')| \cdot \cB\theta ^4.
\end{align*}
By Lemmas~\ref{boundpolydiff1} and \ref{boundpolydiff2} and the fact that $t = \tanh(\theta)\leq \theta$, we have
\begin{align}
f_{ G, G'}(t) - f_{ G}(t) f_{ G'}(t) &\leq | E(G)\cap  E(G)|\theta^2 + \Delta_{G,G'} \theta^3 + |V(G)\cap V(G')|  \cB\theta ^4\nonumber \\
&\leq  |V(G)\cap V(G')| \cdot ( \cR + 2\cR\cdot \Gamma\theta + \cB \theta^2) \cdot \theta^2.\label{eq:coefficientbound}
\end{align}
Note that $f_{ G}(t)f_{ G'}(t) \geq 1$ for $t \geq 0$ since all coefficients $a_k$ and $b_k$ are non-negative. By
\eqref{eq:chisquarepolyrepresentation}, \eqref{eq:coefficientbound} and the assumption that $\theta \leq  [8(\Lambda\lor \Gamma)]^{-1} \leq (2\Gamma)^{-1}$ and $\theta \leq \sqrt{\cR/\cB}$, we have
\begin{align*}
\EE_{0 }\bigg[ \frac{\PP_{\Theta}}{\PP_{0 }}\frac{\PP_{\Theta'}}{\PP_{0 }} \bigg] \leq 1+ 3 |V(G)\cap V(G')| \cR \theta^2.
\end{align*}
Plugging the inequality above into the definition of $\chi^2$-divergence \eqref{eq:chisquaredef} gives
\begin{align}\label{eq:chisquarebound}
D_{\chi^2}(\overline{\PP},\PP_{0 ,n}) \leq  \frac{1}{|\cS^*|^2}\sum_{\Theta,\Theta'\in \cS^*} \big[ 1 + 3|V(G)\cap V(G')| \cdot \cR\theta^2 \big]^n - 1 .
\end{align}
To complete the proof, we invoke the incoherence condition of $\cG^*$. We summarize the result as the following lemma. 
\begin{lemma}\label{lemma:negativeassociation}
	If $\cG^*$ is incoherent, then the following inequality holds.
\begin{align*}
\frac{1}{|\cS^*|^2}\sum_{\Theta,\Theta'\in \cS^*} \exp[ 3n\cR|V(G)\cap V(G')| \theta^2 \big] \leq \exp [ N(\cG^*)\cdot \exp ( 3n\cR \theta^2 )  ].
\end{align*}
\end{lemma}
Now by \eqref{eq:chisquarebound}, Lemma~\ref{lemma:lecam} and Lemma~\ref{lemma:negativeassociation}, if $\theta \leq \sqrt{\frac{\log [N^{-1}(\cG^*)]^{-1}\}}{6nR}} $ and $N(\cG^*) = o(1)$, we have 
\begin{align*}
\liminf_{n\rightarrow \infty} \gamma(\cS^*) = 1.
\end{align*}

\section{Proof of Theorem~\ref{thm:upperbound}}\label{section:six}
	In this section we give the proof of Theorem~\ref{thm:upperbound}. The key part of our proof is to derive concentration inequalities for $W_H$. Following the definition in \cite{vershynin2010introduction}, we define the $\psi_1$-norm of the random variable $Z$ as follows.
	\begin{align*}
	    \| Z \|_{\psi_1} := \sup_{p\geq 1} p^{-1} (\EE |Z|^p)^{1/p}. 
	\end{align*}
	If a random variable $Z$ has finite $\psi_1$-norm, we say $Z$ is a sub-exponential random variable. 
	The following lemma gives bounds for the $\psi_1$-norm of $W_H$.
	\begin{lemma}\label{lemma:psi1normbound}
		Let $\bX \in \{\pm 1\}^d$ be a random vector generated from the high temperature ferromagnetic Ising model with parameter matrix $\Theta$. For any graph $H$, define 
		\begin{align*}
		W_{H} := \frac{1}{|E(H)|} \sum_{(i,j)\in E(H)} X_iX_j.
		\end{align*}
		If $\|\Theta\|_F \leq 1/2$, then we have $\|W_H\|_{\psi_1} \leq C|E(H)|^{-1/2}$, where $C>0$ is an absolute constant.
		\end{lemma}
	We first prove that $\PP_{0,n}(\psi = 1) < \alpha / 2$. Under the null, $X_1,\ldots, X_d$ are independent Rademacher random variables. Therefore for every $H\in \cH$ we have $\EE_{0,n} \hat{W}_H = 0$. By Lemma~\ref{lemma:psi1normbound} with $\Theta = (0)_{d\times d}$, we have $\|W_H\|_{\psi_1} \leq C_1|E(H)|^{-1/2}$ for an absolute constant $C_1>0$. By Proposition~5.16 in \cite{vershynin2010introduction}, for $ \varepsilon\leq |E(H)|^{-1/2}$ we have
	\begin{align*}
	\PP_{0,n}\Big(\Big|\hat{W}_H - \EE_{0,n} \hat{W}_H \Big| > \varepsilon\Big) \leq 2\exp ( -C_2\cdot |E(H)|\cdot n\varepsilon^2),
	\end{align*}
	where $C_2$ is an absolute constant. 
	Setting the right-hand side above to be $\alpha/(2|\cH|)$ and solving for $\varepsilon$ shows that under the null hypothesis, with probability at least $1 - \alpha/(2|\cH|)$, we have
	\begin{align*}
	\hat{W}_H \leq C_3 \sqrt{\frac{\log|\cH| + \log(2/\alpha)}{|E(H)| n}} \leq C_3 \sqrt{\frac{2 \log|\cH|}{|E(H)| n}}, \end{align*}
	for absolute constant $C_3$. Note that the condition $\varepsilon\leq |E(H)|^{-1/2}$ is satisfied since we assume that $\log (|\cH|)/n = o(1)$. By definition, we have $m(\cH) \leq |V(H)|$. Moreover, we have
	\begin{align}\label{eq:EoverVbound}
	\cR \leq \frac{|E(H)|}{|V(H)|-1} + 1 \leq \frac{2|E(H)|}{|V(H)|} + 1 \leq \frac{4|E(H)|}{|V(H)|},
	\end{align}
	where the last inequality follows by $2|E(H)| \geq |V(H)|$.
	Therefore with probability at least $1 - \alpha/(2|\cH|)$,
	\begin{align*}
	\hat{W}_H \leq C_4 \sqrt{\frac{\log|\cH|}{m(\cH)}\cdot \frac{|V(H)|}{|E(H)|}\cdot \frac{1}{n}}\leq C_5 \sqrt{ \frac{M(\cH)}{\cR n} },
	\end{align*}
	where $C_4$, $C_5$ are absolute constants. Therefore by union bound, when $\kappa$ is chosen to be a large enough constant we have $\PP_{0,n}(\psi = 1) \leq \alpha / 2$.
	
	For any $\Theta\in \cS(\cG_1,\theta)$ with corresponding graph $G$, we now prove that $\PP_{\Theta,n}(\psi = 0) < \alpha / 2$. By the definition of witnessing set and \eqref{eq:EoverVbound}, there exists $H\in \cH$ which is a subgraph of $G$ and we have $|E(H)|/|V(H)| \geq \cR/4$. It now suffices to prove that 
	$$
	\PP_{\Theta,n}\Bigg(\hat{W}_{H} \leq \frac{\kappa}{4}\sqrt{\frac{M(\cH)}{\cR n}}\Bigg) \leq \frac{\alpha}{2}.
	$$
	Since $H$ is a subgraph of $G$, each edge in $E(H)$ is also an edge in $E(G)$. By the second Griffiths inequality (see \cite{griffiths1967correlations,kelly1968general}), for $\theta \leq 1$ we have
	\begin{align*}
	\EE_{\Theta,n} \hat{W}_{H} \geq \tanh(\theta)\geq \theta/2.
	\end{align*}
	Applying Lemma~\ref{lemma:psi1normbound} gives $\|W_H\|_{\psi_1} \leq C_6|E(H)|^{-1/2}$ for an absolute constant $C_6$. By Proposition~5.16 in \cite{vershynin2010introduction}, for $ \eta\leq |E(H)|^{-1/2}$ we have
	\begin{align*}
	\PP_{\Theta,n}\Big(\Big|\hat{W}_H - \EE_{\Theta,n} \hat{W}_H \Big| > \eta\Big) \leq 2\exp ( -C_7\cdot |E(H)|\cdot n\eta^2),
	\end{align*}
	for an absolute constant $C_7$. Therefore with probability at least $1-\alpha/2$, we have
	\begin{align*}
	\hat{W}_H &\geq \EE_{\Theta}\hat{W}_H - C_8 \sqrt{\frac{ \log(2/\alpha)}{|E(H)| n}} \geq \frac{\theta}{2} - C_9\sqrt{\frac{ |V(H)| }{m(\cH)} \cdot \frac{ \log(|\cH|) }{ |E(H)| n}} > \bigg(\frac{\kappa}{2} - C_{10}\bigg)\sqrt{\frac{ M(\cH)}{\cR n}},
	\end{align*}
	where $C_8$, $C_9$ and $C_{10}$ are absolute constants. Therefore when $\kappa$ is chosen as a large enough constant we have $\PP_{\Theta,n}(\psi = 0) \leq \alpha / 2$, and
	\begin{align*}
	\PP_{0,n}(\psi = 1) + \sup_{\Theta\in \cS(\cG_1,\theta)} \PP_{\Theta,n}(\psi = 0) \leq \alpha.
	\end{align*}
	This completes the proof.

\section{Proof of Theorems in Section~\ref{section:complowerbound}}\label{section:complowerboundproof}

\begin{proof}[Proof of Theorem~\ref{thm:CWN:closeto:Gaussian}]
We consider the normal distribution $N(0, (1- 2s\theta)^{-1})$ which has density
\begin{align*}
    \bigg(\frac{1-2s\theta}{2\pi}\bigg)^{1/2} \exp(s\theta y^2) \exp(-y^2/2).
\end{align*}
Denote with 
\begin{align*}
\int_{-\infty}^{\infty} \cosh(\sqrt{2\theta}y)^s \frac{\exp(-y^2/2)}{\sqrt{2\pi}} dy = \frac{Z_{Y'}}{\sqrt{2\pi}} =: C(\theta, s),
\end{align*}
where $Z_{Y'}$ is the notation used in Lemma \ref{CWN:lemma}.

The proof begins by using Pinsker's inequality to bound the total variation with the square root of the KL divergence. We need to control

\begin{align*}
& n\int_{-\infty}^\infty \log \frac{(\frac{1-2s\theta}{2\pi})^{1/2} \exp(s\theta y^2) \exp(-y^2/2)}{\cosh(\sqrt{2\theta}y)^s \frac{\exp(-y^2/2)}{C(\theta,s)\sqrt{2\pi}}}\bigg(\frac{1-2s\theta}{2\pi}\bigg)^{1/2} \exp(s\theta y^2) \exp(-y^2/2)dy\\
& = n \bigg(\log(C(\theta,s) \sqrt{1 - 2s\theta}) + \int_{-\infty}^\infty (s\theta y^2 - s \log \cosh(\sqrt{2\theta}y)) \bigg(\frac{1-2s\theta}{2\pi}\bigg)^{1/2} \exp(s\theta y^2) \exp(-y^2/2)dy\bigg)
\end{align*}
One can check that $x^2/2 - \log(\cosh(x)) - x^4/12 \leq 0$ (we give a proof in the supplement), therefore the above is bounded as
\begin{align*}
& n \bigg(\log(C(\theta,s) \sqrt{1 - 2s\theta}) + \int_{-\infty}^\infty s \frac{4\theta^2 y^4}{12} \bigg(\frac{1-2s\theta}{2\pi}\bigg)^{1/2} \exp(s\theta y^2) \exp(-y^2/2)dy\bigg)\\
& = n (\log(C(\theta,s) \sqrt{1 - 2s\theta}) + s \theta^2 (1-2s\theta)^{-2}).
\end{align*}
We will now bound $C(\theta,s)$ from above. By Lemma \ref{CWN:lemma} it follows that:
\begin{align*}
C(\theta, s) = \frac{Z_{Y'}}{\sqrt{2\pi}} = \sum_{k = 0}^s \frac{{s \choose k} \exp(\theta(2k-s)^2)}{2^s} =\sum_{m = 0}^{\infty} \frac{\theta^m \EE (2k - s)^{2m}}{m!},
\end{align*}
where $k \sim \mathrm{Bin}(s,.5)$. 

%Now clearly $\EE (2k - s)^2 = s$. In addition one can show that $\EE (2k - s)^4 = 3s^2 - 2s$.  It remains to show that 
%\begin{align*}
%\EE (2k-s)^{2m} \leq (2m)!! s^{m},
%\end{align*}
%for all $m$. (the above come from the moments of the approximately normal distribution $\frac{2k - s}{\sqrt{s}}$). 

%I believe the above follows along the following lines. Take $\EE \exp(t (2k - s)) = \cosh(t)^s$. This generates all centered moments of the binomial. Now notice that all terms in the Taylor expansion of this are bounded by the terms in $\exp(st^2/2) = (\exp(t^2/2))^s$ which is the generating function of a normal with variance $\sqrt{s}$. Hence the moments of the normal dominate those of the binomial. 

%Once we show this we conclude that the KL divergence is bounded as
%\begin{align*}
%n \log(1 - s\theta^2 \sqrt{1 - 2s\theta}) +n s\theta^2 (1-2s\theta)^{-2} \asymp n s\theta^2 ((1-2s\theta)^{-2}-\sqrt{1 - 2s\theta}) \asymp n s^2 \theta^3
%\end{align*}

We will now show that the central moments of the balanced binomial distribution, $\mathrm{Bin}(s,.5)$, satisfy certain sharp inequalities in terms of polynomials in $s$. Note that the moments $\EE (2k - s)^{2m}$ can be thought of as $\EE (X_1 + \ldots + X_s)^{2m}$ where $X_i$ are i.i.d. Rademacher random variables. Denote with $P_{2m}(s) = \EE (X_1 + \ldots + X_s)^{2m}$. We will now show the following recursive formula for $P_{2m}(s)$
\begin{lemma}[Recursion for $P_{2m}(s)$]\label{tricky:combinatorial:lemma}
\begin{align*}
    P_{2m}(s) = \sum_{k = 0}^{m-1} (-1)^k{2m-1 \choose 2k + 1} E_{2k + 1} s P_{2m-2k - 2}(s),
\end{align*}
where $E_{2k+1}$ are the tangent (aka zag) numbers, the first few of which are given by $E_{1} = 1$, $E_3 = 2$, $E_5 = 16$, $E_7 = 272$ and so on. In addition it holds that (for $s \in \NN$)
\begin{align}
    &\sum_{k = 2l + 1}^{m-1} (-1)^k{2m-1 \choose 2k + 1} E_{2k + 1} s P_{2m-2k - 2}(s) \leq 0, \label{negative:summation}\\
    & \sum_{k = 2l}^{m-1} (-1)^k{2m-1 \choose 2k + 1} E_{2k + 1} s P_{2m-2k - 2}(s) \geq 0.\label{positive:summation}
\end{align}

\end{lemma}

The proof of Lemma \ref{tricky:combinatorial:lemma} (and all subsequent lemmas) is deferred to the supplement. Here we give a little comment on the tangent numbers: first they appear in the Taylor expansion of $\tan(x)$ around $0$. Second, they are the number of alternating (aka zigzag) permutations of $2k+1$ numbers, i.e., permutations of the type $a_1 < a_2 > a_3 < a_4 > a_5 < \ldots > a_{2k+1}$. Using Lemma \ref{tricky:combinatorial:lemma} and the fact that $\EE(2k-s)^0 = 1$, it is simple to verify that $\EE(2k-s)^{2m}$ are polynomials of $s$ with integer coefficients of degree $m$. Suppose that $P_{2m}(s) = \sum_{k = 0}^m a^{2m}_k s^{m-k}$. Next we need the following 

\begin{lemma}[Sharp Bounds for $P_{2m}(s)$]\label{bounds:for:p2ms} We have that for any $s \in \NN$ and $l$ the following is true (provided that the index in the summation makes sense):
\begin{align*}
    \sum_{k = 0}^{2l + 1} a^{2m}_k s^{m-k}  \leq P_{2m}(s) \leq \sum_{k = 0}^{2l} a^{2m}_k s^{m-k} 
\end{align*}
\end{lemma}

We are now ready to begin our calculations. We will bound the moments of $\EE(2k-s)^{2m}$ with the first three coefficients of the polynomial $P_{2m}(s)$. In order to do so we need to derive the values of these coefficients. We do so in the following lemma.

\begin{lemma}[Coefficients of $P_{2m}(s)$]\label{coefficients:of:p2ms}
We have 
\begin{align*}
    a^{2m}_0 & = (2m)!!, a^{2m}_{1} = -m(m-1) (2m)!!/3\\
    a^{2m}_{2} & = (2m-1)a^{2m-2}_{2} + 2 {2m-1\choose 3}(m-2)(m-3)\frac{(2m-4)!!}{3}  + 16{ 2m-1\choose 5}(2m-6)!!,
\end{align*}
where $a^{2m}_0$ is defined as $1$ for $m = 0$, $ a^{2m}_{2}$ is defined for $m \geq 3$ and is $0$ otherwise.
\end{lemma}

Next we will control the coefficient $a^{2m}_{2}$ from above. We have
\begin{lemma}\label{bound:on:a2m2}
We have that for a sufficiently large absolute constant $C > 0$:
\begin{align*}
    a^{2m}_2 \leq C (m-2)(m-1)m(m+1) (2m)!!
\end{align*}
\end{lemma}

Using Lemma \ref{bound:on:a2m2} we conclude that
\begin{align*}
    \EE (2k-s)^{2m} & \leq (2m)!! s^m - m(m-1) (2m)!!/3 s^{m-1}  + C (m-2)(m-1)m(m+1)(2m)!! s^{m-2} 
\end{align*}
Hence
\begin{align*}
    C(\theta, s) & \leq \sum_{m = 0}^{\infty} \frac{ (2m)!! (\theta s)^m}{m!} - \theta \sum_{m = 2}^{\infty} \frac{m(m-1) (2m)!! (\theta s)^{m-1}}{3 m!} + \theta^2 \sum_{m = 3}^{\infty} \frac{C (m-2)(m-1)m(m+1)(2m)!! (\theta s)^{m-2}}{m!}
\end{align*}

Simple algebra verifies that $\frac{(2m)!!}{m!} = {-1/2 \choose m}(-2)^m$ and $\frac{m(m-1) (2m)!!}{3m!} = {-5/2 \choose m-2}(-2)^{m-2}$ and therefore (using $(2m)!!/m! \leq 2^m$):
\begin{align*}
    C(\theta, s) & \leq (1-2s\theta)^{-1/2} - s\theta^2(1-2s\theta)^{-5/2} + \theta^2 (\theta s)\sum_{m = 3}^{\infty} C (m-2)(m-1)m(m+1) (2\theta s)^{m-3}.
\end{align*}
Now we use a simple trick that 
\begin{align*}
     \sum_{m = 3}^{\infty}(m-2)(m-1)m(m+1) x^{m-3} = \frac{d^4}{dx^4} (\sum_{m = 3}^{\infty} x^{m+1}) = \frac{d^4}{dx^4} \frac{x^4}{1-x} = \sum_{i = 1}^5 C_i \frac{x^{i-1}}{(1-x)^{i}},
\end{align*}
where $C_i$ are some absolute constants. Finally we obtain
\begin{align*}
    C(\theta, s) & \leq (1-2s\theta)^{-1/2} - s\theta^2(1-2s\theta)^{-5/2} + \theta^2 \sum_{i = 1}^5 C'_i \frac{(s\theta)^{i}}{(1-2s\theta)^{i}},
\end{align*}
for some asbolute constants. The above can clearly be made very small as $s\theta$ is made small.

Finally using $\log(1-x) \leq -x$ for $x \geq 0$ small enough we have 
\begin{align*}
    n \log(C(\theta,s) \sqrt{1 - 2s\theta})+n s\theta^2 (1-2s\theta)^{-2} \asymp n \theta^2 \sum_{i = 1}^5 C'_i \frac{(s\theta)^{i}}{(1-2s\theta)^{i - 1/2}}
\end{align*}
which is going to be very small given that $\theta \ll \frac{1}{\sqrt{n}} \wedge \frac{1}{s}$.
\end{proof}

\begin{proof}[Proof of Theorem \ref{thm:Ising:computation}]
Recall that $\bU_{i}$ and $\bV_{i}$ for $i \in [n]$, denote the $n$ i.i.d. copies of the models $\sign(N(0, I + \sigma 1_{[s]}1_{[s]}^T))$ and $n$ i.i.d. copies of the Curie-Weiss model with parameters $s$ and $\theta$. Let $\Ub = \{\bU_i\}_{i \in [n]}$ and $\Vb = \{\bV_{i}\}_{i \in [n]}$ be the collections of the variables $\bU_{i}$ and $\bV_{i}$ for $i \in [n]$ and recall that $\bY$ and $\bY'$ are the collections of the variables $\{Y_i\}_{i \in [n]}$ and $\{Y_i'\}_{i \in [n]}$, i.e., $\bY$ is a vector consisting of $n$ i.i.d copies from a Gaussian distribution with variance $(1-2s\theta)^{-1}$ and $\bY'$ is a vector consisting of $n$ i.i.d. copies from a CWN distribution with parameters $s, \theta$. As we argued earlier, it suffices to prove that the total variation $\operatorname{TV}(\mathcal{L}(\Ub),\mathcal{L}(\Vb))$ is small.

To do so we will use Fact 3.1. part 5 of \cite{brennan2019optimal}, which states that
\begin{align*}
\operatorname{TV}(\mathcal{L}(\Ub),\mathcal{L}(\Vb)) & \leq \operatorname{TV}(\mathcal{L}(\bY),\mathcal{L}(\bY'))  + \EE_{\by \sim \bY} \operatorname{TV}(\mathcal{L}(\Ub | \bY = \by), \mathcal{L}(\Vb | \bY' = \by)). 
\end{align*}
By Theorem \ref{thm:CWN:closeto:Gaussian} we know that $\operatorname{TV}(\mathcal{L}(\bY),\mathcal{L}(\bY'))$ is small under the condition $\theta \ll \frac{1}{\sqrt{n}} \wedge \frac{1}{s}$. We will now argue that the second term $\operatorname{TV}(\mathcal{L}(\Ub | \bY = \by), \mathcal{L}(\Vb | \bY' = \by))$ is also small under the same condition which will complete the proof. 

Recall that $\sigma = \frac{\pi \theta}{1- 2 s\theta}$ and let $\kappa = \pi \theta$. Let \begin{align*}
    g(\sqrt{\kappa 2/\pi } y) := \frac{\exp(\sqrt{\kappa 2/\pi } y)}{\exp(\sqrt{\kappa 2/\pi } y) + \exp(-\sqrt{\kappa 2/\pi } y)}.
\end{align*} 
We will now control 
\begin{align}\label{TV2:bound}
\MoveEqLeft \operatorname{TV}(\mathcal{L}(\Ub | \bY = \by), \mathcal{L}(\Vb | \bY' = \by))^2 \nonumber\\
& = \bigg(\sum_{\{k_i\}_{i = 1}^n} \bigg|\prod_{i \in [n]} {s \choose k_i}(\Phi(\sqrt{\kappa} y_i))^{k_i} (1 - \Phi(\sqrt{\kappa} y_i))^{s-k_i}\nonumber \\
& - \prod_{i \in [n]} {s \choose k_i} g(\sqrt{\kappa 2/\pi } y_i)^{k_i} (1-g(\sqrt{\kappa 2/\pi } y_i)^{s-k_i}\bigg|\bigg)^2\nonumber\\
& \leq \frac{n}{2} \sum_{k_i = 0}^s \bigg[k_i (\log \Phi(\sqrt{\kappa}y_i) - \log g(\sqrt{\kappa 2/\pi } y_i) ) \nonumber\\
& + (s - k_i) (\log (1 - \Phi(\sqrt{\kappa}y_i)) - \log (1 - g(\sqrt{\kappa 2/\pi}y_i)) \bigg] \times\nonumber
\\& \times {s \choose k_i} \Phi(\sqrt{\kappa } y_i)^{k_i} (1-\Phi(\sqrt{\kappa } y_i))^{s-k_i},
\end{align}
where the last inequality is Pinsker's inequality. We now write the following inequalities which can be verified with a direct calculation from the Taylor expansions of the functions (we provide short proofs in the supplement)
\begin{align*}
\log \Phi(x) & \leq -\log(2) + \sqrt{\frac{2}{\pi}}x -\frac{x^2}{\pi} - \frac{\pi-4}{3 \sqrt{2} \pi^{3/2}} x^3  + \frac{\pi- 3}{3 \pi^2}x^4 +  \frac{96-40\pi + 3\pi^2}{60\sqrt{2}\pi^{5/2}}x^5 + Cx^6,\\
\log (1-\Phi(x)) & \leq -\log(2) - \sqrt{\frac{2}{\pi}}x -\frac{x^2}{\pi} + \frac{\pi-4}{3 \sqrt{2} \pi^{3/2}} x^3 + \frac{\pi - 3}{3 \pi^2}x^4 -  \frac{96-40\pi + 3\pi^2}{60\sqrt{2}\pi^{5/2}}x^5 + Cx^6,\\
\log g(\sqrt{2/\pi}x) & \geq - \log(2)+ \sqrt{\frac{2}{\pi}}x -\frac{x^2}{\pi} + \frac{x^4}{3\pi^2} - \frac{8x^6}{45\pi^3},\\
\log (1-g(\sqrt{2/\pi}x)) & \geq - \log(2)- \sqrt{\frac{2}{\pi}}x -\frac{x^2}{\pi} + \frac{x^4}{3\pi^2} - \frac{8x^6}{45\pi^3},
\end{align*}
for a sufficiently large absolute constant $C > 0$. Hence continuing the bound from \eqref{TV2:bound} we obtain the following:

\begin{align}
\MoveEqLeft \sum_{k_i = 0}^s \bigg[(s - 2k_i)\frac{\pi-4}{3\sqrt{2}\pi^{3/2}} (\sqrt{\kappa}y_i)^3 + s\frac{\pi-4}{3\pi^2}(\sqrt{\kappa}y_i)^4 + (2k_i-s)\frac{96-40\pi + 3\pi^2}{60\sqrt{2}\pi^{5/2}}(\sqrt{\kappa}y_i)^5 \nonumber \\
&+ (\frac{8}{45 \pi^3} + C)s(\sqrt{\kappa}y_i)^6\bigg]{s \choose k_i} \Phi(\sqrt{\kappa } y_i)^{k_i} (1-\Phi(\sqrt{\kappa } y_i))^{s-k_i} \nonumber\\
& = \frac{\pi-4}{3\sqrt{2}\pi^{3/2}}s(1 - 2 \Phi(\sqrt{\kappa}y_i))(\sqrt{\kappa}y_i)^3 + s\frac{\pi-4}{3\pi^2}(\sqrt{\kappa}y_i)^4 + \nonumber \\
& + s(2 \Phi(\sqrt{\kappa}y_i) - 1)\frac{96-40\pi + 3\pi^2}{60\sqrt{2}\pi^{5/2}}(\sqrt{\kappa}y_i)^5 + s(\sqrt{\kappa}y_i)^6(\frac{8}{45 \pi^3} + C). \label{long:summation:eq}
\end{align}
Next we note that (and give proofs in the supplement)
\begin{align*}
& (1 - 2\Phi(x))x^3 \leq -\sqrt{\frac{2}{\pi}}x^4 + \frac{x^6}{3 \sqrt{2\pi}}, \\
&(2 \Phi(x) - 1) x^5 \leq \sqrt{\frac{2}{\pi}}x^6. 
\end{align*}
Putting everything together we conclude that there exists a universal constant $C$ such that \eqref{long:summation:eq} is bounded as $C s \kappa^3 y_i^6$
% \begin{align*}
% \leq s C \kappa^3 y_i^6
% \end{align*}
Then using the fact that square root is a concave function, by Jensen's inequality we can bring the expectations over $y_i$ inside the square root to obtain
\begin{align*}
 \EE_{y \sim \bY} \operatorname{TV}(\cL(\Ub | \bY = \by), \cL(\Vb | \bY' = \by)) \leq \sqrt{n s \kappa^3 C\EE y_i^6} \asymp \sqrt{ns\kappa^3/(1 - 2s\theta)^3}. 
 \end{align*}
 which completes the proof. 
\end{proof}

\section{Discussion}

In this paper we studied structure detection problems in zero-field ferromagnetic Ising models. Our upper and lower bounds demonstrated that graph arboricity is a key concept which drives the testability of structure detection. We furthermore argued that under a sparse PCA conjecture no polynomial time tests can test the problem unless the signal strength is of the order of $\frac{1}{\sqrt{n}}$, which is statistically sub-optimal for graphs with high arboricity. 

There are several important questions which we leave for future work. First, our upper bound results are derived under the assumption that $\|\Theta\|_F \leq \frac{1}{2}$. This assumption is needed to ensure that the terms \eqref{WH:def} concentrate around their mean value. This may not be a necessary condition, and we anticipate that the tests we develop might work beyond this regime.

Second, an interesting question that is left open is whether one can develop upper and lower bounds for problems of the type \eqref{eq:testdefS2} in the dense regime when $s \gg \sqrt{d}$. We believe that this regime may require completely different tests than the ones we developed in this paper. 

%Finally, our computational lower bound, which relies on the sparse PCA conjecture, works only for linear tests on the covariance matrix. As we mentioned earlier, the computational hardness of sparse PCA conjecture has been established under the widely believed planted clique conjecture \citep{gao2014sparse,berthet2013complexity,brennan2018reducibility}. It will be interesting to extend our results beyond linear tests on the covaraince matrix. We currently do not know of a way to prove such a result based on the planted clique conjecture. However, our results under the oracle computational model strongly suggest that indeed it is unlikely that polynomial time tests for detection exist when the signal strength is of smaller order than $\frac{1}{\sqrt{n}}$.

\section*{Funding}

This work was supported by the National Science Foundation [BIGDATA 1840866, RI 1408910, CAREER 1841569, TRIPODS 1740735 to H.L.]; and an Alfred P Sloan Fellowship to H.L.

\section*{Acknowledgements}
We would like to thank the anonymous reviewer, associate editor and editor for their helpful comments.

\section*{Data Availability Statement}
No new data were generated or analysed in support of this research.

\appendix
\section*{Appendix}

\section{Proofs}\label{proofs:mainsection}
\subsection{Lower Bound Proofs}\label{subsec:lowerproof}

We first introduce two important lemmas.

\begin{lemma}\label{lemma:cEbounds} For a multigraph $G=(\overline{V},E)$, define the following two classes of Eulerian spanning subgraphs and connected Eulerian subgraphs of $G$ with $k$ edges. 
	\begin{align*}
	\cE_c(k,G) :=& \big\{ \tilde{G} = (\tilde{V},\tilde{E}):\tilde{V}\subseteq \overline{V},~\tilde{E}\subseteq E,~|\tilde{E}| = k,~\tilde{G}\text{ is a connected}\\&~~\text{Eulerian graph}\big\},\\
	\cE(k,G) :=& \big\{ \tilde{G} = (\overline{V},\tilde{E}): \tilde{E}\subseteq E,~|\tilde{E}| = k,~\tilde{G}\text{ is an Eulerian graph}\big\}.
	\end{align*}
	Let $\rmA$ be the adjacency matrix of $G$. Then for $k\geq 2$, we have
	\begin{align*}
	|\cE_c(k,G)| \leq \|\rmA\|_F^k,\text{ and }|\cE(k,G)| \leq 2^k\|\rmA\|_F^k.
	\end{align*} 
\end{lemma}

\begin{proof}
	For the first inequality, note that we have
	\begin{align}
	(\rmA^k)_{(i,i)} &= \sum_{r_1,\ldots,r_{k-1}\in \overline{V}} \rmA_{ir_1}\rmA_{r_1r_2}\cdots \rmA_{r_{k-2}r_{k-1}}\rmA_{r_{k-1}i}\nonumber\\
	&=\text{the number of length-$k$ closed walks starting at vertex $i$}.\label{eq:countclosedwalk}
	\end{align}
	Summing up all possible starting vertices, we get
	\begin{align*}
	|\cE_c(k,G)|\leq |\{ \text{legnth-$k$ closed walks in $G$} \}|\leq \tr\big( \rmA^k \big) \leq \|\rmA\|_F^k.
	\end{align*}
	This proves the first inequality. For the second inequality, we use induction. First for $|\cE(2,G)|$, we have
	\begin{align*}
	|\cE(2,G)| = |\cE_c(2,G)|\leq \|\rmA\|_F^2 \leq 2^2\|\rmA\|_F^2.
	\end{align*}
	Suppose that for $l\leq k$ we have $|\cE(l,G)| \leq 2^l\|\rmA\|_F^l$. Then for $|\cE(k+1,G)|$, by the fact that $\cE(1,G)=\cE_c(1,G)=\emptyset$, we have
	\begin{align*}
	|\cE(k+1,G)| \leq \sum_{l=2}^{k-1} |\cE_c(l,G)|\cdot |\cE(k+1-l,G)| + |\cE_c(k+1,G)|.
	\end{align*}
	Plugging in the inequalities for $|\cE(l,G)|$, we get
	\begin{align*}
	|\cE(k+1,G)|&\leq \sum_{l=2}^{k-1} \|\rmA\|_F^{l}\cdot 2^{k+1-l}\|\rmA\|_F^{k+1-l} + \|\rmA\|_F^{k+1}\\
	&\leq \|\rmA\|_F^{k+1} \cdot \bigg(\sum_{l=2}^{k-1} 2^{k+1-l} + 1 \Bigg)\\
	&\leq 2^{k+1}\|\rmA\|_F^{k+1}.
	\end{align*}
	Therefore by induction we get the second inequality.
\end{proof}

\begin{lemma}\label{lemma:pqbound}
	Let $G$ be a multigraph with vertex set $\overline{V}=\{1,\ldots, d\}$ and adjacency matrix $\rmA$. Let $V\subseteq \overline{V}$ be a vertex set. For $k\geq 2$, we define 
	\begin{align*}
	p_k(G,V) &= \big|\big\{\tilde{G}\in \cE_c(k,G): \tilde{G}\text{ contains at least two distinct vertices in $V$}\big\}\big|,\\
	q_k(G,V) &= \big|\big\{\tilde{G}\in \cE(k,G): \exists i,j\in V,i,j\text{ are contained in one connected} \\&\quad~~~\text{component of $\tilde{G}$} \big\}\big|.
	\end{align*} 
	Then we have
	\begin{align}
	p_k(G,V)&\leq (k-1) \cdot |V|^2\cdot \|\rmA\|_1^{k-2}, \label{eq:pbound}\\
	q_k(G,V)&\leq \big(2^k\cdot |V| \cdot \|\rmA\|_F^k\big) \land \big[ k\cdot 2^{k-2}\cdot |V|^2\cdot(\|\rmA\|_1 \lor \|\rmA\|_F)^{k-2}\big]   .\label{eq:qbound}
	\end{align}
\end{lemma}
\begin{proof}
	We first prove \eqref{eq:pbound}. By definition, we have
	\begin{align*}
	p_k(G,V)&\leq |V|\cdot(|V|-1)\cdot \max_{i,j\in V} \big|\big\{\tilde{G}\in \cE_c(k,G): \tilde{G}\text{ contains vertices }i\text{ and }j\big\}\big|\\
	&\leq |V|^2\cdot \max_{i,j\in V} |\{\text{length-$k$ closed walks in $G$ starting at $i$ and traversing $j$}\}|.
	\end{align*}
	Note that each vertex can have at most $\|\rmA\|_1$ neighbors. Therefore we can bound the number of length-$k$ Eulerian circuits starting at vertex $i$ and containing vertex $j$ by counting the possible vertices on the walk:
	\begin{itemize}
		\item The number of possible positions of vertex $j$ in $V$ is $k-1$.
		\item The number of choices of the rest $k-2$ vertices is at most $\|\rmA\|_1^{k-2}$.
	\end{itemize}
	This completes the proof of \eqref{eq:pbound}.
	
	Now we prove \eqref{eq:qbound}. Suppose that $\tilde{G}$ is a subgraph of $G$ with $k$ edges such that one of its connected components contains at least two distinct vertices in $V$. Let $l$ be the number of edges of this connected component. Then by definition, clearly the rest connected components form a graph in $\cE(k-l, G)$. Therefore we have
	\begin{align*}
	q_k(G,V)&\leq \sum_{l=2}^{k-2} p_l(G,V)\cdot |\cE(k-l,G)| + p_k(G,V).
	\end{align*}
	By \eqref{eq:pbound} and Lemma~\ref{lemma:cEbounds}, we have
	\begin{align*}
	q_k(G,V)&\leq \sum_{l=2}^{k-2} (l-1)\cdot |V|^2 \|\rmA\|_1^{l-2} \cdot 2^{k-l}\|\rmA\|_F^{k-l} + k\cdot |V|^2 \|\rmA\|_1^{k-2}\\
	&\leq |V|^2\cdot(\|\rmA\|_1 \lor \|\rmA\|_F)^{k-2}\cdot \Bigg( \sum_{l=2}^{k-2}(l-1)\cdot2^{k-l} + k  \Bigg)\\
	&\leq k\cdot 2^{k-2}\cdot |V|^2\cdot(\|\rmA\|_1 \lor \|\rmA\|_F)^{k-2},
	\end{align*}
	where the last inequality holds because for $l \geq 2$ we have $l-1 \leq 2^{l-2}$. 
	Moreover, for $V\neq \emptyset$, by Lemma~\ref{lemma:cEbounds}, clearly we have $q_k(G,V)\leq |\cE(k,G)| \leq 2^k\|\rmA\|_F^k \leq 2^k \cdot |V|\cdot \|\rmA\|_F^k$. When $V=\emptyset$, by definition we have $q_k(G,V) = 2^k \cdot |V|\cdot \|\rmA\|_F^k =0$. This completes the proof.
\end{proof}

\begin{proof}[Proof of Lemma~\ref{lemma:equivalentprobabilitymassfunction}]
	For any $i,j\in \overline{V}$, we have
	\begin{align*}
	\exp(\theta X_i X_j) = \cosh(\theta X_i X_j) + \sinh(\theta X_iX_j) = \cosh(\theta X_i X_j) [1 + \tanh(\theta X_i X_j)].
	\end{align*}
	Note that $\cosh(x)$ is an even function, and $X_iX_j$ is binary. Therefore we have $\cosh(\theta X_i X_j) \equiv \cosh(\theta)$. Similarly, $\tanh(x)$ is an odd function, by checking the function values at $X_i X_j = 1$ and $X_i X_j = -1$ we obtain $\tanh(\theta X_iX_j) = \tanh(\theta) X_iX_j$. Therefore we have 
	\begin{align}\label{eq:simpleIsing}
	\exp(\theta X_i X_j) = c (1 + t X_i X_j),
	\end{align}
	where $c= \cosh(\theta)$ and $t= \tanh(\theta)$. Plugging \eqref{eq:simpleIsing} into the definition of $\PP_\Theta(\bX)$ proves \eqref{eq:simplifymassfunction}.
\end{proof}

\begin{proof}[Proof of Lemma~\ref{lemma:lecam}]
	Define
	\begin{align*}
	\overline{\PP} = \frac{1}{|\cS^*|}\sum_{\Theta\in\cS^*} \PP_{\Theta,n}, 
	\end{align*}
	then by Neyman-Pearson's lemma we have
	\begin{align*}
	\gamma(\cS^*)\geq \inf_{\psi}\Big[ \PP_{0 }(\psi = 1) + \overline{\PP}(\psi = 0) \Big] = 1 - \mathrm{TV}(\overline{\PP},\PP_{0 ,n}),
	\end{align*}
	where $\mathrm{TV}(\overline{\PP},\PP_{0 ,n}) := \max_{A\subseteq \{\pm1\}^{n\times d}} |\overline{\PP}(A) - \PP_{0 ,n}(A)|$ is the total variation distance between $\overline{\PP}$ and $\PP_{0 ,n}$. Note that for total variation distance we have
	\begin{align*}
	\mathrm{TV}(\overline{\PP},\PP_{0 ,n}) &= \frac{1}{2} \sum_{X\in \{\pm 1 \}^{n\times d}} |\overline{\PP}(X)- \PP_{0 ,n}(X)|\\ &= \frac{1}{2} \sum_{X\in \{\pm 1 \}^{n\times d}} \bigg|\frac{\overline{\PP}(X)}{\PP_{0 ,n}(X)} - 1 \bigg|\cdot \PP_{0,n}(X).
	\end{align*}
	Applying Cauchy-Schwartz inequality to the right-hand side above gives
	\begin{align*}
	\mathrm{TV}(\overline{\PP},\PP_{0 ,n})  \leq \frac{1}{2} \sqrt{\EE_{0 ,n} \Bigg\{ \bigg[ \frac{\overline{\PP}(X)}{\PP_{0 ,n}(X)} - 1 \bigg]^2 \Bigg\} } = \frac{1}{2} \sqrt{ \EE_{0 ,n} \Bigg[ \frac{\overline{\PP}^2(X)}{\PP_{0 ,n}^2(X)} \Bigg] - 1} .
	\end{align*}
	It then suffices to show that 
	\begin{align*}
	\EE_{0 ,n} \Bigg[ \frac{\overline{\PP}^2(X)}{\PP_{0 ,n}^2(X)} \Bigg]  =  \frac{1}{|\cS^*|^2}\sum_{\Theta,\Theta'\in \cS^*} \EE_{0 ,n}\bigg[ \frac{\PP_{\Theta,n}}{\PP_{0 ,n}}\frac{\PP_{\Theta',n}}{\PP_{0 ,n}}  \bigg],
	\end{align*}
	which follows by direct calculation.
\end{proof}

\begin{proof}[Proof of Lemma~\ref{boundpolydiff1}] Since there cannot be multiple edges in $G$ connecting the same two vertices, the coefficient of $t^2$ in $f_{G}(t)$ is $0$. For the same reason the coefficient of $t^2$ in $f_{G'}(t)$ is also $0$. 
In $f_{ G, G'}(t)$, the only possible way to form a two-edge Eulerian circuit is to pick one edge from $E(G)$ and to pick another edge from $E(G')$ connecting to the same two vertices. Therefore $u_2 = | E(G)\cap  E(G')|$.

For $u_3$, note that $3$-edge Eulerian subgraphs must be triangles. If a triangle only uses edges in $E(G)$, then it is counted in the coefficient of $t^3$ in $f_{G}(t)$. Similarly, if a triangle only uses edges in $G'$, it is also counted in the coefficient of $t^3$ in $f_{G'}(t)$. Therefore $u_3$ is the number of triangles that use at least one edge in $E(G)$ and another edge in $E(G')$, which is defined as $\Delta_{G.G'}$.

We denote by $\cE(G)$ and $\cE(G')$ the sets of Eulerian subgraphs of $G$ and $G'$ respectively. 
For $k\geq 4$, by \eqref{eq:fE}, the coefficient of $t^k$ in $f_{ G}(t) f_{ G'}(t)$ is equal to 
$$|\{\tilde{G}\in \cE(k,G\oplus G'): \exists G_1 \in \cE(G), G_2\in\cE(G')~s.t.~\tilde{G}=G_1\oplus G_2 \}|.$$ 
We now prove that, for $\tilde{G}\in\cE(k,G\oplus G')$, if each connected component contains at most one vertex in $V(G)\cap V(G')$, then there exist $G_1 \in \cE(G)$ and $ G_2\in\cE(G')$ such that $\tilde{G}=G_1\oplus G_2$. To prove this statement, take a fixed connected component of $\tilde{G}$. Suppose first that the connected component does not contain any vertices in $V(G)\cap V(G')$. Then it follows that all of its edges must be contained either in $E(G)$ or $E(G')$. Next consider the case when the connected component contains only one vertex $v\in V(G)\cap V(G')$. Since this connected component must be a connected Eulerian graph, we can consider the Eulerian circuit starting and ending at $v$. If we start walking along the circuit on an edge in $E(G)$, then since $v$ is the only vertex contained in the intersection $V(G)\cap V(G')$, we cannot reach vertices in $E(G')$ until we return to $v$. Upon returning to $v$, we have completed a closed walk using purely edges in $G$. We can continue this process to obtain closed walks on $G$ and $G'$ starting and ending at $v$. Concatenating all the closed walks on $G$ gives $G_1$. Similarly, concatenating all the closed walks on $G'$ gives $G_2$. We have proved that 
\begin{align*}
& \cE(k,G\oplus G') \backslash \{\tilde{G}\in \cE(k,G\oplus G'): \exists G_1 \in \cE(G), G_2\in\cE(G')~s.t.~\tilde{G}=G_1\oplus G_2 \} \subseteq \\
&  \big\{\tilde{G}\in \cE(k,G): \exists i,j\in V(G)\cap V(G'),i,j\text{ are in one connected component of $\tilde{G}$} \big\}.
\end{align*} 
Therefore by the definition of $q_k(\cdot,\cdot)$ we have $ u_k \leq q_k[G\oplus G',V(G)\cap V(G')]$.
\end{proof}

	\begin{proof}[Proof of Lemma~\ref{boundpolydiff2}]
		The bounds for $\cE(k,G)$ and $q_k(G,V)$ are included in Lemma~\ref{lemma:cEbounds} and Lemma~\ref{lemma:pqbound}. We now prove the bound for $\Delta_{G,G'}$. We remind the reader that for a graph $G$ and a vertex set $V$, $G_V$ denotes the graph obtained by restricting $G$ on the vertex set $V$. Note that if a triangle has one edge in $E(G)$ and two edges in $E(G')$, then the two vertices of the edge in $E(G)$ must be in $V(G)\cap V(G')$. Therefore, an upper bound of the number of triangles that have one edge in $E(G)$ and two edges in $E(G')$ is given by the following procedure:
		\begin{itemize}
			\item Pick an edge $e$ from $E[ G_{ V(G)\cap V(G') } ]$.
			\item Pick a common neighbour of the two vertices of edge $e$.
		\end{itemize}
		Since all graphs in $\cG^*$ have arboricity $\cR$, by the definition of arboricity we have 
		\begin{align*}
		\Delta_{G,G'} &\leq \big|E[ G_{ V(G)\cap V(G') } ]\big| \cdot \|\rmA_{G'}\|_1 + \big|E[ G'_{ V(G)\cap V(G') } ]\big| \cdot \|\rmA_{G}\|_1 \\
		&\leq 2 |V(G)\cap V(G')|\cdot \cR \cdot \Gamma.
		\end{align*}
		This completes the proof.
	\end{proof}

\begin{proof}[Proof of Lemma~\ref{lemma:negativeassociation}]
	Let 
	$$A(\cG^*) = \frac{1}{|\cS^*|^2}\sum_{\Theta,\Theta'\in \cS^*} \exp[ 3n\cR|V(G)\cap V(G')| \theta^2 \big].$$
	Then we have
	\begin{align*}
	A(\cG^*) \leq \max_{\Theta\in\cS^*} \frac{1}{|\cS^*|}\sum_{\Theta'\in \cS^*} \exp \Bigg\{ 3n\cR \theta^2 \cdot \sum_{v\in V(G)}\mathds{1}\big[v\in V(G')\big] \Bigg\}.
	\end{align*}
	Consider drawing $\Theta'$ uniformly from $\cS^*$, and let $\PP_{\Theta' \sim U(\cS^*)}$ be the probability measure. By assumption, the random variables $\{ \mathds{1}[v\in V(G')] ~|~ v\in V(G)\}$ are negatively associated. Therefore
	\begin{align*}
	A(\cG^*)  &\leq \max_{\Theta\in\cS^*} \EE_{\Theta'\sim U(\cS^*)} \prod_{v\in V(G)}\exp \big\{  3n\cR \theta^2 \cdot \mathds{1}\big[v\in V(G')\big] \big\} \\
	&\leq \max_{\Theta\in\cS^*} \prod_{v\in V(G)} \EE_{\Theta'\sim U(\cS^*)} \exp \big\{ 3n\cR \theta^2 \cdot \mathds{1}\big[v\in V(G')\big]  \big\}.
	\end{align*}
	Expanding the expectation and applying the inequality $1 + x \leq \exp(x)$ gives
	\begin{align*}
	A(\cG^*)  &\leq \max_{\Theta\in\cS^*} \prod_{v\in V(G)} \Big\{ \exp \big( 3n\cR \theta^2 \big) \PP_{\Theta' \sim U(\cS^*)}[v\in V(G')] + 1 \\
	&\qquad\qquad\qquad\quad- \PP_{\Theta' \sim U(\cS^*)}[v\in V(G')]\Big\}\\
	&\leq \max_{\Theta\in\cS^*} \prod_{v\in V(G)} \exp \Big\{  \big[\exp \big( 3n\cR \theta^2 \big) - 1 \big] \PP_{\Theta' \sim U(\cS^*)}[v\in V(G')]\Big\}.
	\end{align*}
	Rearranging terms, we get
	\begin{align*}
	A(\cG^*) &\leq \max_{\Theta \in\cS^*}  \exp \Bigg\{\big[\exp \big( 3n\cR \theta^2 \big) - 1 \big] \cdot \sum_{v\in V(G)}\PP_{\Theta' \sim U(\cS^*)}[v\in V(G')]  \Bigg\} \\
	& \leq \exp \Big\{ \exp \big( 3n\cR \theta^2 \big)  \cdot \max_{\Theta\in\cS^*}\EE_{\Theta'\sim U(\cS^*)} |V(G)\cap V(G')| \Big\}\\
	& = \exp [ N(\cG^*)\cdot \exp ( 3n\cR \theta^2 )  ].
	\end{align*}
	This completes the proof.
\end{proof}

\begin{proof}[Proof of Corollary \ref{cor:lowerbound}]
Let $\cG^*$ be the set of graphs isomorphic to $G_*$. Then clearly, if $G'$ is uniformly sampling from $\cG^*$, then $\{\ind[i \in V(G')]\}_{i=1}^d$ is just a permutation of $s$ $1$s and $d-s$ $0$s. Therefore by Theorem~2.11 in \cite{joag1983negative}, the incoherence condition is satisfied. For any $G\in \cG^*$ and $v\in V(G)$, we have
\begin{align*}
	\EE_{G'\sim U(\cG^*)} |V(G)\cap V(G')| = \sum_{i\in V(G)}\EE_{G'\sim U(\cG^*)} \ind[i \in V(G')] = s\cdot s/d = s^2/d.
\end{align*}
And therefore $N(\cG^*) = s^2/d$. Moreover, by definition we have
\begin{align*}
\cR = \cR(G_*),~ V_{\max} = s,~\Lambda= \| \rmA_{G_*} \|_F,~\Gamma= \| \rmA_{G_*} \|_1,\text{ and }\cB = \cB(G_*).
\end{align*}
Therefore by Theorem~\ref{thm:lowerbound}, if
\begin{align*}
\theta \leq \sqrt{\frac{\log (d/s^2)}{6n\cR(G_*)}} \land \sqrt{\frac{\cR(G_*)}{\cB(G_*)}} \land \frac{1}{8(\|\rmA_{G_*}\|_F \vee \|\rmA_{G_*}\|_1)},
\end{align*}
then we have
\begin{align*}
\liminf _{n \rightarrow \infty} \gamma(\cS^*) = 1.
\end{align*}
\end{proof}

\subsection{Upper Bound Proofs}\label{subsec:upperproof}
The following lemma given by \cite{bhattacharya2015inference} is helpful for bounding the $\psi_1$-norm of $W_H$.
	\begin{lemma}\label{lemma:psi1normineq}
		Let $J$ be a $d\times d$ symmetric matrix with non-negative off-diagonal entries and zeros on the diagonal. If $\|J\|_2\leq 1$, then we have
		\begin{align*}
		\sum_{1\leq i,j\leq d} \log \cosh(J_{ij}) \leq  \log \EE_{0} \exp\bigg(\frac{1}{2} X^T J X\bigg) \leq -\frac{1}{2}\sum_{i = 1}^{n} \log[1-\lambda_i(J)],
		\end{align*}
		where $\lambda_1(J),\ldots,\lambda_d(J)$ are the eigenvalues of $J$.
	\end{lemma}
	
\begin{proof}[Proof of Lemma~\ref{lemma:psi1normbound}] By (5.16) in \cite{vershynin2010introduction} as an equivalent definition of $\psi_1$-norm, it suffices to prove
	\begin{align}\label{eq:upperboundproofgoal}
	\EE_\Theta \exp\bigg( \frac{\sqrt{2}|E(H)|^{1/2}}{8}\cdot W_H \bigg) \leq e.
	\end{align}
	To prove \eqref{eq:upperboundproofgoal}, first note that we have $\|\rmA_H\|_F^2 = 2|E(H)|$. By definition of the Ising model, we have
	\begin{align*}
	\EE_\Theta \exp\bigg( \frac{\sqrt{2}|E(H)|^{1/2}}{8}\cdot W_H \bigg) = \EE_\Theta \exp\bigg( \frac{\|\rmA_H\|_F}{8}\cdot W_H \bigg) = \frac{\EE_0 \exp(X^TJ X/2)}{\EE_0 \exp(X^T\Theta X/2)},
	\end{align*}
	where $J := \Theta + \rmA_H/(4\|\rmA_H\|_F)$. Therefore, 
	\begin{align}\label{eq:psi1bound2.1}
	\log \EE_\Theta \exp\bigg( \frac{\|\rmA_H\|_F}{8}\cdot W_H \bigg) = \log\EE_0 \exp\bigg(\frac{1}{2}X^TJ X\bigg) - \log \EE_0 \exp\bigg(\frac{1}{2}X^T\Theta X\bigg).
	\end{align}
	By Lemma~\ref{lemma:psi1normineq}, we have
	\begin{align*}
	\log \EE_{0} \exp\bigg(\frac{1}{2} X^T J X\bigg) \leq -\frac{1}{2}\sum_{i = 1}^{n} \log[1-\lambda_i(J)] \leq \frac{1}{2}\sum_{i = 1}^{n} \big[\lambda_i(J) + 2\lambda_i^2(J)\big],
	\end{align*}
	where the second inequality holds because for $|x|\leq 3/4$ we have $-\log(1-x) = \sum_{k\geq 1} x^k/k \leq x + 2x^2$ and by assumption we have $\|J\|_2 \leq \|J\|_F \leq 3/4$. Since $\tr(J) = \tr(\rmA_H)/(2\|\rmA_H\|_F) = 0$, we have
	\begin{align}\label{eq:psi1bound2.2}
	\log \EE_{0} \exp\bigg(\frac{1}{2} X^T J X\bigg) \leq \|J\|_F^2 \leq \frac{9}{16}.
	\end{align}
	Moreover, since $\theta_{ij} \geq 0$ for all $i,j = 1,\ldots,d$, by Lemma~\ref{lemma:psi1normineq} clearly we have
	\begin{align}\label{eq:psi1bound2.3}
	\log \EE_0 \exp(X^T\Theta X/2) \geq 0.
	\end{align}
	Plugging \eqref{eq:psi1bound2.2} and \eqref{eq:psi1bound2.3} into \eqref{eq:psi1bound2.1}, we obtain
	\begin{align*}
	\log \EE_\Theta \exp\bigg( \frac{\|\rmA_H\|_F}{8}\cdot W_H \bigg)\leq \frac{9}{16}.
	\end{align*}
	Therefore by (5.16) in \cite{vershynin2010introduction} as an equivalent definition of $\psi_1$-norm, we have $\|W_H\|_{\psi_1} \leq C|E_H|^{-1/2}$ for an absolute constant $C$.
\end{proof}

\subsection{Computational Lower Bound}
% \subsection{Proof of Lemmas in Section~\ref{section:complowerboundproof}}

\begin{proof}[Proof of Lemma~\ref{CWN:lemma}] Suppose that $\bV$ is generated according to the procedure described in Lemma~\ref{CWN:lemma}. Then we can calculate 
\begin{align*}
\MoveEqLeft \PP(V_1 = 1, \ldots, V_k = 1, V_{k+1} = -1, \ldots, V_s = -1) \\
& = \int_{-\infty}^\infty \frac{\exp( \sqrt{2\theta} y (2k-s))}{2^s\cosh(\sqrt{2\theta} y)^s} \frac{\cosh(\sqrt{2\theta}y)^s \exp(-y^2/2)}{Z_{Y'}} dy\\
& = \frac{\sqrt{2\pi}}{2^s Z_{Y'}}\exp(\theta (2k-s)^2),
\end{align*}
where $Z_{Y'}$ is the normalizing constant of $p_{Y'}(y)$. This is proportional precisely to $\exp(\theta (2k-s)^2)$ which is what the Curie-Weiss model is also proportional to. Moreover, from the above reasoning it follows that \eqref{ZY:equation} is satisfied. With this the proof is complete.
\end{proof}

\begin{proof}[Proof of Lemma~\ref{tricky:combinatorial:lemma}]
The proof will begin from the last number and work its way through the chain of additions and subtractions. By induction we will prove that $\sum_{k = l}^{m-1} (-1)^k{2m-1 \choose 2k + 1} E_{2k + 1} s P_{2m-2k - 2}(s)$ can be associated with $(-1)^l$ times the number of all sequences of the type $a_1 < a_2 > a_3 < \ldots < a_{2l} > a_{2 l + 1} > a_{2l + 2} > a_{2l + 3} > \ldots > a_{2l + 2r + 1}$ for $r = 0, \ldots, m - (l+1)$. The number $s$ multiplying the summation indicates which number $X_t$ we have selected from the first of the $2m$ brackets.
The sequence $a_1 < a_2 > a_3 < \ldots < a_{2l} > a_{2 l + 1} > a_{2l + 2} > a_{2l + 3} > \ldots > a_{2l + 2r + 1}$ indicates the positions of the brackets that have been selected to choose $X_t$ from. The remaining number $2m - 2l -2r - 2$ of the bracket positions have selected all possible $X_{-t}$ elements so that they appear in an even power (without over-counting). %The $2l + 1$ numbers in the sequence $a_1 < a_2 > a_3 < \ldots < a_{2l} > a_{2 l + 1}$ indicate the positions of the brackets that have been selected so far to choose $X_t$ from. The remaining ordered part of the sequence $a_{2l + 2} > a_{2l + 3} > \ldots > a_{2l + 2r + 1}$ indicates the brackets selected for $X_t$ from the polynomial $P_{2m-2l - 2}(s)$ (they are ordered in a decreasing manner since the polynomial counts bracket permutations uniquely).

When $l = m-1$, we have $(-1)^{m-1} E_{2m - 1} s$ can be simply equated to all sequences that form an alternating permutation on $2m - 1$ numbers, so the induction step holds. Suppose now this is true for $l+1$. We will show it for $l$. 

Consider the element $(-1)^l{2m-1 \choose 2l + 1} E_{2l + 1} s P_{2m-2l - 2}(s)$. It consists of all sequences selecting brackets for $X_t$ such that $a_1 < a_2 > a_3 < \ldots < a_{2l} > a_{2 l + 1}$ and $a_{2l + 2} > a_{2l + 3} > \ldots > a_{2l + 2r + 1}$ for $r = 0, \ldots, m - (l+1)$ (which come from $P_{2m-2l - 2}(s)$; they are ordered in a decreasing manner since   $P_{2m-2l - 2}(s)$ counts bracket permutations for choosing $X_t$ uniquely). By the induction step the previous summation consists of all sequences such that $a_1 < a_2 > a_3 < \ldots < a_{2l} > a_{2 l + 1} < a_{2l + 2} > a_{2l + 3} \ldots > a_{2l + 2r + 1}$ for $r = 0, \ldots, m - (l+1)$. Hence after the subtraction all that remains are sequences for selecting brackets such that $a_1 < a_2 > a_3 < \ldots < a_{2l} > a_{2 l + 1} > a_{2l + 2} > a_{2l + 3} \ldots > a_{2l + 2r + 1}$ which is what we wanted to show. 

It follows immediately from this discussion that \eqref{negative:summation} and \eqref{positive:summation} hold (since the number of sequences is multiplied by $(-1)^l$, and hence it will be negative with odd $l$ and positive with even $l$). 
% \begin{align*}
%     &\sum_{k = 2l + 1}^{m-1} (-1)^k{2m-1 \choose 2k + 1} E_{2k + 1} s P_{2m-2k - 2}(s) \leq 0,\\
%     & \sum_{k = 2l}^{m-1} (-1)^k{2m-1 \choose 2k + 1} E_{2k + 1} s P_{2m-2k - 2}(s) \geq 0.
% \end{align*}
In addition, the last line (i.e., $l = 0$) selects brackets for $X_t$ such that $a_1 > a_2 > \ldots > a_{2r + 1}$ for $r= 0, \ldots, m-1$, and hence it does not double count summations. With this the proof is complete.
\end{proof}

\begin{proof}[Proof of Lemma~\ref{bounds:for:p2ms}]
We will only show the right hand side inequality with the other one being similar (however in the induction step we will assume it holds in both directions). We will use induction in $m$. For $m = 1$ we have $P_{2}(s) = s + 0$ and hence the inequality is true by default in both directions. Suppose now it's true for $m \leq r$, we will show it for $m = r + 1$. WLOG assume that $2l < m = r + 1$ (if $2l = m = r + 1$ was we trivially have equality).
% \begin{align*}
%     P_{2(r + 1)}(s) = \sum_{k = 0}^{r} (-1)^k{2r + 1 \choose 2k + 1} E_{2k + 1} s P_{2r-2k}(s).
% \end{align*}
Using Lemma \ref{tricky:combinatorial:lemma} and subtracting off $\sum_{k = 0}^{2l} a^{2(r+1)}_k s^{r + 1-k}$ we obtain
\begin{align*}
    P_{2(r + 1)}(s) - \sum_{k = 0}^{2l} a^{2(r+1)}_k s^{r + 1-k}&= \sum_{k = 0}^{r} (-1)^k{2r + 1 \choose 2k + 1} E_{2k + 1} s P_{2r-2k}(s)  - \sum_{k = 0}^{2l} a^{2(r+1)}_k s^{r + 1-k}.
\end{align*}
Next observe that this subtraction will erase the first $2l + 1$ coefficients in the term multiplied by $s P_{2r}(s)$ and the first $2l$ coefficients in the term multiplied by $s P_{2r-2}(s)$ and so on. Now by the induction hypothesis we know that $(2r + 1)E_{1} s P_{2r}(s)$ minus its $2l + 1$ coefficients is non-positive, while ${2 r + 1 \choose 3} E_3 s P_{2r-2}(s)$ minus its first $2l$ coefficients is non-negative (hence $-s {2 r + 1 \choose 3} E_3     P_{2r-2}(s)$ is non-positive) and so on. It remains to notice that the remaining part of the summation which is not affected by the polynomial subtraction is non-positive by \eqref{negative:summation} of Lemma \ref{tricky:combinatorial:lemma}. This completes the proof.
\end{proof}

\begin{proof}[Proof of Lemma~\ref{coefficients:of:p2ms}]
This lemma simply compares the coefficients of the polynomials given in Lemma \ref{tricky:combinatorial:lemma}, and can be verified by a direct calculation. We omit the details.
\end{proof}

\begin{proof}[Proof of Lemma~\ref{bound:on:a2m2}]
The proof of this fact is by induction. For $m = 3$ it suffices to use $C \geq 16/360$ so that $a_2^6 = 16 \leq C (3-2)(3-1)3(3+1) (6)!!= C 360$. Suppose that $a^{2m - 2}_2 \leq C (m-3)(m-2)(m-1)m (2m-2)!!.$ We have that $P_3(m) := (m-2)(m-1)m(m+1)- (m-3)(m-2)(m-1)m = 4(m-2)(m-1)m$ is a polynomial of degree $3$. We have
\begin{align*}
    \MoveEqLeft (2m-1)a^{2m-2}_{2} + 2 {2m-1\choose 3}(m-2)(m-3)\frac{(2m-4)!!}{3} + 16{ 2m-1\choose 5}(2m-6)!! \\
    & \leq C (m-2)(m-1)m(m+1) (2m)!! - CP_3(m) (2m)!! \\
    &+ 2 {2m-1\choose 3}(m-2)(m-3)\frac{(2m-4)!!}{3} + 16{ 2m-1\choose 5}(2m-6)!!
\end{align*}

Note that when $C$ is sufficiently large $CP_3(m)(2m-1)(2m-3)$ will dominate the remaining polynomial in $m$ which is of degree $5$ for all $m \geq 3$ which completes the proof. 
\end{proof}

Below we prove several of the inequalities given in the paper. We start with
\begin{align*}
    f(x) = x^2/2 - \log(\cosh(x)) - x^4/12 \leq 0
\end{align*}

These are the first two terms in the Taylor expansion of $\log(\cosh(x))$. We will show that the first derivative of $f$ is decreasing, with a unique $0$ at $0$. This will imply that $0$ is a point of maximum of $f$ and since $f(0) = 0$ this will complete the proof. Note that since the terms match the Taylor expansion we have that $f'(0) = 0$.  

To see that $d/dx f(x)$ is decreasing we will look at the second derivative of $f(x)$. Direct calculation verifies that
\begin{align*}
    \frac{d^2}{d x^2} f(x) = -x^2 - 1/\cosh^2(x) + 1
\end{align*}
Using $\cosh(x)\leq \exp(x^2/2)$ (which can be verified by a Taylor expansion), we conclude that 
\begin{align*}
    \frac{d^2}{d x^2} f(x) \leq -x^2 - \exp(-x^2) + 1
\end{align*}

Now we consider the function $g(y) = -y - \exp(-y) + 1$ for $y \geq 0$. Taking the derivative we conclude $g'(y) = -1 + \exp(-y) \leq 0$ for $y \geq 0$ with equality when $y = 0$. Hence when $y \geq 0$ it follows that $-y - \exp(-y) + 1 \leq 0$. Hence 
\begin{align*}
    \frac{d^2}{d x^2} f(x) \leq -x^2 - \exp(-x^2) + 1 \leq 0,
\end{align*}
which completes the proof.

Next we will show that 

\begin{align*}
    f(x) = \log(\cosh(x)) - x^2/2 + x^4/12 - x^6/45 \leq 0,
\end{align*}
which will show the second two inequalities that we used. The strategy is similar to the one above. 

The second derivative of $f(x)$ is
\begin{align*}
    d^2/d x^2 f(x)= -\frac{2}{3}x^4 + x^2 + \frac{1}{\cosh^2(x)} - 1 
\end{align*}

Now we use that 
\begin{align*}
    -\frac{2}{3}x^4 + x^2 + \frac{1}{\cosh^2(x)} - 1 & = -\frac{2}{3}x^4 + 4x^2 - 3x^2 + \frac{1}{\cosh^2(x)} - 1 \\
    & = 12\frac{2}{3} (-\frac{x^4}{12} + \frac{x^2}{2}) - 3x^2 +  \frac{1}{\cosh^2(x)} - 1\\
    & \leq 8 \log(\cosh(x)) + \frac{1}{\cosh^2(x)} - 1 - 3x^2,
\end{align*}
where we used the previous part that $\log(\cosh(x)) \geq x^2/2 - x^4/12$. Now evaluating the second derivative of $g(x) := 8 \log(\cosh(x)) + \frac{1}{\cosh^2(x)} - 1 - 3x^2$, we obtain
$d^2/dx^2 g(x) = -6\tanh^4(x) \leq 0$, so the function is concave, and in addition its first derivative (which equals to $-6x -2\tanh(x)(\operatorname{sech}^2(x) - 4)$) has a $0$ at $0$. Hence we conclude that 
\begin{align*}
    d^2/d x^2 f(x) \leq 0,
\end{align*}
with a $0$ at $0$. So that $f(x)$ is concave and since its first derivative ($-2/15 x^5 + x^3/3 -x +\tanh(x)$) has a $0$ at $0$ we conclude that $f(x)\leq 0$.

Next we will show the inequality for $\log \Phi(x)$. The first five terms are from the Taylor expansion of $\log \Phi(x)$. Clearly then $\log \Phi(0) - \sum_{i = 0}^5 C_i 0^i - C0^6 = 0$ (where $C_i$ are the constants from the main text). Furthermore, by increasing the constant $C$ it is clear that the outside of some small interval $[-c_0, c_0]$ ($c_0$ depends on $C$) the function $\sum_{i = 1}^5 C_i x^i + Cx^6$ will dominate $\log \Phi(x)$. Next by the choice of the constants $C_i$ we have that $\frac{d}{dx} \log \Phi(x) - \sum_{i = 1}^5 i C_i x^{i-1} = C' x^5 + o(x^5)$ locally around $0$. Hence for a small enough $c_0$ we will have that $\frac{d}{dx} \log \Phi(x) - \sum_{i = 1}^5 i C_i x^{i-1} - 6C x^5 = (C' - 6C) x^5 + o(x^5)$, and hence the sign of the derivative will equal to $\sign(C' - 6C) \sign(x^5) = - \sign(x)$ (the last identity holding for $C$ large enough). We will therefore have that the function is increasing for $x$ in $[-c_0, 0]$ and decreasing on $[0,c_0]$ which implies that $0$ is a local maximum on $[-c_0, c_0]$ which shows that the inequality holds.

Next we will show that $(2 \Phi(x) - 1)x^5 \leq \sqrt{\frac{2}{\pi}} x^6$. For $x \geq 0$ this is equivalent to $2\Phi(x) - 1 - \sqrt{\frac{2}{\pi}}x \leq 0$. The derivative of the function is $\sqrt{\frac{2}{\pi}}\exp(-x^2/2) - \sqrt{\frac{2}{\pi}}  \leq 0$ hence $0$ is the maximum on $x \geq 0$ which completes the proof for the $x \geq 0$ case. Next for $x < 0$, the the above reasoning says that $2\Phi(x) - 1 - \sqrt{\frac{2}{\pi}}x \geq 2\Phi(0) - 1 - \sqrt{\frac{2}{\pi}}0 = 0$ which completes the proof.

Next we will argue that $(1 - 2\Phi(x))x^3 \leq -\sqrt{\frac{2}{\pi}}x^4 + \frac{x^6}{3 \sqrt{2\pi}}$. For $x \geq 0$ this is equivalent to 
$(1 - 2\Phi(x)) \leq -\sqrt{\frac{2}{\pi}}x + \frac{x^3}{3 \sqrt{2\pi}}$. Taking the derivative of $2\Phi(x) - 1-\sqrt{\frac{2}{\pi}}x + \frac{x^3}{3 \sqrt{2\pi}} $ results in $\sqrt{\frac{2}{\pi}}\exp(-x^2/2) - \sqrt{\frac{2}{\pi}} + \frac{x^2}{ \sqrt{2\pi}} \geq 0$ since $\exp(-y) \geq 1 - y$ for all $y \geq 0$. Hence the minimum of $2\Phi(x) - 1-\sqrt{\frac{2}{\pi}}x + \frac{x^3}{3 \sqrt{2\pi}}$ for $x \geq 0$ is reached at $0$ and therefore $2\Phi(x) - 1-\sqrt{\frac{2}{\pi}}x + \frac{x^3}{3 \sqrt{2\pi}} \geq 0$ for $x \geq 0$. For $x < 0$ we need to show $1 - 2\Phi(x) \geq -\sqrt{\frac{2}{\pi}}x + \frac{x^3}{3 \sqrt{2\pi}}$. Taking the derivative of $2\Phi(x) - 1 -\sqrt{\frac{2}{\pi}}x + \frac{x^3}{3 \sqrt{2\pi}}$ equals to $\sqrt{\frac{2}{\pi}}\exp(-x^2/2) - \sqrt{\frac{2}{\pi}} + \frac{x^2}{ \sqrt{2\pi}} \geq 0$ as before hence for $x \leq 0$ we have  $2\Phi(x) - 1 -\sqrt{\frac{2}{\pi}}x + \frac{x^3}{3 \sqrt{2\pi}} \leq  2\Phi(0) - 1 -\sqrt{\frac{2}{\pi}}0 + \frac{0^3}{3 \sqrt{2\pi}} = 0$ which implies that $1 - 2\Phi(x) \geq -\sqrt{\frac{2}{\pi}}x + \frac{x^3}{3 \sqrt{2\pi}}$ which is what we wanted to show.

\section{Computational Lower Bound Under Oracle Computational Model}\label{section:oralowerbound}
In this section we consider an oracle computational model \citep{kearns1998efficient,feldman2015statistical,wang2015sharp,yi2016more,feldman2017statistical,feldman2018complexity}, based on which we derive another computational lower bound result for detection problems in Ising model. The main idea of oracle computational model is to use the number of rounds of interactions between data and a certain algorithm to represent the algorithmic complexity of this algorithm. In specific, let $\bX$ be the random vector of interest and $\cX$ be the domain of $\bX$. We define
\begin{align}\label{queryspace}
    \cQ^* = \{ q: q(\bX)~\text{is a sub-exponential variable}\}.
\end{align}
We call every subset $\cQ \subseteq \cQ^*$ a query space.
Next we define the statistical query oracle.
\begin{definition}[statistical query oracle]\label{oracle} Let $n$ be the sample size of a testing problem. A statistical query oracle $r_n$ on a query space $\cQ\subseteq \cQ^*$ is a random mapping from $\cQ$ to $\RR$. Given a query $q\in \cQ$, the oracle $r_n$ returns an output $Z_q\in \RR$, such that for any tail probability $\xi\in[0,1)$, 
	\$
	\PP \Bigg(\bigcap_{q\in\cQ}\Big\{| Z_q  - \EE[q(\bX)]| \leq  \| q (\bX)\|_{\psi_1} \cdot \tau \Big\} \Bigg) \geq 1- 2\xi, \text{ where }
	\$
	\vspace{-10pt}
	\#\label{eq:query_2}
	\tau = \max  \Bigg\{\frac{\eta(\cQ)+\log (1/\xi)}{n}, \sqrt{ \frac{ 2[\eta(\cQ)+\log (1/\xi)]}  {n}} \Bigg\}.
	\#
	Here we call $\eta(\cQ)> 0$ the capacity measure of $\cQ$. When $\cQ$ is finite, we define $\eta(\cQ) = \log(|\cQ|)$.
\end{definition}
Given a query space $\cQ\subseteq \cQ^*$, we define $R_n(\cQ)$ to be the set of all statical query oracles on $\cQ$ with sample size $n$. 
We now give the definition of oracle computational model.
\begin{definition}[oracle computational model]\label{def:querymodel}
	An oracle computational model $\Psi$ is defined as a tuple $\Psi=\Psi(\cQ_\Psi, T_{\Psi} ,q_{\rm init}, \{\delta_t\}_{t=1}^{T_\Psi}, \psi)$, where
	\begin{itemize}
		\item $\cQ_\Psi$ is a subset of $\cQ^*$ that contains all queries the test will potentially use.
		\item $T_{\Psi}$ is the maximum number of rounds the model queries an oracle.
		\item $q_{\text{init}}\in \cQ_\Psi$ is the initial query.
		\item $\delta_t: (\cQ_\Psi\times \RR)^{t-1} \rightarrow \cQ_\Psi \cup \{\text{HALT}\}$ is the transition function at the $t$-th round. If $\delta_t$ returns HALT, then the model stops querying the oracle.
		\item $\psi:(\cQ_\Psi\times \RR)^{T_{\Psi}}\rightarrow \{0,1\}$ is the test function that takes the results of at most $T_{\Psi}$ queries as input, and returns the test result as binary output.
	\end{itemize}
\end{definition}

Each instance of $\Psi(\cQ_{\Psi}, T_{\Psi} ,q_{\text{init}}, \{\delta_t\}_{t=1}^{T_\Psi}, \psi)$ refers to a test algorithm. The parameter $T_\Psi$ is the query complexity of algorithm $\Psi$. We define
$\mA(T) = \{ \Psi: T_\Psi\leq T \}$ to be the set of all algorithms with query complexity at most $T$. 
Under oracle computational model, the risk of detection problem \eqref{eq:testdefS2} with maximum query complexity $T$ is defined as
\begin{equation}\label{eq:oraclerisk}
\gamma_{\rm{oracle}}\{\cS[\cG_1(G_*), \theta]\}= \inf_{\Psi\in \mA(T)}  \sup_{r_n\in R_n(\cQ_\Psi)}  \bigg\{\PP_{0}(\psi = 1) + \max_{\Theta\in\cS[\cG_1(G_*), \theta]}\PP_{\Theta}(\psi = 0)\bigg\}
\end{equation}
Note that in \eqref{eq:oraclerisk}, the supreme over $r\in R_n(\cQ_\Psi)$ implies that we consider the worst oracle. If $\liminf_{n \rightarrow \infty} \gamma_{\rm{oracle}}\{\cS[\cG_1(G_*), \theta]\} = 1$, then when $n$ is large enough, for any algorithm that queries at most $T$ rounds, there exists an oracle $r_n$ such that the algorithm cannot distinguish the null and alternative hypotheses.
We now give our main result. 
\begin{theorem}\label{thm:complowerbound_oracle} Let $G_*$ be a graph with $s$ vertices. 
	Under the statistical query model, if $T\leq d^p$ for some constant $p>0$, $s\leq d^{(1-\eta)/2}$ for some constant $\eta > 0$, and 
	\begin{align}
	\theta \leq \kappa\sqrt{\frac{1}{ n}} \land 
	\frac{1}{16s},
	\end{align}
	where $\kappa$ is some sufficiently small positive constant, then
	\begin{align*}
	    \liminf_{n \rightarrow \infty} \gamma_{\rm{oracle}}\{\cS[\cG_1(G_*), \theta]\} = 1.
	\end{align*}
\end{theorem}
\begin{proof}[Proof of Theorem~\ref{thm:complowerbound_oracle}]
	We denote by $G_\emptyset$ the empty graph, Similar to the computational lower bound analysis in Section~\ref{section:complowerbound}, we only need to consider the case where $G_*$ is an $s$-clique. Therefore we set $\cG^*$ to be the set of graphs isomorphic to $G_*$, and let $S^* = \{ \theta \rmA_{G}: G\in \cG^* \}$. Each parameter matrix $\Theta\in S^*$ can be represented by a graph $G\in \cG^*$. In the following, we always denote by $\Theta$ the parameter matrix with underlying graph $G$, and by $\Theta'$ the parameter matrix with underlying graph $G'$. For a graph $G$, in order to successfully detect it with the worst-case oracle, a test has to utilize at least one query $q$ that can distinguish $G$ from $G_\emptyset$. We define
	\begin{align*}
	\cG(q) = \{ G\in \cG^*: | \EE_{\Theta} q(\bX) - \EE_{0} q(\bX) | \geq \|q(\bX)\|_{\psi_1,0} \cdot\tau \},
	\end{align*}
	where $\|q(\bX)\|_{\psi_1,0}$ is the $\psi_1$-norm of $q(\bX)$ when $\bX$ follows the distribution $\PP_0$, and $\tau$ is defined in Definition \ref{oracle}. By the definition of $\cG(q)$, if $T\cdot\sup_{q\in \cQ_{\Psi}} |\cG(q)| < |\cG^*|$, then there must be some $G'\in \cG^*$ such that none of the $T$ queries used by the test can distinguish $G$ from $G_\emptyset$. Therefore the worst case oracle that returns $\EE_{\Theta'} q(\bX)$ when $\bX \sim\PP_0$ can still satisfy Definition \ref{oracle} but will make all the tests powerless. This gives the following lemma. 
	\begin{lemma}\label{lemma:distinguish}
		For any algorithm $\Psi$ that queries the oracle at most $T$ rounds, if $T\cdot\sup_{q\in \cQ_{\Psi}}|\cG(q)| < |\cG^*|$, then there exists an oracle $r_n\in R_n(\cQ_\Psi)$ defined in Definition~\ref{oracle} such that
		$\liminf_{n\rightarrow \infty} \gamma_{\rm{oracle}}(S^*)\geq 1. $
	\end{lemma}
	\begin{proof}
		See Section~\ref{subsec:auxiliaryproof} for a detailed proof.
	\end{proof}
	
	By Lemma~\ref{lemma:distinguish}, to prove $\liminf_{n \rightarrow \infty} \gamma_{\rm{oracle}}\{\cS[\cG_1(G_*), \theta]\} = 1$, it suffices to show that $T\cdot\sup_{q\in \cQ_{\Psi}}|\cG(q)| / |\cG^*|$ is  asymptotically smaller than one. In the rest of the proof, for any $q\in\cQ_{\Psi}$, we derive an upper bound on $|\cG(q)|$. To do so, we first split $\cG(q)$ into two subsets $\cG^+(q)$ and $\cG^-(q)$, which are given by
	\begin{align}
	&\cG^+(q) = \{ G \in \cG^* : \EE_\Theta [q(\bX)] - \EE_0[q(\bX)] > \|q(\bX)\|_{\psi_1,0}\cdot \tau \},\label{eq:Cplus}\\
	&\cG^-(q) = \{ G \in \cG^* : \EE_0 [q(\bX)] - \EE_\Theta[q(\bX)] > \|q(\bX)\|_{\psi_1,0}\cdot \tau \}.\label{eq:Cminus}
	\end{align}
	We now bound $|\cG^+(q)|$. $|\cG^-(q)|$ can be bounded in exactly the same way.
	The following lemma summarizes an inequality derived from the definition \eqref{eq:Cplus}.   
	\begin{lemma}\label{lemma:ChiSquareLowerBound}
		For any query function $q$,
		\begin{align}\label{eq:ChiSquareLowerBound}
		\frac{1}{|\cG^+(q)|^2} \sum_{G,G'\in \cG^+(q)} \EE_0\bigg[\frac{\ud \PP_\Theta}{\ud\PP_{0}} \frac{\ud \PP_{\Theta'}}{\ud\PP_{0}}\bigg] > 1 + \frac{1}{n}.
		\end{align}
	\end{lemma}
	\begin{proof}
		See Section~\ref{subsec:auxiliaryproof} for a detailed proof.
	\end{proof}
	It remains to calculate the left-hand side of \eqref{eq:ChiSquareLowerBound}. By Lemma~\ref{boundpolydiff1}, we have
	\begin{align*}
    \EE_{0 }\bigg[ \frac{\PP_{\Theta}}{\PP_{0 }}\frac{\PP_{\Theta'}}{\PP_{0 }} \bigg] &\leq 1 + | E(G)\cap  E(G')|\theta^2 + \Delta_{G,G'} \theta^3 \\
    &\quad+ \sum _{k\geq 4} q_k[G\oplus G',V(G)\cap V(G')] \theta^k.
    \end{align*}
For $| E(G)\cap  E(G')|$, we use the trivial bound that $| E(G)\cap  E(G')| \leq | V(G)\cap  V(G')|^2/2$. For $q_k[G\oplus G',V(G)\cap V(G')]$, $k\geq 4$, we apply the bound given by Lemma~\ref{boundpolydiff2} and obtain
\begin{align*}
    q_k[G\oplus G',V(G)\cap V(G')] &\leq k\cdot 2^{k-2}\cdot |V(G)\cap V(G')|^2\cdot(\|\rmA_{G\oplus G'}\|_1 \lor \|\rmA_{G\oplus G'}\|_F)^{k-2}\\
    &\leq k\cdot 2^{k-2}\cdot |V(G)\cap V(G')|^2\cdot(2s)^{k-2}\\
    &\leq 2^{k-2}\cdot 2^{k-2}\cdot |V(G)\cap V(G')|^2\cdot(2s)^{k-2}\\
    &= 8^{k-2} \cdot s^{k-2}\cdot |V(G)\cap V(G')|^2.
\end{align*}
Therefore by the assumption that $\theta \leq (16s)^{-1}$, we have
\begin{align*}
    \sum_{k\geq 4} q_k[G\oplus G',V(G)\cap V(G')] \theta^k &\leq 64|V(G)\cap V(G')|^2 s^2\theta^4 \\
    &\leq |V(G)\cap V(G')|^2 \theta^2/4. 
\end{align*}
For $\Delta_{G,G'}$, we use a bound similar to Lemma~\ref{boundpolydiff2} but more specific for cliques. If a triangle has one edge in $E(G)$ and two edges in $E(G')$, then the two vertices of the edge in $E(G)$ must be in $V(G)\cap V(G')$. Therefore, an upper bound of the number of triangles that have one edge in $E(G)$ and two edges in $E(G')$ is given by the following procedure:
\begin{itemize}
	\item Pick an edge $e$ from $E[ G_{ V(G)\cap V(G') } ]$.
	\item Pick a common neighbour of the two vertices of edge $e$.
\end{itemize}
Therefore by the trivial bound $\big|E[ G_{ V(G)\cap V(G') } ]\big|, \big|E[ G'_{ V(G)\cap V(G') } ]\big| \leq |V(G)\cap V(G')|^2/2$, we have
\begin{align*}
    \Delta_{G,G'} &\leq \big|E[ G_{ V(G)\cap V(G') } ]\big| \cdot \|\rmA_{G'}\|_1 + \big|E[ G'_{ V(G)\cap V(G') } ]\big| \cdot \|\rmA_{G}\|_1 \\
    &\leq |V(G)\cap V(G')|^2\cdot s.
\end{align*}
Therefore, we have
\begin{align*}
    \EE_{0 }\bigg[ \frac{\PP_{\Theta}}{\PP_{0 }}\frac{\PP_{\Theta'}}{\PP_{0 }} \bigg] &\leq 1 + | V(G)\cap  V(G')|^2\theta^2/2 + |V(G)\cap V(G')|^2\cdot s \cdot\theta^3 \\
    &\quad + |V(G)\cap V(G')|^2 \theta^2/4\\
    &\leq 1 + | V(G)\cap  V(G')|^2\theta^2.
\end{align*}
\iffalse
	\begin{align*}
	\EE_0\bigg[\frac{\ud \PP_G}{\ud\PP_{0}} \frac{\ud \PP_{G'}}{\ud\PP_{0}}(X)\bigg]\leq 1+ | E(G)\cap  E(G')|\theta^2 + \Delta_{G,G'} \theta^3 + |V(G)\cap V(G')| \cdot \cB\theta ^4.
	\end{align*}
	Note that it is only necessary to consider clique detection problem, by the bound of $\Delta_{G,G'} $ in Lemma~\ref{boundpolydiff1} and the assumption that $\theta < 1/(32s)$, we have
	\begin{align*}
	\EE_0\bigg[\frac{\ud \PP_G}{\ud\PP_{0}} \frac{\ud \PP_{G'}}{\ud\PP_{0}}(X)\bigg]\leq 1+ 4| V (G)\cap V(G')|^2\theta^2.
	\end{align*}
\fi
	Denote by $U[\cG^+(q)]$ uniformly choosing a graph in $\cG^+(q)$. 
	Then by Lemma~\ref{lemma:ChiSquareLowerBound}, we get
	\begin{align}\label{eq:countCq1}
	\frac{1}{n} &< \frac{1}{|\cG^+(q)|^2} \sum_{G,G'\in \cG^+(q)}  |V (G)\cap V( G')|^2\theta^{2}\\
	&\leq \theta^2 \cdot \sup_{G\in \cG^*} \EE_{G'\sim U[\cG^+(q)]} |V (G)\cap V(G')|^2.\nonumber
	%\frac{1}{|\cG^+(q)|} \sum_{G'\in \cG^+(q)} 4  |V (G)\cap V(G')|^2\theta^{2}.
	\end{align}
	\eqref{eq:countCq1} gives an lower bound of the expectation defined on the right-hand-side. In the following, we utilize this lower bound to derive an upper bound of $|\cG^+(q)|$. Inspired my similar results given in \cite{fan2018curse,lu2018edge}, we give the following lemma.
	
	\begin{lemma}\label{lemma:graphexpectationboundoraclecomplowerbound}
	    For $j=0,\ldots,s$, define $m_j = \max_{G\in\cG^*}|\{ G'\in \cG^*:|V (G)\cap V( G')|  = s - j\}|$.
	    For $k\leq |\cG^*|$, let $\mG(k) = \{\cG \subseteq \cG^*: |\cG| = k \}$ and $l(k) = \max\{r\leq s: \sum_{j=0}^r m_j \leq k \} $. Then we have
	    \begin{align*}
	        \sup_{G\in \cG^*} \sup_{\cG \in \mG(k)} \EE_{G'\sim U(\cG)} |V (G)\cap V(G')|^2 \leq \frac{\sum_{j=0}^{l(k)} (s-j)^2 m_j}{\sum_{j=0}^{l(k)}  m_j}.
	    \end{align*}
	\end{lemma}
	The intuition of Lemma~\ref{lemma:graphexpectationboundoraclecomplowerbound} is that, among all sets of graphs $\cG$ with cardinality $k$ (i.e., sets of graphs $\cG\in \mG(k)$), the ones that maximize the expectation $\EE_{G'\sim U(\cG)} |V (G)\cap V(G')|^2$ consist of graphs that make $|V (G)\cap V(G')|^2$ as large as possible.
    Let
	$
	\zeta = \inf_{0\le j \le s-1}  m_{j+1} / m_j .
	$
	Then for clique detection problem we have
%	\begin{align*}
%	m_j = \binom{s}{s-j}\binom{d-s}{j}, 
%	\end{align*}
%	and
	\begin{align*}
	\zeta = \inf \frac{m_{j+1}}{m_j} = \inf \left[\binom{s}{s-j-1}\binom{d-s}{j+1} \right]\bigg/ \left[\binom{s}{s-j}\binom{d-s}{j} \right] \geq \frac{d}{s^2} \geq d^{\eta}
	\end{align*}
	Clearly for large enough $d$ we have $\zeta > 2$. Let $h(j)=(s-j)^2$. Then by assumption, for $i<j$ we have
	$
	m_i\zeta^j-m_j\zeta^i<0
	$, $h(i) - h(j)>0$ and therefore $(m_i\zeta^j-m_j\zeta^i)[h(i) - h(j)]\leq0$. Similarly, for $i\geq j$ the same inequality $(m_i\zeta^j-m_j\zeta^i)[h(i) - h(j)]\leq0$ still holds. Therefore we have $\sum_{0\leq i, j\leq l(k)} (m_i\zeta^j-m_j\zeta^i)[h(i) - h(j)] \leq 0$. Rearranging terms gives
	\begin{align}\label{eq:mbound1}
	\frac{ \sum_{j=0}^{l(k)}  h(j)  m_j }{\sum_{j=0}^{l(k)} m_j } \leq \frac{ \sum_{j=0}^{l(k)}  h(j)  \zeta^j }{\sum_{j=0}^{l(k)} \zeta^j } = \frac{ \sum_{j=0}^{l(k)}  h(j)  \zeta^{-(s-j)} }{\sum_{j=0}^{l(k)} \zeta^{-(s-j)} }.
	\end{align} 
	We now bound the right-hand-side of \eqref{eq:mbound1}. Note that $\zeta^{-1} \leq 1/8$ for large enough $d$. For the numerator, we have
	\begin{align*}
	    \sum_{j=0}^{l(k)}  h(j)  \zeta^{-(s-j)} = \sum_{i= s-l(k)}^s i^2 \zeta^{-i} \leq [s-l(k)]^2 \zeta^{-[s-l(k)]} +  \sum_{i= s-l(k)+1}^s i^2 \zeta^{-i}.
	\end{align*}
	Since $s-l(k)+1 \geq 1$, for $i\geq s-l(k)+1 $ we have $i^2 \leq [s-l(k)+1]^2 4^{i - s+l(k)-1}$. Therefore,
	\begin{align*}
	    \sum_{j=0}^{l(k)}  h(j)  \zeta^{-(s-j)} &\leq [s-l(k)+1]^2\zeta^{-[s-l(k)+1]}\cdot \sum_{i= s-l(k)+1}^s (4\zeta^{-1})^{i - s+l(k)-1}\\
	    &\quad + [s-l(k)]^2 \zeta^{-[s-l(k)]} \\
	    & \leq  2[s-l(k)+1]^2\zeta^{-[s-l(k)+1]} + [s-l(k)]^2 \zeta^{-[s-l(k)]}\\
	    &\leq 2[s-l(k)+1]^2\zeta^{-[s-l(k)]}.
	\end{align*}
	For the denominator of the right-hand-side of \eqref{eq:mbound1}, we have
	$ \sum_{j=0}^{l(k)} \zeta^{-(s-j)} \geq \zeta^{-[s - l(k)]}$. 
	Therefore, we have
	\begin{align}\label{eq:mbound2}
	    \frac{ \sum_{j=0}^{l(k)}  h(j)  \zeta^{-(s-j)} }{\sum_{j=0}^{l(k)} \zeta^{-(s-j)}} \leq 2[s-l(k)+1]^2.
	\end{align}
	By \eqref{eq:mbound2}, \eqref{eq:countCq1} and Lemma~\ref{lemma:graphexpectationboundoraclecomplowerbound}, for $k = |\cG^+(q)|$ we have
	\begin{align*}
	    2[s-l(k)+1]^2 \geq \frac{1}{n}.
	\end{align*}
	Therefore, for large enough $d$ we have
	\begin{align}\label{eq: KminuslBound}
	s-l(k) \geq \sqrt{  \frac{1}{2\theta^2n}  } - 1.
	\end{align}
	On the other hand, by the definition of $l(k)$, we have  
	\begin{align}\label{eq:upper_bound_Cq}
	|\cG^+(q)| \leq \sum_{j=0}^{l(k)+1 } m_j \leq  m_s \cdot \sum_{j=0}^{l(k)+1 } \zeta^{j - s} \leq \frac{\zeta^{-[ s  - l(k) -1]} |\cG^*|}{1 - \zeta^{-1}} \leq 2 \zeta^{-[ s  - l(k)- 1 ] } | \cG^*| ,
	\end{align}
	where the last inequality follows from the fact that $\zeta^{-1} \leq 1/2$ for large enough $d$. Plugging \eqref{eq: KminuslBound} into \eqref{eq:upper_bound_Cq} gives
	\begin{align*}
	|\cG^+(q)| \leq 2 |\cG^*| \exp \Bigg[ -\log (\zeta)\cdot \Bigg( \sqrt{  \frac{1}{2\theta^2n}  } -2  \Bigg)  \Bigg].
	\end{align*}
	Applying the same analysis to $ |\cG^-(q)|$, we obtain
	\begin{align*}
	|\cG^-(q)| \leq 2 |\cG^*| \exp \Bigg[ -\log (\zeta)\cdot \Bigg( \sqrt{  \frac{1}{2\theta^2n}  } -2  \Bigg)\Bigg  ].
	\end{align*}
	Therefore we have
	\begin{align*}
	    |\cG(q)| \leq 4 |\cG^*| \exp \Bigg[ -\log (\zeta)\cdot \Bigg( \sqrt{  \frac{1}{2\theta^2n}  } -2  \Bigg)\Bigg  ].
	\end{align*}
	Since the inequality above holds for all $q\in \cQ_{\Psi}$, we have
	\begin{align*}
	T\cdot\frac{\sup_{q \in \cQ_{\Psi}} |\cG(q)|}{|\cG^*|} \leq 4\exp \Bigg[ \log (T) - \Bigg(\sqrt{  \frac{1}{2\theta^2n} } - 2 \Bigg)\cdot \log \zeta    \Bigg].
	\end{align*}
	If $T \leq d^p$, then
	\begin{align*}
	T\cdot\frac{\sup_{q \in \cQ_{\Psi}} |\cG(q)|}{|\cG^*|} \leq \exp \Bigg[\log (4) + p \log (d) - \Bigg(\sqrt{  \frac{1}{2\theta^2n} } - 2 \Bigg)\cdot \log \zeta    \Bigg].
	\end{align*}
	Let $\kappa < [\sqrt{2}(2+p/\eta)]^{-1}$. Then if $\theta\leq \kappa\sqrt{\frac{1}{ n}}$, for large enough $d$ we have
	\begin{align*}
	&\log (4) + p\log (d) - \Bigg(\sqrt{  \frac{1}{2\theta^2n} } - 2 \Bigg)\cdot \log \zeta \\
	&\qquad\quad \leq \log (4) + p\log (d) - \eta\Bigg(\sqrt{  \frac{1}{2\theta^2n} } - 2 \Bigg)\log d \leq -1,
	\end{align*}
	and therefore $T\cdot\sup_{q \in \cQ_{\Psi}} |\cG(q)| /|\cG^*| < 1$.  By Lemma~\ref{lemma:distinguish}, there exists an oracle $r$ such that $\liminf_{n\rightarrow \infty} \gamma_{\rm{oracle}}(S^*)\geq 1$. This completes the proof. 
\end{proof}

\subsection{Proofs of Auxiliary Lemmas}\label{subsec:auxiliaryproof}

\begin{proof}[Proof of Lemma~\ref{lemma:distinguish}]
	We consider an algorithm $\Psi$ with query space $\cQ_{\Psi}$ and $T_\Psi = T$. If $T\cdot\sup_{q\in \cQ_{\Psi}}|\cG(q)| < |\cG^*|$, then for any $T$ queries $q_1,\ldots,q_T \in \cQ_{\Psi}$, there exists $G_0\in \cG \backslash \bigcup_{t = 1}^T \cG(q_t)$. Let $\Theta_0 = \theta \rmA_{G_0}$ be the parameter matrix with underlying graph $G_0$. 
	Then by definition, for $t=1,\ldots, T$ we have
	\begin{align*}
	|\EE_{{\Theta_0}}q_t(\bX) - \EE_{0}q_t(\bX) | \leq \|q_t(\bX)\|_{\psi_1, 0}\cdot \tau.
	\end{align*}
	We set $r$ to be the oracle that returns $Z_{q_t}$ such that
	\begin{align*}
	&\PP_{0}(Z_{q_t} = \EE_{{\Theta_0}} [q_t(\bX)]) = 1,\\
	&\PP_{\Theta}(Z_{q_t} = \EE_{{\Theta}} [q_t(\bX)]) = 1,~G\in \cG_1.
	\end{align*}
	Then clearly
	\begin{align*}
	\PP_{0}( |Z_{q_t} - \EE_{{\Theta_0}} [q_t(\bX)]| \leq \|q_t(\bX)\|_{\psi_1, 0}\cdot \tau_{q_t} ) = 1,
	\end{align*}
	and hence $r$ satisfies the definition~\ref{def:querymodel}. However for $t=1,\ldots,T$, the oracle always returns the same $Z_{q_t}$ under $\PP_0$ and $\PP_{\Theta_0}$. Therefore we have
	\begin{align*}
	\PP_{0}(\psi =1) + \PP_{\Theta_0}(\psi = 0) = 1.
	\end{align*}
	This completes the proof.
\end{proof}

\begin{proof}[Proof of Lemma~\ref{lemma:ChiSquareLowerBound}]
	By \eqref{eq:Cplus}, we have
	\begin{align*}
	\|q( \bX)\|_{\psi_1,0}\cdot \tau &< \frac{1}{|\cG^+(q)|} \sum_{G\in\cG^+(q)} \left\{\EE_\Theta [{q}( \bX)] - \EE_0 [{q}( \bX)]\right\}\\
	&= \EE_0\Bigg\{{q}( \bX)\cdot \frac{1}{|\cG^+(q)|}\sum_{G\in\cG^+(q)}\left[\frac{\ud \PP_\Theta}{\ud\PP_0}( \bX)-1\right] \Bigg\}.
	\end{align*}
	Applying Cauch-Schwartz inequality on the right-hand side above gives
	\begin{align}\label{eq:ChiSquareLowerBound1}
	\|q( \bX)\|_{\psi_1,0}\cdot \tau < \underbrace{ \{ \EE_{0} [{q}^2( \bX)]\}^{1/2}}_{\dr (i)} \cdot \underbrace{\biggl(\EE_0\biggl\{\frac{1}{|\cG^+(q)|}\sum_{G\in\cG^+(q)} \biggl[\frac{\ud \PP_{\Theta}}{\ud \PP_0}( \bX)  - 1\biggr] \biggr\}^2\biggr)^{1/2}}_{\dr (ii)}.
	\end{align}
	For term (i), by the definition of $\psi_1$-norm we have  
	\begin{align}\label{eq:ChiSquareLowerBound2}
	(\EE_{0} \{[{q}( \bX)]^2\})^{1/2}\leq 2\| q( \bX) \|_{\psi_1,0}.  
	\end{align}
	For term (ii), we have 
	\begin{align}\label{eq:ChiSquareLowerBound3}
	&\biggl[\EE_{0}\biggl(\biggl\{\frac{1}{|\cG^+(q)|}\sum_{G\in\cG^+(q)} \biggl[\frac{\ud \PP_{\Theta}}{\ud \PP_0}( \bX)  - 1\biggr] \biggr\}^2\biggr)\biggr]^{1/2} \notag\\
	&\quad = \biggl(\frac{1}{|\cG^+(q)|^2}\sum_{G,G'\in\cG^+(q)}   \EE_{0} \biggl\{ \biggl[ \frac{\ud \PP_{\Theta}}{\ud \PP_0}( \bX)  - 1\biggr] \cdot \biggl[\frac{\ud \PP_{\Theta'}}{\ud \PP_0}( \bX)  - 1\biggr] \biggr\}\biggr)^{1/2}   \notag\\
	&\quad= \biggl\{ \frac{1}{|\cG^+(q)|^2}\sum_{G,G'\in\cG^+(q)}  \EE_{0} \biggl[ \frac{\ud \PP_{\Theta}}{\ud \PP_0}\frac{\ud \PP_{\Theta'}}{\ud \PP_0}( \bX)\biggr] - 1\bigg\} ^{1/2}
	\end{align}
	Plugging \eqref{eq:ChiSquareLowerBound2} and \eqref{eq:ChiSquareLowerBound3} into \eqref{eq:ChiSquareLowerBound1} and using the bound $\tau \geq \sqrt{\frac{1}{n}}$, we obtain 
	\begin{align*}
	\frac{1}{|\cG^+(q)|^2} \sum_{G,G'\in \cG^+(q)} \EE_0\left[\frac{\ud \PP_\Theta}{\ud\PP_0} \frac{\ud \PP_{\Theta'}}{\ud\PP_0}\right] > 1 + \frac{1}{n}.
	\end{align*}
	Therefore we conclude the proof. 
\end{proof}

\begin{proof}[Proof of Lemma~\ref{lemma:graphexpectationboundoraclecomplowerbound}]For any $G\in \cG^*$, we have
\begin{align*}
    \EE_{G'\sim U(\cG)} |V (G)\cap V(G')|^2 = \sum_{j=0}^s (s-j)^2 |\{G'\in \cG: |V(G)\cap V(G')| = s-j\} |.
\end{align*}
%    We have $\sum_{j=0}^{s} m_j= |\cG^*| \geq k$. Therefore there exists integer $l(k) \geq 1$ such that
%	\begin{align}\label{eq:ldef}
%	\textstyle\sum_{j=0}^{l(k)+1} m_{j} > |\cG^+(q)| \geq \sum_{j=0}^{l(k)} m_j.
%	\end{align} 
	We define
	\begin{align*}
	\overline{m} = k - \sum_{j=0}^{l(k)} m_j.
	\end{align*}
	Note that $h(j) := (s-j)^2$ is a decreasing function of $j$, and $\sum_{G'\in \cG}   |V (G)\cap V( G')|^2$ is a sum of $m_1+\ldots + m_{l(k)} + \overline{m}=k$ terms, with at most $m_j$ terms being $h(j)$. 
	Therefore by \eqref{eq:countCq1}, we have
	\begin{align*}
	\sup_{\cG \in \mG(k)} \EE_{G'\sim U(\cG)} |V (G)\cap V(G')|^2  &\leq  \frac{ \sum_{j=0}^{l(k)} h(j)\cdot m_j  + h[l(k)+1] \cdot \overline{m}}{\sum_{j=0}^{l(k)} m_j + \overline{m}}
	\\
	&\leq \frac{ \sum_{j=0}^{l(k)} h(j)\cdot  m_j }{\sum_{j=0}^{l(k)} m_j }.
	\end{align*}
	This finishes the proof.
\end{proof}

{
\bibliographystyle{ims}
\bibliography{sandwich,graphbib}
}

\end{document}